\documentclass[11pt,letterpaper,twoside,reqno,nosumlimits,final]{amsart}

\usepackage[utf8]{inputenc} 
\usepackage[T1]{fontenc}    
\usepackage{url}            
\usepackage{booktabs}       
\usepackage{nicefrac}       
\usepackage{microtype}      
\usepackage{xcolor}         

\usepackage[most]{tcolorbox}

\newcommand{\MSE}{\text{MSE}}
\newcommand{\Lhat}{\widehat{L}}
\newcommand{\Ltilde}{\tilde{L}}
\newcommand{\Ltrue}{L} 

\DeclareMathOperator{\E}{\mathbb{E}}
\newcommand{\ind}[1]{\mathbf{1}_{\{#1\}}}
\newcommand{\Fq}{\mathbb{F}_q}
\newcommand{\F}{\mathbb{F}}
\newcommand{\var}{\text{Var}}
\newcommand{\R}{\mathbb{R}}
\newcommand{\Var}{\mathrm{Var}}
\newcommand{\LLT}{\text{LLT}}

\newcommand{\N}{\mathbb{N}}
\newcommand{\calO}{O}
\newcommand{\phiG}{\phi} 
\newcommand{\Z}{\mathbb{Z}}
\newcommand{\A}{\mathcal A}
\newcommand{\D}{\mathcal D}

\newcommand{\Rqn}[3]{R_q(#1, #2, #3)} 
\newcommand{\bigO}[1]{O\left(#1\right)} 
\newtheorem{questioncounter}{Question}
\newtheorem{question}[questioncounter]{Question}

\definecolor[named]{myLipicsLightGray}{rgb}{0.85,0.85,0.86}

\newenvironment{ShadedBox}{\begin{tcolorbox}[colback=myLipicsLightGray,colframe=myLipicsLightGray,arc=-1pt]\begin{minipage}{0.975\textwidth}}{\end{minipage}\end{tcolorbox}}

\usepackage[usenames,dvipsnames]{xcolor}
\usepackage{fancyhdr}
\usepackage{amsmath,amsfonts,amsbsy,amsgen,amscd,mathrsfs,amssymb,amsthm}
\usepackage{stackengine}
\usepackage{pifont}

\usepackage{mathtools}

\usepackage{graphicx}
\graphicspath{{figures/}}

\usepackage{tikz}
\usetikzlibrary{arrows,chains,matrix,positioning,scopes}
\usepackage{tikz-3dplot}

\usepackage{xcolor}

\definecolor{cit}{rgb}{0.91,0.39,0.16}

\definecolor{dark-gray}{gray}{0.3}

\definecolor{dkgray}{rgb}{.3,.3,.3}
\definecolor{medgray}{rgb}{.5,.5,.5}
\definecolor{ltgray}{rgb}{.7,.7,.7}
\definecolor{dkblue}{rgb}{0,0,.5}
\definecolor{medblue}{rgb}{0,0,.75}
\definecolor{ltblue}{rgb}{0.97,0.97,1}
\definecolor{rust}{rgb}{0.5,0.1,0.1}
\definecolor{ltyellow}{rgb}{1, 1, 0.9}

\usepackage{soul}
\sethlcolor{ltyellow}

\usepackage[normalem]{ulem}

\usepackage[
style=alphabetic, citestyle=alphabetic,
giveninits=true,
sorting=nyt, sortcites=false,
autopunct=true, autolang=hyphen,
abbreviate=false,
backref=false,
backend=bibtex
]{biblatex}
\defbibheading{bibempty}{}

\renewbibmacro{in:}{%
  \ifentrytype{article}{}{\printtext{\bibstring{in}\intitlepunct}}}

\AtEveryBibitem{%
	\clearfield{issn}%
	\clearfield{isbn}%
	\clearlist{location}%
	\ifentrytype{book}%
		{%
		\clearfield{series}%
		\clearfield{volume}%
		\clearfield{pages}}%
		{%
		\clearname{editor}}%
}

\renewbibmacro*{doi+eprint+url}{%
  \newunit\newblock
  \iftoggle{bbx:doi}
    {\printfield{doi}}
    {}%
  \iftoggle{bbx:eprint}
    {\iffieldundef{doi}{\usebibmacro{eprint}}{}}
    {}%
  \newunit\newblock
  \iftoggle{bbx:url}
    {\iffieldundef{doi}{\iffieldundef{eprint}{\usebibmacro{url}}}} %
    {}}

\AtBeginBibliography{\footnotesize}				%

\usepackage{url} %
\makeatletter
\g@addto@macro{\UrlBreaks}{\UrlOrds}
\makeatother

\usepackage{hyperref}

\hypersetup{hidelinks,colorlinks=true,breaklinks=true,urlcolor=cit,citecolor=dkblue,linkcolor=dkblue,bookmarksopen=false,hypertexnames=false}

\usepackage{bookmark}
\bookmarksetup{
open,
numbered,
addtohook={%
\ifnum\bookmarkget{level}=0 %
\bookmarksetup{bold}%
\fi
\ifnum\bookmarkget{level}=-1 %
\bookmarksetup{color=cit,bold}%
\fi
}
}

\usepackage[full]{textcomp}

\usepackage[scaled=.98,sups,osf]{XCharter}%
\usepackage[scaled=1.04,varqu,varl]{inconsolata}%
\usepackage[type1,scaled=0.98]{cabin}%
\usepackage[utopia,vvarbb,scaled=1.07]{newtxmath}
\usepackage[cal=euler]{mathalfa}
\usepackage{bm} %

\usepackage{microtype} %
\usepackage[utf8]{inputenc} %
\usepackage[T1]{fontenc} %
\usepackage[english]{babel} %
\usepackage{setspace}

\usepackage{enumitem}

\setlist{noitemsep} %

\setlist[enumerate]{font=\sffamily\bfseries\footnotesize\textcolor{dkgray},label=\arabic*.}
\setlist[itemize]{font=\textcolor{dkgray},label=\small\textcolor{dkgray}\textbullet}

\usepackage[capitalize]{cleveref}

\numberwithin{equation}{section}

\newtheorem{theorem}{Theorem}[section]
\newtheorem{lemma}[theorem]{Lemma}

\newtheorem{proposition}[theorem]{Proposition}

\newtheorem{corollary}[theorem]{Corollary}

\theoremstyle{definition}

\newtheorem{definition}[theorem]{Definition}

\newtheorem*{ideaT}{Theme}
\newtheorem*{problemT}{Problem}

\Crefname{proposition}{Proposition}{Propositions}
\Crefname{figure}{Figure}{Figures}

\RequirePackage[framemethod=default]{mdframed}

\newmdenv[skipabove=6pt,
skipbelow=6pt,
rightline=false,
leftline=true,
topline=false,
bottomline=false,
backgroundcolor=ltyellow,
linecolor=cit,
innerleftmargin=10pt,
innerrightmargin=10pt,
innertopmargin=0pt,
innerbottommargin=5pt,
leftmargin=0cm,
rightmargin=0cm,
linewidth=4pt]{iBox}

\newmdenv[skipabove=0pt,
skipbelow=0pt,
backgroundcolor=ltblue,
linecolor=dkblue,
linewidth=2pt,
rightline=false,
leftline=false,
topline=false,
bottomline=false,
innerleftmargin=7pt,
innerrightmargin=10pt,
innertopmargin=6pt,
innerbottommargin=6pt,
leftmargin=0cm,
rightmargin=0cm,
innerbottommargin=5pt]{aBox}

\newmdenv[skipabove=10pt,
skipbelow=10pt,
backgroundcolor=white,
linecolor=dkblue,
linewidth=0.5pt,
rightline=true,
leftline=true,
topline=true,
bottomline=true,
innerleftmargin=10pt,
innerrightmargin=0.5in,
innertopmargin=5pt,
innerbottommargin=5pt,
leftmargin=0cm,
rightmargin=0cm]{lfBox}

\usepackage{float}
\floatstyle{plaintop}

\numberwithin{figure}{section}
\numberwithin{table}{section}

\newfloat{recipe}{thp}{lor}
\floatname{recipe}{Recipe}
\numberwithin{recipe}{section}

\usepackage{caption}
\usepackage{subcaption}

\usepackage{booktabs,longtable,tabu} %
\usepackage{multirow} %

\renewcommand{\phi}{\varphi}

\newcommand{\Cov}{\operatorname{Cov}}

\DeclareFontFamily{U}{matha}{\hyphenchar\font45}
\DeclareFontShape{U}{matha}{m}{n}{
  <-6> matha5 <6-7> matha6 <7-8> matha7
  <8-9> matha8 <9-10> matha9
  <10-12> matha10 <12-> matha12
  }{}

\DeclareSymbolFont{matha}{U}{matha}{m}{n}
\DeclareMathSymbol{\abscont}{3}{matha}{"CE}

\makeatletter
\def\paragraph{\@startsection{paragraph}{4}%
  \z@\z@{-\fontdimen2\font}%
  {\normalfont\scshape}}
\makeatother

\addglobalbib{arxiv.bib} %

\allowdisplaybreaks

\usepackage[foot]{amsaddr}

\title[The Structure of Cross-Validation Error]{The Structure of Cross-Validation Error: \\ Stability, Covariance, and Minimax Limits}

\makeatletter
\newcommand\thankssymb[1]{\textsuperscript{\@fnsymbol{#1}}}
\makeatother

\author{Ido Nachum\thankssymb{1}\thankssymb{2}}\email{\href{inachum@univ.haifa.ac.il}{inachum@univ.haifa.ac.il}}

\author{Rüdiger Urbanke\thankssymb{1}\thankssymb{3}}\email{\href{rudiger.urbanke@epfl.ch}{rudiger.urbanke@epfl.ch}}

\author{Thomas Weinberger\thankssymb{1}\thankssymb{3}}\email{\href{thomas.weinberger@epfl.ch}{thomas.weinberger@epfl.ch}}

\thanks{\thankssymb{1} Alphabetical order.}
\thanks{\thankssymb{2} Department of Statistics, University of Haifa, Haifa, Israel.}
\thanks{\thankssymb{3} School of Computer and Communication Sciences, EPFL, Lausanne, Switzerland.}


\begin{document}


\begin{abstract}

Despite ongoing theoretical research on cross-validation (CV), many theoretical questions about CV remain widely open. This motivates our investigation into how properties of algorithm-distribution pairs can affect the choice for the number of folds in $k$-fold cross-validation.

Our results consist of a novel decomposition of the mean-squared error of cross-validation for risk estimation, which explicitly captures the correlations of error estimates across overlapping folds and includes a novel algorithmic stability notion, squared loss stability, that is considerably weaker than the typically required hypothesis stability in other comparable works.

Furthermore, we prove:

1. For any learning algorithm that minimizes empirical risk, the mean-squared error of the \(k\)-fold cross-validation estimator \(\widehat{L}_{\mathrm{CV}}^{(k)}\) of the population risk \(L_{\mathcal D}\) satisfies the following minimax lower bound:

\[
\min_{k \mid n}\; \max_{\D}\; 
\mathbb{E}\!\left[\big(\widehat{L}_{\mathrm{CV}}^{(k)} - L_{\D}\big)^{2}\right]
\;=\;
\Omega\!\big(\sqrt{k^*}/n\big),
\]
where $n$ is the sample size, $k$ the number of folds,  and \(k^*\) denotes the number of folds attaining the minimax optimum. 
This shows that even under idealized conditions, for large values of $k$, CV cannot attain the optimum of order $1/n$ achievable by a validation set of size $n$, reflecting an inherent penalty caused by dependence between folds. 

2. Complementing this, we exhibit learning rules for which
\[
 \max_{\D}\;
\mathbb{E}\!\left[\big(\widehat{L}_{\mathrm{CV}}^{(k)} - L_{\D}\big)^{2}\right]
\;=\;
\Omega(k/n),
\]
matching (up to constants) the accuracy of a hold-out estimator of a single fold of size $n/k$. 

Together these results delineate the fundamental trade-off in resampling-based risk estimation: CV cannot fully exploit all $n$ samples for unbiased risk evaluation, and its minimax performance is pinned between the $k/n$ and $\sqrt{k}/n$ regimes.

\end{abstract}

\maketitle

\section{Introduction}

$k$-fold cross-validation (CV) is a popular model validation technique used in many settings in statistics, data science and machine learning, see \cite{arlot2010survey} for a comprehensive introduction. Given the errors of the models obtained by training on subsets of the full sample and then validating on the remaining samples, the goal is usually of the following two: (Risk estimation) given a model chosen independently of the error computations, estimate its risk by averaging the errors; (Model selection) given the  error estimates computed with CV, select the candidate model that looks best. In this work we mathematically analyze the accuracy of risk estimation under various statistical and algorithmic aspects by harnessing a novel error decomposition. 

Although cross-validation is a workhorse for statistical validation in the empirical sciences, its theoretical properties remain surprisingly poorly understood. For instance, there is still no principled way to choose $k$, the number of folds. As noted by \cite[Sec. 10.3]{arlot2010survey}:

\begin{center}
\begin{quote}
\textit{``VFCV [V-fold cross-validation] is certainly the most popular CV procedure, in particular because of its mild computational cost. Nevertheless, the question of choosing V remains widely open, even if indications can be given towards an appropriate choice.''}
\end{quote}
\end{center}

While well-known theoretical investigations such as \cite{rogers,devroye,blum,kearns,bousquet,bengioUnbiased} are insightful and contributed significant progress towards our understanding of CV, they tend to suffer from at least one of the following limitations\footnote{We will address these limitations in more detail in our related works section.}: (1) sufficient conditions for CV performance that can be arbitrarily loose 
(2) guarantees not in absolute terms but merely in relation to other error measures (e.g. empirical error or hold-out) (3) guarantees that are restricted to leave-one-out cross validation (4) results about certain statistical aspects of CV with no clear connection to the performance of CV. 

The widespread confusion surrounding theoretical aspects of CV within the broader scientific community is perhaps best exemplified by recent informal online discussions \cite{S13,S17}. These discussions reveal the presence of numerous conflicting interpretations concerning the role of specific variance and covariance terms that emerge in the context of CV. Our results provide clear, theoretically grounded insights into these quantities, offering direction for selecting the number of folds based on properties of the algorithm and distribution at hand.

\subsection{Setup and Notation}

We start by establishing the framework for our investigation. Let \(\mathcal{X}\) be the input space and \(\mathcal{Y}\) the output space, and set \(\mathcal{Z}=\mathcal{X}\times\mathcal{Y}\).
We study (possibly randomized) learning rules \(\mathcal{A}:\mathcal{Z}^\star\to\mathcal{H}\) that map a sample
\(S^n=(Z_1,\dots,Z_n)\in\mathcal{Z}^n\) to a hypothesis \(h=\mathcal{A}(S^n)\in\mathcal{H}\subseteq \mathcal{Y}^{\mathcal{X}}\).
The observations are i.i.d.: \(Z_i\sim\mathcal{D}\), hence \(S^n\sim\mathcal{D}^{ n}\).
As common in previous works, we assume throughout that \(\mathcal{A}\) is permutation-invariant (symmetric): for any permutation
\(\pi\) of \(\{1,\dots,n\}\),
\[
\mathcal{A}(S^n)=\mathcal{A}(S^n_\pi)\quad\text{a.s.},\qquad
S^n_\pi:=(Z_{\pi(1)},\dots,Z_{\pi(n)}).
\]

Fix \(k\in\mathbb{N}\) with \(k\mid n\).
Partition the index set \(\{1,\dots,n\}\) into \(k\) disjoint blocks \(I_1,\dots,I_k\) of size \(m:=|I_i|=n/k\),
and define \(S_i=\{Z_j:j\in I_i\}\) and \(S_{-i}=S^n\setminus S_i\).
Given a loss function \(\ell:\mathcal{Y}\times\mathcal{Y}\to\mathbb{R}\), the \(k\)-fold cross-validation estimator is 
\[
\widehat L_{\mathrm{CV}}^{(k)}(\mathcal A,S^n)=\frac{1}{k}\sum_{i=1}^{k}\widehat L_i^{(k)},
\quad
\widehat L_i^{(k)}=\frac{1}{|S_i|}\sum_{(x,y)\in S_i}\ell\big(\mathcal{A}(S_{-i})(x),\,y\big)
=\frac{k}{n}\sum_{(x,y)\in S_i}\ell\big(\mathcal{A}(S_{-i})(x),\,y\big).
\]
That is, \(\widehat L_i^{(k)}\) is the average loss on the \(i\)th hold-out fold, and
\(\widehat L_{\mathrm{CV}}^{(k)}\) averages these across folds.\footnote{When $k=n$, the definition coincides with leave-one-out cross validation.} We omit the subscript $\mathrm{CV}$  whenever it is clear from context.

We assess the performance of cross–validation via the mean squared error (MSE)
\[
\mathrm{MSE}_{\mathrm{CV}}^{(k)}(\mathcal{A}, \mathcal{D})
:= \mathbb{E}_{\,S^n\sim \mathcal{D}^{ n},\,\mathcal A}
\!\left[\big(   \widehat L_{\mathrm{CV}}^{(k)}(\mathcal A,S^n)   -L(\mathcal A(S^n) )\big)^2\right],
\]
where the population risk is
\[
L(h):=\mathbb{E}_{(x,y)\sim\mathcal D}\big[\ell(h(x),y)\big].
\]
For the \(i\)th fold, we also write
\[
L_i^{(k)}(S^n):=   L(\mathcal A(S_{-i})),
\]
i.e., the risk of the hypothesis trained on the complement \(S_{-i}\) of the \(i\)th hold–out block.

Finally, we denote the risks averaged over a sample of arbitrary size $m$ (and the algorithm’s internal randomness) by
\[
\bar L_m \;:=\; \mathbb{E}_{\,S^m\sim\mathcal D^{ m},\,\mathcal A}\!\big[L(\mathcal A(S^m))\big],
\]
when no ambiguity arises, we suppress explicit dependence on \(\mathcal A\), \(\mathcal D\), and the sample.

\subsection{Background and Motivation}

The main motivation for using $k$-fold cross-validation for risk estimation is that partitioning the data into non-overlapping subsets typically reduces statistical variability compared to relying on a single hold-out set \cite{blum}. Moreover, compared to the empirical error, CV generally avoids overly optimistic error estimates caused by overfitting, which is a phenomenon prevalent when deploying overparametrized models such as neural networks.

While there are many schemes for partitioning the folds in CV, based for example on combinatorial partitioning \cite{combinatCV} or Monte-Carlo resampling \cite{montecarlo}, we focus on the variant where one partitions the sample into a non-overlapping partition of equi-sized folds. This includes for example the widely employed variants of $5$- and $10$-fold CV. In practice, guidelines for choosing the number of folds $k$ are usually of heuristic nature. Typical lines of reasoning emphasize the importance of the following terms.
\begin{enumerate}

\item \textit{Per-fold variance:} since each fold computes its own empirical estimate across $m\coloneqq n/k$ i.i.d. test points independent of its training set, the variance per fold decreases as $1/m$.
\item \textit{Inter-fold covariance:} a large number of small folds should intuitively lead to a higher correlation between the individual fold estimates. This is because for each fixed sample $S^n$, decreasing $k$ means that the per-fold output hypotheses share a smaller fraction of the training set which should typically de-correlate the error estimates.
\item \textit{Stability of the algorithm:} if $\mathcal{A}(S^n)$ typically behaves vastly differently than $\mathcal{A}(S^{n-m})$, the per-fold estimates can admit large (typically positive) biases w.r.t. $L$, causing a large bias of the aggregated estimate $\hat L_k$.
\end{enumerate}
In other words, it is argued that choosing $k$ is a balancing act that consists of regulating the overall variance $\text {Var}\big(\widehat L_{\mathrm{CV}}^{(k)} \big) = (1/k^2)\big[\sum_i \text{Var}\big(\hat L_i^{(k)}\big) + \sum_{i\neq j}\Cov\big(\hat L_i^{(k)},\hat L_j^{(k)}\big)\big]$ (items 1 and 2) all while simultaneously not sacrificing too much stability (item 3).

Although the above reasoning is intuitively appealing, to the best of our knowledge there is no corresponding  rigorous treatment in the literature. This motivates the following question.

\begin{ShadedBox}
    \begin{question}\label{q1}
Which formal notion of algorithmic stability best captures the performance of cross-validation, and how does it quantitatively influence $\mathrm{MSE}_{\mathrm{CV}}^{(k)}$ relative to the overall variance across folds?
 \end{question}
 \end{ShadedBox}

Moreover, as a minimal requirement, one may ask that CV perform at least as well as the empirical training error. This question, formalized in terms of so-called \emph{sanity-check bounds}, was studied by \cite{kearns}, who proved that for loss-stable empirical risk minimizers over VC classes, leave-one-out CV performs essentially no worse than the empirical error. Similar results for general $k$-fold CV were obtained by \cite{anthony}.

Another natural sanity-check is to require that CV perform no worse than a hold-out estimate over a  single fold of size $n/k$. The work of \cite{blum} confirms this property for a specific (non-standard) cross-validation setting, where the algorithm’s final output on the full sample $S^n$ is defined as the average of the hypotheses trained on each fold.

Importantly, these works do not \emph{quantify} the advantage of CV.
To this end, we define the \emph{minimax cross-validation risk} for a given algorithm \(\mathcal{A}\) as
\[
\mathfrak{R}_{\mathrm{CV}}(\mathcal{A})
:= \min_{k\mid n}\max_{\mathcal{D}} 
   \mathrm{MSE}_{\mathrm{CV}}^{(k)}(\mathcal{A},\mathcal{D}),
\]
which represents the optimal achievable MSE over all choices of $k$ in the absence of knowledge about the underlying distribution $\mathcal{D}$.

For illustration of the minimax CV risk, consider binary classification with the \(0\!-\!1\) loss function and a constant algorithm \(\mathcal{A}_h\) that always outputs the same hypothesis \(h\), independent of the input sample. If the population risk of \(h\) is \(p\), then $\mathrm{MSE}_{\mathrm{CV}}^{(k)} = \frac{p(1-p)}{n}.$\footnote{This follows because \(L(\mathcal{A}_h,S^n)=p\) and \(\widehat L_{\mathrm{CV}}^{(k)} \sim \mathrm{Bin}(n,p)/n\), hence \(\mathrm{MSE}_{\mathrm{CV}}^{(k)} = \mathrm{Var}(\mathrm{Bin}(n,p))/n^2.\)}
Therefore, \(\mathfrak{R}_{\mathrm{CV}}(\mathcal{A}_h) = 1/(4n)\).

This simple example provides a natural baseline or reference point against which to compare more sophisticated algorithms. Since overly simplistic procedures such as the constant algorithm are not practically useful—yet may achieve similar minimax rates—we focus instead on empirical risk minimization (ERM) algorithms, a standard assumption in statistical learning theory. This leads to the following question:

\begin{ShadedBox}
  \begin{question}\label{q2}
    Can any ERM achieve an \(O(1/n)\) minimax rate, and if not, how close can it get?
  \end{question}
\end{ShadedBox}

We may also inquire about the opposite extreme: how far can an ERM deviate from the optimal minimax rate? 
Although previous work has shown that cross-validation outperforms a single hold-out estimator whenever \(2 < k < n\)~\cite{blum}, the magnitude of this improvement remains unquantified. 
Moreover, in the limiting cases \(k=2\)~\cite{blum} and \(k=n\)~\cite{kearns}, certain algorithms yield cross-validation estimates that coincide exactly with the corresponding hold-out estimates. 
This leads to the following open question:

\begin{ShadedBox}
  \begin{question}\label{q3}
    For intermediate values \(2 < k < n\), do there exist algorithms for which cross-validation performs no better than a hold-out estimator, up to a constant factor?
  \end{question}
\end{ShadedBox}

\subsection{Our Contributions}

We now provide a high-level overview of our answers to Questions~\ref{q1}--\ref{q3}.
\smallskip
\paragraph{\textbf{Question~\ref{q1}.}} 
In Section~\ref{sec:decomp}, we derive a novel decomposition of the MSE of cross-validation in Lemma~\ref{lem:mse_decomp}. 
This decomposition reveals two principal components: 
(i) a new notion of algorithmic stability, which we term \emph{Squared Loss Stability} (SLS), and 
(ii) the covariance between loss estimates across different folds. 

The decomposition also includes two additional correction terms, which under the practical assumption of low loss variance, are dominated by the squared loss stability and fold-covariance terms. 

Overall, this result clarifies which formal quantities theoreticians and practitioners must consider when analyzing or applying cross-validation. Importantly, it shows that there is no universally optimal choice of \(k\). We demonstrate this by analyzing two extreme cases that emphasize opposite regimes:
On the one hand, a \emph{linear function learner} exhibits poor squared loss stability, making stability the dominant term; in this case, using as many folds as possible (e.g., leave-one-out) is advantageous. On the other hand, the \emph{majority algorithm} exhibits high squared-loss stability, making fold covariance the dominant term; thus, using fewer folds is preferable.
\smallskip
\paragraph{\textbf{Question~\ref{q2}.}} 
In Section~\ref{sec:minimax_maj}, we show that ERM algorithms cannot achieve a minimax rate of \(O(1/n)\):
For any ERM algorithm and for every $k$ it holds

$$ \max_{\mathcal{D}} 
   \mathrm{MSE}_{\mathrm{CV}}^{(k)}(\mathcal{A},\mathcal{D}) = \Omega( \sqrt{k}/n ) $$

Therefore, for any ERM algorithm \(\mathcal{A}\) that achieves the minimax optimum with  \(k^*\) folds, the MSE of cross-validation scales as
\[
\mathfrak{R}_{\mathrm{CV}}(\mathcal{A}) = \Omega\!\left(\sqrt{k^*}/n \right).
\]
Hence, even the most carefully designed ERM cannot fully exploit the entire dataset as if it were a single hold-out set of size \(n\); there remains a factor of \(\sqrt{k^*}\) in the rate. 
\smallskip

\paragraph{\textbf{Question~\ref{q3}.}} 
Although \cite{blum} established that cross-validation outperforms a single hold-out estimator for all \(2<k<n\) (without quantifying the gap), 
we show in Section~\ref{sec:minimax_sq} that there are learning algorithms for which cross-validation achieves the same asymptotic rate as a single hold-out set. 
That is, for every \(k\), there exists an algorithm \(\mathcal{A}_k\) such that
\[
   \max_{\mathcal{D}} 
   \mathrm{MSE}_{\mathrm{CV}}^{(k)}(\mathcal{A}_k,\mathcal{D})  = \Omega\!\left(\frac{k}{n}\right).
\]
Thus, although cross-validation can outperform hold-out estimation in general, for certain algorithms this advantage is limited to at most a universal constant factor.

\smallskip

\paragraph{\textbf{On the significance of our study of the majority algorithm.}}
While an algorithm that outputs a constant hypothesis is uninteresting from a theoretical standpoint (the MSE simply scales in accordance to a simple concentration of measure argument), the majority algorithm  features a very rich behavior (as evidenced by the non-trivial proof found in Appendix~\ref{app:majority}). This is despite majority being arguably the next simplest algorithm one could conceive: it can merely output two different hypotheses; and its decision rule is solely based on counting the occurrence of labels, while entirely disregarding the input features.
Majority serves as a critical test case where our MSE bound demonstrably supersedes those of the foundational works \cite{kearns, blum,kale2011cross,kumar2013near}, which highlights the importance of keeping the fold-covariance term intact (or carefully bounding it) when analyzing the MSE of CV. 
\begin{ShadedBox}
We identify Majority as a natural benchmark and advocate that demonstrating tightness for this instance should be a minimal requirement for any future bounds on the error of CV.
\end{ShadedBox}
To illustrate non-tightness of previous analyses, we can instantiate \cite[Theorem 1]{kumar2013near} for the majority algorithm. Then, the contribution of the fold variance term is dominated by that of the loss stability parameter which is of order $1/\sqrt n$, yielding a variance upper bound of order $1/\sqrt n$ for any choice of $k$, the number of folds. By contrast, our analysis (Theorem~\ref{thm:majority}) shows that choosing three folds achieves an MSE of order $1/n$---the theoretical optimum.

\section{Preliminaries on Algorithmic Stability}
Before positioning our work within the context of previous works, it is instructive to familiarize oneself with commonly used notions of algorithmic stability. While there are many notions of algorithmic stability in the literature, we will focus on the two perhaps most widely used variants. We also note that most classical works on the performance of leave-one-out CV consider the following notions for the special case where $m=1$, while some newer works also consider leave-$m$ notions with $m>1$ \cite{gastpar2024algorithms}.
\begin{definition}[Hypothesis Stability] 
We call a pair $(\mathcal A,\mathcal D)$ \textit{hypothesis stable} with parameters $(\beta_1,m)$ if 
$$
\mathbb E_{S^{n-m}\sim \mathcal D^{n-m},S^{m} \sim \mathcal D^{m},(x,y) \sim \mathcal D,\mathcal A} [\ind{\mathcal A(S^{n-m}\cup S^{m})(x) \neq\mathcal A(S^{n-m})(x)}] < \beta_1.
$$
\end{definition}
Intuitively, hypothesis stability is a stronger assumption than necessary. It provides a quantitative measure of how similar the hypotheses trained on different folds are to the one obtained from the full dataset. In this sense, a hypothesis-stable algorithm behaves almost like a constant algorithm—whose outputs, and hence fold predictions, are identical by definition. However, the key factor governing the accuracy of cross-validation (CV) error estimation is not the similarity of hypotheses themselves, but rather the stability of their loss values when a small subset of training samples is removed.

For this reason, it is more natural to require a weaker property, called loss stability (or error stability). This condition ensures that the per-fold loss estimates remain nearly unbiased, even when the training data are slightly perturbed.
\begin{definition}[Loss Stability]\label{def:lossstab}
We call a pair $(\mathcal A,\mathcal D)$ \textit{loss stable} with parameters $(\beta_2,k)$ if 
$$
\mathbb E_{S^{n-k}\sim \mathcal D^{n-k},S^{k}\sim \mathcal D^{k},\mathcal A} [|L(\mathcal A(S^{n-k}\cup S^{k}))-L(\mathcal A(S^{n-k}))|] < \beta_2.
$$
\end{definition}
One might think that loss stability on the other hand is necessary since each single validation in isolation is an unbiased estimator of the loss over $n-m$ samples, meaning that an aggregation of such estimates can only accurately predict $L(\mathcal A(S^n))$ if it is generally not too far from $L(\mathcal A(S^{n-m}))$. This analogy is however not entirely rigorous since the individual estimates are correlated, and as we will see in Lemma~\ref{lem:anticorr}, there exist pathological algorithms that do not admit low loss stability but whose loss can be accurately estimated with CV. 
This directly contradicts \cite[Thm. 5.3]{kearns}, but this is because their result is erroneous (see Appendix~\ref{app:err} for clarification).
In general, it is unclear in which cases loss stability is necessary for low MSE.  Our form of loss stability below is provably necessary under low loss-variance, as we will show in Corollary~\ref{cor:MSElower}.


\begin{definition}[Squared Loss Stability]
We call a pair $(\mathcal A,\mathcal D)$ \textit{squared loss stable} (SLS) with parameters $(\beta,k)$ if 
$$
 \mathbb E_{S^{n}\sim \mathcal D^{n}}[(\tilde L^{(k)}(\mathcal A , S^n)-L(\mathcal A(S^n)))^2] < \beta
$$
where $\tilde L^{(k)}(\mathcal A , S^n):= \frac{1}{k}\sum_{i=1}^k L^{(k)}_i(\mathcal A , S^n)$ is the symmetrized leave-$m$ loss.
\end{definition}
Once can easily show that both $\beta_2^2$ and $\beta$ are upper bounded by the same quantity $\mathbb{E}[(L_{n-k} - L_n)^2]$. For a further discussion of the role of squared loss stability, see the section after Lemma~\ref{lem:mse_decomp}.

The following Lemma states that control over the first two moments of the risk allows us to bound the squared-loss stability
\begin{lemma}[Bounds on the Squared Loss Stability]\label{lem:boundsSLS}
Assume that the loss functional is bounded between $0$ and $1$ and that the risk has means $\mathbb E[L]=\bar L$ and $\mathbb E[L^{(k)}]=\bar L^{(k)}$ and denote the variances loss as $\sigma_n^2:=\text{Var}(L)$ and $\sigma_{n-m}^2:=\text{Var}(L_1^{(k)})$.
Then, the squared loss stability $\E[(\tilde L - L)^2]$ can be bounded as
$$ (\bar L^{(k)} - \bar L)^2 \le \E[(\tilde L^{(k)} - L)^2] \le \left(\sigma_{n-m}+ \sigma\right)^2 + (\bar L^{(k)} - \bar L)^2 $$
\end{lemma}
This result will become useful later for controlling the stability of linear functions.

\section{Prior Work}
The works \cite{rogers} and \cite{devroye,devroye2} have been among the first to establish rigorous stability-based performance guarantees for classification problems using leave-one-out CV. Though in their works, they assume that the considered algorithms be 'local' (e.g. nearest neighbors) and the data distribution be arbitrary, their results directly generalize to the class of hypothesis stable algorithms (in which case the bounds are no longer distribution-free).

The well-known work by \cite{bousquet} provided a streamlined presentation of classical results and novel error bounds for leave-one-out CV and the empirical error under various strengthened assumptions on algorithmic stability and/or the loss functional.

Estimating the population loss in an algorithm-dependent manner is closely related to statistical learning theory. The principal aim of this field is the development of generalization bounds, typically in the form of high-probability upper bounds $L(\mathcal A(S))<\hat L_{\text{emp}}(\mathcal A(S),S)+C$, where $\hat L_{\text{emp}}(\mathcal A(S),S)$ denotes the empirical error over the training set and the generalization measure $C$ accounts for the over-optimism of the empirical error induced by the complexity of the model. 

A classical result \cite{vapnik,blumer} states it is sufficient and necessary to let $C=\Theta (\sqrt {d/n})$ to ensure that the generalization bound holds in a tight manner even for the worst-case distribution, where $d$ is the VC dimension, a combinatorial measure of the richness of the hypothesis class $\mathcal H$ associated with the algorithm. 

These bounds are often too pessimistic because they are not sensitive to the (possibly benign) characteristics of the specific data distribution at hand. Moreover, it can be shown that in overparametrized settings (which are ubiquitous in machine learning), generalization measures that are not distribution-dependent face limitations both empirically \cite{jiangfantastic, dziugaite2020search} and theoretically \cite{nachum2024fantastic,gastpar2024algorithms}. 

With this in mind, CV becomes conceptually interesting as a flexible alternative to generalization measures for overparametrized settings, where the empirical error is typically uninformative, and a distribution-dependent measure is required---though admittedly CV is no silver bullet (theoretical bounds require estimating the algorithms stability, and CV can be computationally expensive). 

In the light of this comparison, a sound minimal requirement is that CV performs at least as well as the empirical error. This question, formalized in terms of so-called \textit{sanity-check bounds}, has been studied by \cite{kearns}. One of their central results is that for loss-stable empirical risk minimizers over VC classes, leave-one-out CV is guaranteed to perform essentially no worse than the empirical error.
\cite{anthony} derived similar results for the more general case of $k$-fold CV.

Yet another valid \textit{sanity-check} might be to require CV to do no worse than a single hold out set of corresponding size. The work of \cite{blum} shows that this does indeed hold for a specific (non-standard) cross-validation setting.

Another influential line of work \cite{bengioUnbiased,bengio2} considers the limitations of unbiased estimation of the variance of CV.

A more recent line of work is \cite{kale2011cross, kumar2013near}. Therein, the authors devise upper bounds on the MSE based novel notions of loss stability.
Unfortunately, the main Theorem in \cite{kale2011cross} is erroneous (see Appendix~\ref{app:errKale}), which makes it difficult to assess the implications of these results.
The follow-up work introduces a version of loss stability that leads to a stronger result \cite[Theorem 1]{kumar2013near} since the related stability parameter is a lower bound on the one appearing in \cite[Theorem 2]{kale2011cross}.
In both works, the authors aim to bound the performance of the non-standard algorithm that at test time picks one of the  cross-validated hypotheses uniformly at random, while we directly bound the MSE of the full-sample hypothesis.
Lastly, another key difference is that our Theorem~\ref{thm:mse_characterization} presents a characterization (i.e. two-sided bound) of the MSE of CV, not just an upper bound, and the gap between our lower and upper bound can approach zero (under low loss variance).


\section{Results}

\subsection{MSE Decomposition and the Role of Squared Loss Stability}\label{sec:decomp}

We show that our new notion of SLS is one of the two principal components governing the MSE of CV which answers Question~\ref{q1}. The following lemma formalizes this relationship.

\begin{lemma}[Decomposition of the MSE]
		\label{lem:mse_decomp}
		Denote the expected conditional variance of the risk as $\bar\sigma^2 := \E\big[\Var_{(x,y)\sim \D}\big(\ell(\A(S_{-1}^n)(x),y\big) | S_{-1}^n)\big]$. The MSE of $k$-fold cross-validation for a symmetric algorithm admits the exact decomposition:
		\begin{align*}
			\MSE_{CV}^{(k)} = &\underbrace{\E[(\Ltilde^{(k)}-\Ltrue)^{2}]}_{\text{Squared Loss Stability}} + \frac{k-1}{k}\cdot\underbrace{\Cov(\Lhat_{1}^{(k)}, \Lhat_{2}^{(k)})}_{\text{Inter-fold Covariance}}  \qquad~~~~~+\frac{1}{k}\cdot\mkern-76mu\underbrace{\frac{\bar\sigma^2}{m}}_{\text{Single-Fold Estimation Variance}} \\
			&+ \underbrace{2 \Cov(\Ltrue, L_1^{(k)}-\Lhat_1^{(k)}) - \frac{k-1}{k}\Cov(L_1^{(k)}, L_2^{(k)})}_{\text{Correction Terms}}
		\end{align*}
	\end{lemma}
	
	\begin{proof}
		See Appendix~\ref{app:mse_decomp}.
	\end{proof}

The individual terms in the decomposition have the following interpretations.
\begin{itemize}
\item The first term, the \textit{squared loss stability}, is a measure of algorithmic stability that captures how quickly the averaged loss across the folds deteriorates as a function of $n$ and $k$. Unlike other stability notions, SLS explicitly captures the inter-dependence of the different risks of the $k$ cross-validated hypotheses that comes from the overlap of the training sets. 
Since we will show that being stable in the SLS sense is necessary (see Corollary~\ref{cor:MSElower}), this suggests that MSE bounds based on stability notions that do not capture the fold-wise structure of CV must necessarily be loose in some settings.
\item The second term is a constant $\in [1/2,1)$ times the inter-fold covariances $\Cov(\hat L^{(k)}_1,\hat L^{(k)}_2)$. A large fold-covariance means that estimates from different cross-validated hypotheses tend to be small or large simultaneously, which degrades the $1/k$ variance reduction one would otherwise obtain from averaging $k$ independent estimates.
\item The per-fold variance term. This term (together with its $1/k$ pre-factor) typically contributes an irreducible $\Theta(1/n)$ error floor in case of a bounded loss function, see the Lemma below.
\item The correction terms do not immediately offer a straightforward interpretation, but as we will see, it can easily be upper bounded under an additional assumption.
\end{itemize}

\begin{lemma}[Expected Risk Variance for Bounded Loss]\label{lem:boundedloss}
	For a loss functional $\ell$ bounded in $[0, M]$, the expected risk variance term $\bar\sigma^2 := \E\big[\Var_{(x,y)\sim \D}\big(\ell(\A(S_{-1}^n)(x),y\big) | S_{-1}^n)\big]$ is bounded as 
	\[\bar\sigma^2 \le \E[L_1^{(k)}(M-L_1^{(k)})]\le \frac{M^2}{4}.\]
\end{lemma}
	
	\begin{proof}
		For a random variable $X$ bounded in $[0, M]$ with mean $\mu$, the variance is bounded by the Bhatia-Davis inequality as
		$
		\Var(X) \le (\mu - 0)(M - \mu) = \mu(M-\mu).
		$
		Setting $X=\big[\ell(\mathcal A(S^n_{-1})(x),y)\big]|S^n_{-1}$ such that $\mu=L_1^{(k)}(S^n_{-1})$ and taking the expectation gives $\E[\Var\big(\ell(\A(S_{-1}^n)(x),y\big) | S_{-1}^n)] \le \E[L_1^{(k)}(M-L_1^{(k)})]$. Further, it is easy to see that this quantity is maximized when $L_1^{(k)}(S^n_{-1})=M/2$ almost surely.
	\end{proof}

Our first result is that under the low variance assumption, we can bound the MSE from both sides.
	\begin{theorem}[Characterization of the MSE]
		\label{thm:mse_characterization}
		Denote the expected conditional variance of the risk as $\bar\sigma^2 := \E\big[\Var_{(x,y)\sim \D}\big(\ell(\A(S_{-1}^n)(x),y\big) | S_{-1}^n)\big]$. Assume $(\mathcal A, \mathcal D)$ has risk variance $\sigma^2_{n'}:=Var(\Ltrue(\mathcal A, S^{n'}))$. Then, the MSE of $k$-fold cross-validation (where $m=n/k$ is the fold size) is characterized by:
		\begin{align*}
			\MSE_{CV}^{(k)} = &\E[(\Ltilde^{(k)}-\Ltrue)^{2}] + \frac{k-1}{k}\Cov(\Lhat_{1}^{(k)}, \Lhat_{2}^{(k)}) + \frac{\bar\sigma^2}{n} + \mathcal{C}
		\end{align*}
		where the magnitude of the correction term $\mathcal{C}$ is bounded by:
		$$
		|\mathcal{C}| \le \frac{k-1}{k}\sigma_{n-m}^2 + 2 \sqrt{\frac{\sigma_n^2\bar\sigma^2}{m}}
		$$
	\end{theorem}
	\begin{proof}
See Appendix~\ref{app:charact}.
    \end{proof}

    In order for the above characterization to become meaningful, $\mathcal{C}$ must be negligible compared to the largest of the main term which, necessitates small enough loss variances $\sigma_{n-m}^2, \sigma_{n}^2$.
When comparing different algorithms or hyperparameter settings, we rely on their estimated performance (be it on a validation set, CV, or some other estimate) to make decisions. If these performance estimates have high variance due to the variability introduced by the training sample, it becomes difficult to confidently assert that one algorithm is truly better than another. Low variance of the population loss makes these comparisons more robust and increases our confidence in the selected model's expected performance in real-world scenarios. As part of the standard model selection process, practitioners typically evaluate multiple algorithms, during which those exhibiting high variability in validation error across samples are naturally excluded. For this reason, our low-variance assumption should be considered rather benign. In other words, for large enough samples size and low enough loss  variance, the fold covariance and the loss stability essentially fully characterize the mean-squared error.

In general, loss stability is not necessary for ensuring small MSE of CV. This follows from the existence of pathological algorithm-distribution combinations that are quite squared loss stable but have arbitrarily small MSE.

\begin{lemma}\label{lem:anticorr}
There exists a algorithm-distribution combinations with squared loss stability $1/8$ and MSE $0$.
\end{lemma}
\begin{proof}
Consider the following setup. Let $\mathcal X=[0,1]$, $\mathcal Y=\{0,1\}$, with input distribution $\mathcal D_\mathcal X=\mathcal U(\mathcal X)$, and conditional output distribution $f = \ind{x > 1/2}$ such that $\mathcal D=(\text{Id},f)\circ \mathcal D_\mathcal X$. Consider the algorithm $\mathcal A(S^n)=\ind{1/2-p/2< ~\cdot~< 1-p/2}$ where $p=p(S)=\sum_i y_i/n$, $\mathcal A(S^{n-k})=h_0$ and where $h_0$ is the constant zero hypothesis. Then, $\hat L^m=\sum_i \hat L_i^m/k=\sum_i y_i/n=p(S)$ and $L=p(S)$ so that the MSE is zero. Simultaneously, the squared loss stability is $\mathbb{E}[(L^m - L)^2] = \mathbb{E}_{p \sim \mathrm{Bin}(k, 1/2)}[(1/2 - p)^2] = 1/(4n)$ which can be as large as $1/8$ for $n=2, m=1$.
\end{proof}

Beyond the squared-loss stability, we can also generically bound the inter-fold covariance. Notably, the covariance can never be strongly negative and thereby balance out the influence of low loss stability. This lower bound follows from a geometric argument about the minimal pair-wise inner product of $k$ vectors in euclidean space.
\begin{lemma}\label{lem:cov}
For every algorithm-distribution pair, the covariance between the folds is bounded as follows. 
\[
-1/(2n)\leq-1/[4(n-m)]\leq\text{Cov}(\hat L^k_1,\hat L^k_2)\leq \sigma_{n-m}^2+1/(4m).
\]
\end{lemma}
\begin{proof}
See Appendix~\ref{app:cov}.
\end{proof}

Seeing how the covariance can never be strongly negative, it becomes clear that under the assumption of low loss variances, high squared loss stability is also a necessary condition for small MSE. This thereby rules out examples such as the one in Lemma~\ref{lem:anticorr}.
\begin{corollary}\label{cor:MSElower}
It holds that
$$
\text{MSE}^{(k)}_{CV}\geq \mathbb E[(\tilde L^k-L)^2] -1/(2n) - \sigma_{n-m}^2 - 2 \sqrt{\frac{\sigma_n^2\bar\sigma^2}{m}}
$$
\end{corollary}
\begin{proof}
We simply combine the left-hand sides of Theorem~\ref{thm:mse_characterization} and Lemma~\ref{lem:cov}.
\end{proof}

 To illustrate our decomposition and the role of SLS, we consider two algorithms. In the first, the SLS term dominates the MSE; in the second, the inter-fold covariance is the primary contributor. This demonstrates that there is no universally optimal choice of $k$ in cross-validation: in the first case, performance improves with larger $k$, while in the second, smaller $k$ is preferable.

\subsubsection{Linear Functions}

Let us consider multi-class classification with a randomized algorithm.

To set the stage, let us introduce the class of linear functionals $\mathbf{Lin}_q(d)$ over the vector space $\mathbb F_q^d$ where $\mathbb F_q$ is the finite field with $q$ elements, with $q$ prime. 
\[
    \mathbf{Lin}_q({d}) \equiv (\mathbb F_q^d)^*   := \bigg\{ f_a: \mathbb F_q^d \rightarrow \mathbb F_q :  a \in \mathbb F_q^d~,~f_a(x)= \sum_{i=1}^d a_i \cdot x_i  \mod q   \bigg\}
\]
Note that for example, $\mathbf{Lin}_2(d)$ is the class of all parity functions of dimension $d$. We will consider throughout this section that the distribution is $\mathcal D =\mathcal U ( \mathbb F_q^d )$, the uniform distribution over the space.

An elementary property of this class is that distinct pairs of linear functions agree on exactly a portion $1/q$ of the space. This means that for in-class learning, the risk is polarized between two dissimilar values, making this an interesting case-study for how CV performs under loss instability.

\begin{lemma}\label{lemma:p-parityRisk}
Each two distinct functions $f,h \in \mathbf{Lin}_q (d)$ agree on a fraction $1/q$ of the space and the $0-1$ risk of the function $h$ over samples from $\mathcal D_f= f \diamond \mathcal U ( \mathbb F_q^d )$ is given by
\begin{align*}
 L _ {\mathcal D_f}(h) = \begin{cases}
0 & h = f\\
1 - 1/q & h \neq f 
\end{cases}
\end{align*}
where $f \diamond \mathcal U ( \mathbb F_q^d )$ denotes the distribution of the random variable $(X,f(X))$ where $X \sim  \mathcal U ( \mathbb F_q^d )$.
\end{lemma}
We will study the algorithm $\mathcal A_{lin}: \{\mathbb F_q^d\}^n \rightarrow \mathbf{Lin}_q(d)$ defined as the randomized empirical risk minimizer which outputs one of the sample-consistent linear functions uniformly at random. This algorithm is notably quite hypothesis unstable in the regime $n<d$. In that case, there exist at least $q^{n-d}$ 
sample-consistent linear functions and $\mathcal A_{lin}$ picks one of them uniformly at random. At the same time, $\mathcal A_{lin}$ is quite hypothesis (and hence loss-) stable for $n\ge d$ since here $\mathcal A_{lin}$ will typically select the ground truth assuming that the number of linearly independent samples exceeds the number of linear constraints. One delicate detail that significantly complicates the analysis is the possibility that samples can be linearly dependent. For this reason, in every step of our analysis we need to condition on the set of samples being of a specific rank. This can be handled with random matrix theory results for finite fields \cite{blake2006properties}. In contrast to the majority algorithm, utilizing Theorem~\ref{thm:mse_characterization} now requires controlling the loss variances and the squared loss stability (utilizing Lemma~\ref{lem:boundsSLS}), which further complicates the analysis.

\begin{theorem}[MSE Bounds for CV on Linear Functions]\label{thm:linearMSE}
Let $k$ be the number of folds, $n$ be the total number of samples, and $m$ be the size of each fold. Let $d$ be the feature dimension and $q$ be the finite field size. 

The Mean Squared Error (MSE) of $k$-fold cross-validation for $\mathcal A_{lin}$ is bounded as follows:

\textbf{Case 1: $n < d$}
$$ \text{MSE}^{(k)}_{CV} = O\left(q^{-(d-n)} \right) $$

\textbf{Case 2: $n \ge d$ and $n-m < d$}
$$ \text{MSE}^{(k)}_{CV} = 1 - O(1/q)= \Omega(1) $$

\textbf{Case 3: $n-m \ge d$} 
$$ \text{MSE}^{(k)}_{CV} = O\left(  q^{-(n-m-d+1)} \right) $$
\end{theorem}
\begin{proof}
See Appendix~\ref{app:linearMSE}.
\end{proof}
We see that in this setting, in the case $n<d$ (where $\mathcal A_{lin}$ does typically not output the ground truth $f$), it is beneficial to choose $m$ as large as possible since this can only decrease the MSE. In the case $n\ge d$ on the other hand it is beneficial to set $m=1$ since the bound increases in $m$.

Lastly, we remark that here any MSE bound based on hypothesis stability must be loose since we are highly hypothesis unstable in the cases 1 and 2 above since $n-m \leq d$ implies the existence of multiple sample-consistent linear functions. Yet, our loss-stability based analysis correctly captures the MSE.

\subsubsection{Majority Algorithm}

In contrast to the algorithm $\mathcal{A}_{{lin}}$, which can produce many hypotheses with potentially large variations in their loss values, we now consider the opposite extreme—a setting in which the algorithm can output only two hypotheses with identical loss values.

Let the sample be $S^n = \{z_i\}_{i=1}^n \sim \mathcal{D}^n$ and define $Y := \sum_{i=1}^n y_i$.

The \emph{majority algorithm} is defined as
\[
\mathcal{A}_{{maj}}(S^n) =
\begin{cases}
h_0 : x \mapsto 0, & \text{if } Y \leq n/2, \\[4pt]
h_1 : x \mapsto 1, & \text{if } Y > n/2,
\end{cases}
\]
where $h_i$ denotes the hypothesis that outputs the constant value $i$.

We consider a distribution $\mathcal{D}$ whose marginal over $\mathcal{X}$ is arbitrary, and whose labels $y_i$ are i.i.d.\ draws from $\mathcal{Y} = \{0,1\}$ with $y_i \sim \mathrm{Ber}(1/2)$. 
In this case, $Y \sim \mathrm{Bin}(n, 1/2)$, and the population loss of $\mathcal{A}_{\mathrm{maj}}$ equals $1/2$, independent of both the sample $S$ and the sample size $n$. Consequently, analyzing its mean-squared error (MSE) reduces to controlling the covariance between folds.

\begin{lemma}
The MSE of the majority algorithm equals $\frac{k-1}{k}\text{Cov}(\hat L_1^{(k)},\hat L_2^{(k)}) + \frac{1}{4n}$.
\end{lemma}
\begin{proof}
Since $L=\tilde L^{(k)}=1/2$ and we have zero loss variance, this directly follows from Lemma~\ref{thm:mse_characterization}.
\end{proof}
Let us proceed with an informal analysis. 
First, $\mathcal A_{maj}$ is remarkably stable across many instances of $S$. In specific, whenever we know that $Y$ is bounded away from $n/2$ by at least $m/2$, $\mathcal A_{maj}(S)$ and $S$ are conditionally independent, and Hoeffding's inequality asserts that $\hat L^{(k)}$ concentrates around $1/2$, which is typically close in value to $1/2 \pm \Theta(1/\sqrt n)$, leading to a conditional MSE on the order of $1/n$. By contrast, $\mathcal A_{maj}$ is conditionally highly hypothesis unstable in the regime $Y = n/2\pm \theta( \sqrt{m})$. This follows from $\mathcal A_{maj}(S)$ having constant probability of changing from the all-ones to the all-zeros function (or vice-versa) upon removing a fold of size $m$ so $\mathcal A_{maj}(S)\neq \mathcal A_{maj}(S_{-1})$. 
We fall into this unstable regime with probability proportional to $\sqrt m / \sqrt n$ due to Stirling's approximation of the central probability masses of $Y$. 
It follows that the algorithm becomes more hypothesis unstable as we decrease $k$ (which makes sense because we removing a larger fold of size $m$). Yet, perhaps surprisingly, CV becomes \textit{more accurate} as we decrease $k$ (or increase $m$) as the following theorems suggest. 

\begin{theorem}[Fold-Covariance of Majority: Exact Combinatorial Form]
For $1\le m\le n/2$, $m|n$, we have
	\[
	\Cov(\hat L_1,\hat L_2)\equiv \Cov(n,m)=
	2^{-n}\sum_{j=0}^{m-1}\binom{m-1}{j}^2
	\binom{n-2m}{\lfloor (n-m)/2\rfloor-j}.
	\]
\end{theorem}
\begin{proof}
See Appendix~\ref{app:maj_exact-combinatorial}.
\end{proof}

A more explicit version of the following result, including precise constants, is provided in Appendix~\ref{app:majority}.

\begin{theorem}[Fold-Covariance Asymptotics]\label{thm:majority}
	Throughout, let $n\ge 2$ and $m|n$.
	
	\medskip
	\noindent\textbf{(A)} For all $m=\Omega(n^{1/5})$,
	\[
	\Cov(n,m)
	=\Theta\left(\frac{1}{\sqrt{nm}}\right)=\Theta\left(\frac{\sqrt{k}}{n}\right).
	\]
    
    	\noindent{\bf (B) Monotonicity and minimizer.}
	For all sufficiently large $n$,
	$$
	\Cov(n,1)>\Cov(n,2)>\cdots>\Cov\big(n, n/3\big) ~ \text{and} ~ \Cov\big(n, n/3\big)<\Cov\big(n, n/2\big), $$	
so consequently $k=3$   minimizes fold covariance.
\end{theorem}

We observe that the MSE scales as $\sqrt{k}/n$. In this setting, it is therefore advantageous to choose as few folds as possible. Notably, \emph{hypothesis stability--based bounds are not sufficiently fine-grained} here: they incorrectly predict the MSE to increase when $k$ decreases, since the algorithm becomes less hypothesis stable. A similar lack of tightness, by more than constant factors, arises in existing analyses such as \cite{kearns, blum, kale2011cross, kumar2013near} when applied to the Majority algorithm.

For this reason, we identify \emph{Majority as a natural benchmark}: achieving tightness for this instance should be regarded as a minimal requirement for any future theoretical bounds on the error of cross-validation.

\subsection{A Minimax Lower Bound for Cross-Validation with ERM Algorithms}\label{sec:minimax_maj}

The answer to Question~\ref{q2} follows as a corollary of the preceding analysis of the Majority algorithm. 
To establish this, we consider, for any ERM algorithm, a degenerate distribution supported on a single point $x$, where the labels are drawn uniformly from $\{0,1\}$. 
In this case, an ERM must output a hypothesis whose label for $x$ agrees with the majority label observed in the sample $S$. 
Thus, the behavior of any ERM under this distribution reduces directly to the analysis of the Majority algorithm.

\begin{corollary}\label{cor:minimax}
For any ERM algorithm $\mathcal{A}$, it holds that 
\[
\mathfrak{R}_{\mathrm{CV}}(\mathcal{A}) = \Omega\!\left(\frac{\sqrt{k}}{n}\right),
\]

where $k$ is the number of folds that achieves the minimax optimum.
\end{corollary}

This result shows that, in the distribution-free setting, no ERM algorithm can be designed to utilize all $n$ samples as efficiently as an independent validation set of the same size, whose mean-squared error decreases at the optimal rate of order $1/n$.


\subsection{Algorithms Achieving the Hold-Out Rate}\label{sec:minimax_sq}

We conclude our results by showing that certain algorithms achieve, up to a constant factor, the same rate as a hold-out estimator, regardless of the number of folds. This result provides an affirmative answer to Question~\ref{q3} by constructing an algorithm which attains (up to constant factors) the upper bound in Lemma~\ref{lem:cov}. 

To establish this result, we consider the setting of binary classification under the $0$–$1$ loss and construct a simple family of algorithms that can output only the constant functions $h_0(x)=0$ and $h_1(x)=1$.

\begin{definition}[$r$-Square-Wave Algorithm]
An algorithm $\mathcal{A}$ is called an \emph{$r$-square-wave algorithm} if, for a training sample $S^n = \{(x_i, y_i)\}_{i=1}^n$,
\[
\mathcal{A}(S^n) =
\begin{cases}
h_0, & \text{if } \big\lfloor \tfrac{1}{\sqrt{r}} \sum_{i=1}^n y_i \big\rfloor \bmod 2 = 0, \\[6pt]
h_1, & \text{if } \big\lfloor \tfrac{1}{\sqrt{r}} \sum_{i=1}^n y_i \big\rfloor \bmod 2 = 1
\end{cases}
\]
\end{definition}

\begin{theorem}[Square-Wave Algorithm Fold-Covariance]\label{thm:sq}
	Assume $k\ge 3$ and let $m|n$. Then, for sufficiently large $m$, the fold-covariance of the $m$-square-wave algorithm satisfies
	\[
	\Cov(\hat L_1, \hat L_2) \;=\; \frac{c_0}{m} \;+\; E_L,
	\]
	where $c_0$ is the main constant and $E_L$ is an error term bounded by
	\[
	|E_L| \;\le\; \frac{c_R}{m} \;+\; O\!\big(m^{-3/2}\big).
	\]
	where the above constants are given as
	\[
	c_0=\frac12\sum_{j=0}^\infty e^{-\frac{\pi^2}{4}(2j+1)^2}\approx 0.0424,
	\qquad
	c_R \le 4 \times 10^{-4}.
	\]
	In particular, since $c_0 >c_R$, we have $\Cov(\hat L_1, \hat L_2)=\Theta(1/m)$ positive.
\end{theorem}

In other words, the fold-covariance of the square-wave algorithm is independent of $n$, no matter how large the shared training set (which is of size $n-2m$) is, which is rather remarkable. The square-wave algorithm is carefully designed as to be robust to small changes in the training set, while simultaneously admitting large enough variation in the risk values it can achieve. Generally, these are two diametrically opposed algorithmic properties.

Theorem~\ref{thm:sq} gives the desired distribution-free result in the following corollary.

\begin{corollary}
For every $k$, there exists an algorithm $\A_k$ such that
 
 \[
   \max_{\mathcal{D}} 
   \mathrm{MSE}_{\mathrm{CV}}^{(k)}(\mathcal{A}_k,\mathcal{D})  = \Omega\!\left(\frac{1}{m}\right) = \Omega\!\left(\frac{k}{n}\right).
\]   
\end{corollary}


\section{Conclusion}

We presented a novel decomposition of the MSE for CV that illuminates the respective roles of squared-loss stability and fold covariance. In contrast, existing stability-based analyses are inherently limited, as they are not merely off by a constant factor—a fact underscored by our tight characterization of the Majority algorithm. Consequently, we propose that any future theoretical analysis of CV should benchmark its results against the Majority baseline introduced in Theorem~\ref{thm:majority}.

While low loss variance is a fairly common property in practice, an interesting direction for future research is to identify alternative (and possibly weaker) conditions under which low squared-loss stability becomes necessary for achieving a small MSE.

Finally, as a natural extension of Corollary~\ref{cor:minimax}, it would be compelling to investigate which combined properties of algorithms and data distributions can yield improved minimax rates (or even attain the optimal $1/n$ rate) in settings beyond the distribution-free case.


\section*{References}
\printbibliography[heading=none]

\clearpage
\appendix

\section{Main Part}

\subsection{Proof of Lemma~\ref{lem:mse_decomp}}\label{app:mse_decomp}
\begin{proof}
		The proof begins by adding and subtracting $\Ltilde^k$:
		\begin{align}
			\MSE_{CV}^{(k)} &= \E[(\Lhat^{(k)} - \Ltrue)^2] \notag\\
			&= \E[((\Lhat^{(k)} - \Ltilde^{(k)}) + (\Ltilde^{(k)} - \Ltrue))^2] \notag\\
            &= \E[(\Lhat^{(k)} - \Ltilde^{(k)})^2] + \E[(\Ltilde^{(k)} - \Ltrue)^2] + 2\E[(\Lhat^{(k)} - \Ltilde^{(k)})(\Ltilde^{(k)} - \Ltrue)] \notag\\
		&= \left(\Var(\Lhat^{(k)}) + \Var(\Ltilde^{(k)}) - 2\Cov(\Lhat^{(k)}, \Ltilde^{(k)})\right) + \E[(\Ltilde^{(k)} - \Ltrue)^2] + 2\Cov(\Lhat^{(k)} - \Ltilde^{(k)}, \Ltilde^{(k)} - \Ltrue)\notag\\
        		&= \Var(\Lhat^{(k)}) + \E[(\Ltilde^{(k)} - \Ltrue)^2]+ \Var(\Ltilde^{(k)}) - 2\Var(\Ltilde^{(k)}) \notag\\
		&\quad - 2\Cov(\Lhat^{(k)}, \Ltilde^{(k)}) + 2\Cov(\Lhat^{(k)}, \Ltilde^{(k)}) - 2\Cov(\Lhat^{(k)}, \Ltrue) + 2\Cov(\Ltilde^{(k)}, \Ltrue) \notag\\
			&= \Var(\Lhat^{(k)}) + \E[(\Ltilde^{(k)} - \Ltrue)^2] - \Var(\Ltilde^{(k)}) + 2\Cov(\Ltrue, \Ltilde^{(k)} - \Lhat^{(k)})\label{eq:MSE1}.
		\end{align}
		We now substitute the fold-level decompositions for each term, leveraging symmetry. The key step is the Law of Total Variance for $\Var(\Lhat_1^{(k)})$, via conditioning on $S_{- 1}^n$:
		\begin{align*}
			\Var(\Lhat_1^{(k)}) &= \Var(\E[\Lhat_1^{(k)} | S_{- 1}^n]) + \E[\Var(\Lhat_1^{(k)} | S_{-1}^n)] \\
			&= \Var(L_1^{(k)}) + \E\left[\Var_{S_1\sim \mathcal D^{\otimes m}}\left(\frac{1}{m}\sum_{(x,y) \in S_1} \ell(\mathcal A(S_{- 1}^n)(x),y) \bigg| S_{- 1}^n\right)\right] \\
			&= \Var(L_1^{(k)}) + \frac{1}{m}  \E\left[\Var_{(x,y)\sim \mathcal D}\left(\ell(\mathcal A(S_{- 1}^n)(x),y) | S_{- 1}^n\right)\right] \quad (\text{since validation points are i.i.d.}) \\
			&= \Var(L_1^{(k)}) + \frac{\bar\sigma^2}{m}
		\end{align*}
		The other terms decompose as:
		\begin{itemize}
			\item $\Var(\Lhat^{(k)}) = \frac{1}{k}\Var(\Lhat_{1}^{(k)}) + \frac{k-1}{k}\Cov(\Lhat_{1}^{(k)}, \Lhat_{2}^{(k)})$
			\item $\Var(\Ltilde^{(k)}) = \frac{1}{k}\Var(L_1^{(k)}) + \frac{k-1}{k}\Cov(L_1^{(k)}, L_2^{(k)})$
			\item $\Cov(\Ltrue, \Ltilde^{(k)}-\Lhat^{(k)}) = \Cov(\Ltrue, L_1^{(k)}-\Lhat_1^{(k)})$
		\end{itemize}
		Plugging these into the MSE expression of Eq.~\eqref{eq:MSE1}:
		\begin{align*}
			\MSE_{CV}^{(k)} = &\left( \frac{1}{k}\left(\Var(L_1^{(k)}) + \frac{\bar\sigma^2}{m}\right) + \frac{k-1}{k}\Cov(\Lhat_{1}^{(k)}, \Lhat_{2}^{(k)}) \right)
			+ \E[(\Ltilde^{(k)}-\Ltrue)^{2}] \\
			&- \left( \frac{1}{k}\Var(L_1^{(k)}) + \frac{k-1}{k}\Cov(L_1^{(k)}, L_2^{(k)}) \right) + 2\Cov(\Ltrue, L_1^{(k)}-\Lhat_1^{(k)})
		\end{align*}
		The variance terms cancel, yielding the final form.
	\end{proof}

\subsection{Proof of Theorem~\ref{thm:mse_characterization}}\label{app:charact}
	\begin{proof}
		The proof amounts to bounding the magnitude of the correction term $\mathcal{C} = 2 \Cov(\Ltrue, L_1^{(k)}-\Lhat_1^{(k)}) - \frac{k-1}{k}\Cov(L_1^{(k)}, L_2^{(k)})$. Using the triangle inequality:
		$$
		|\mathcal{C}| \le |2 \Cov(\Ltrue, L_1^{(k)}-\Lhat_1^{(k)})| + \left|\frac{k-1}{k}\Cov(L_1^{(k)}, L_2^{(k)})\right|
		$$
		We bound each term on the right-hand side separately. First, using the general identity for $\Var(L_1^{(k)} - \Lhat_1^{(k)})$ and the law of total variance:
		\begin{align*}
			\Var(L_1^{(k)} - \Lhat_1^{(k)}) &= \Var(L_1^{(k)}) + \Var(\Lhat_1^{(k)}) - 2\Cov(L_1^{(k)}, \Lhat_1^{(k)}) \\
			&= \Var(L_1^{(k)}) + \left(\Var(L_1^{(k)}) + \frac{\bar\sigma^2}{m}\right) - 2\Var(L_1^{(k)}) = \frac{\bar\sigma^2}{m}
		\end{align*}
		Applying this to the covariance bound together with Cauchy-Schwarz:
		\begin{align*}
			|2 \Cov(\Ltrue, L_1^{(k)}-\Lhat_1^{(k)})| &\le 2 \sqrt{\Var(\Ltrue) \cdot \Var(L_1^{(k)} - \Lhat_1^{(k)})} \\
			&\le 2 \sqrt{\sigma_n^2 \cdot \frac{\bar\sigma^2}{m}} 
		\end{align*}
		
	Next, we bound
		\begin{align*}
			\left|\frac{k-1}{k}\Cov(L_1^{(k)}, L_2^{(k)})\right| \le \frac{k-1}{k} \sqrt{\Var(L_1^{(k)})\Var(L_2^{(k)})} \le \frac{k-1}{k} \sigma_{n-m}^2
		\end{align*}
		
		Combining the bounds for the two components gives the final result.
	\end{proof}

\subsection{Proof of Lemma~\ref{lem:cov}}\label{app:cov}
\begin{proof}
We can associate each random variable $F_i
:=(\Lhat_i - \bar L^k)$ with a vector in euclidean space given as $x_i:=[F_i(S^n_1)\cdot\sqrt{\mathbb P(S^n_1)},\dots,F_i(S^n_d)\cdot\sqrt{\mathbb P(S^n_d)}]$ where $S^n_i$ are all the samples in the support of $\mathcal D^n$. With these associations, it is easy to verify that the standard inner product between the $F_i$ equals the euclidean inner product of the corresponding vectors in euclidean space. We know that by symmetry that all of the pairwise inner products with $i\neq j$ are the same number, hence the problem boils down to determining how negative this number can at most become.

Given $k$ $d$-dimensional vectors $x_1, \dots, x_k$ such that $\|x_i\|_2^2 = \|x_j\|_2^2=:\|x\|_2^2$ for all $i,j$, and $x_i^T x_j =: \alpha$ for all $i \neq j$, we form the Gram matrix $G\in \mathbb R^{k \times k}$ with entries $G_{ij} = x_i^T x_j$:
$$G = \begin{pmatrix}
\|x\|_2^2 & \alpha & \dots & \alpha \\
\alpha & \|x\|_2^2 & \dots & \alpha \\
\vdots & \vdots & \ddots & \vdots \\
\alpha & \alpha & \dots & \|x\|_2^2
\end{pmatrix}$$
As $G$ is a Gram matrix, it must be positive semidefinite. The eigenvalues of this specific matrix structure are $\lambda_1 = \|x\|_2^2 + (k-1)\alpha$ (with multiplicity 1) and $\lambda_2 = \|x\|_2^2 - \alpha$ (with multiplicity $k-1$).
For $G$ to be positive semidefinite, all eigenvalues must be non-negative:
\begin{align*}
\|x\|_2^2 - \alpha &\ge 0\\
\|x\|_2^2 + (k-1)\alpha &\ge 0
\end{align*}
The first inequality yields $\alpha \le \|x\|_2^2$ and assuming $k > 1$, the second inequality provides the lower bound $\alpha \ge -\frac{\|x\|_2^2}{k-1}$.
The claim now follows from the fact that $\|x_i\|_2^2=\text{Var}(\Lhat_i) \le \sigma_{n-m}^2+1/(4m)$ and recalling that $k=n/m$.
\end{proof}

\subsection{Proof of Lemma~\ref{lem:boundsSLS}}

\begin{proof}
Let $X_1, \ldots, X_k$ be $k$ identically distributed random variables, each with mean $\E[X_i] = \mu_X$ and variance $\Var(X_i) = \sigma_X^2$. Let $Y$ be a random variable with mean $\E[Y] = \mu_Y$ and variance $\Var(Y) = \sigma_Y^2$. It is assumed that the variables $X_i$ take values in the interval $[0,1]$ for all $i \in \{1, \ldots, k\}$, and so does $Y$.

Define the sample mean $\bar{X} = \frac{1}{k}\sum_{i=1}^k X_i$. Let $\sigma_{\bar{X}}^2 = \Var(\bar{X})$ and define the squared difference $Z = (\bar{X} - Y)^2$.

We will prove the general statement that
$$ (\mu_X - \mu_L)^2 \le \E[Z] \le \left(\sigma_{\bar{X}}+ \sigma_L\right)^2 + (\mu_X - \mu_L)^2 $$
where the variance of the sample mean, $\sigma_{\bar{X}}^2$, is given by
$$ \sigma_{\bar{X}}^2 = \frac{1}{k^2} \left( k\sigma_X^2 + \sum_{i\neq j}  \Cov(X_i, X_j) \right) $$

\textbf{Lower Bound:}
Applying Jensen's inequality, we have
$$ \E[Z] =  \E[(\bar{X} - Y)^2] \ge (\E[\bar{X} - Y])^2 =(\mu_X - \mu_Y)^2$$

\textbf{Upper Bound:}
We use the property that for any random variable $D$, $\E[D^2] = \Var(D) + (\E[D])^2$.
Applied to $D = \bar{X} - Y$, this gives
$$ \E[Z] =\E[D^2] =\Var(\bar{X} - Y) + (\E[\bar{X} - Y])^2 = \Var(\bar{X} - Y) + (\mu_X - \mu_Y)^2 $$
The variance of the difference, $\Var(\bar{X} - Y)$, can be expanded as
$$ \Var(\bar{X} - Y) = \Var(\bar{X}) + \Var(Y) - 2\Cov(\bar{X}, Y) $$
So,
$$ \Var(\bar{X} - Y) = \sigma_{\bar{X}}^2 + \sigma_Y^2 - 2\Cov(\bar{X}, Y) $$
By the Cauchy-Schwarz inequality, $|\Cov(\bar{X}, Y)| \le \sigma_{\bar{X}} \sigma_Y$ hence $-2\Cov(\bar{X}, Y) \le 2\sigma_{\bar{X}} \sigma_Y$.
Substituting this into the expression for $\Var(\bar{X} - Y)$
$$ \Var(\bar{X} - Y) \le \sigma_{\bar{X}}^2 + \sigma_Y^2 + 2\sigma_{\bar{X}} \sigma_Y = \left(\sigma_{\bar{X}}+ \sigma_Y\right)^2$$
Substituting this inequality back into the expression for $\E[Z]$:
$$ \E[Z] = \Var(\bar{X} - Y) + (\mu_X - \mu_Y)^2 \le \left(\sigma_{\bar{X}}+ \sigma_Y\right)^2 + (\mu_X - \mu_Y)^2 $$
establishes the upper bound.

Finally, plugging in $X_i = L^k_i, Y=L$ and using $\sigma_{L^k}^2\leq \sigma_{n-m}^2$, $\Cov(X_i, X_j)\leq \sigma_{n-m}^2$ completes the proof.
\end{proof}

\newpage

\section{Majority Algorithm}\label{app:majority}

Throughout this section, we consider the following setup.
\subsection{Setup and Notation}
	For $1\le m\le n/2$, $m|n$, let $N:=n-2m$ and define
	\[
	E(n, m):=
	2^{-(n-2)}\sum_{j=0}^{m-1}\binom{m-1}{j}^2
	\binom{n-2m}{\lfloor (n-m)/2\rfloor-j}.
	\]
    such that $\Cov(\hat L_1, \hat L_2)=E(n,m)/4$ (see Thm.~\ref{thm:maj_exact-combinatorial} for details).
    We also denote $\Cov(n.m) \equiv\Cov(\hat L_1, \hat L_2)$ to highlight the roles of $n,m$.
    
	Let $\mathsf B_r\sim\mathrm{Bin}(r,\tfrac12)$ with pmf
	$p_r(t)=2^{-r}\binom{r}{t}$, and denote the Gaussian proxy
	\[
	g_r(t):=\sqrt{\frac{2}{\pi r}}\,
	\exp\!\Big(-\frac{(2t-r)^2}{2r}\Big)
	\]
	and central binomial mass
	\[
	S_{r}:=2^{-2r}\binom{2r}{r}
		\]

\subsection{Main Theorem}
\begin{theorem}[Fold-Covariance of the Majority Algorithm]
	Throughout, let $n\ge 2$ and $m|n$.
	
	\medskip

	\noindent\textbf{(A) Binomial form.} One has
	\[
	\Cov(n,m)\;=\;S_{m-1}\,\frac{1}{2\sqrt{\pi (2n-3m)}}\;+O(\sqrt{m}/n^{3/2}),
	\]
	uniformly for all $1\le m\le n/3$, where $S_{m-1}:=2^{-(2m-2)}\binom{2m-2}{m-1}$.
	
	\medskip
    	
	\noindent{\bf (B) Exact expression for $m=1$.}
It holds that
	\[
	\Cov(n,1)=2^{-n}\binom{\,n-2\,}{\big\lfloor \frac{n-1}{2}\big\rfloor}=\sqrt{\frac{1}{8\pi(n-2)}}\;+\;O\!\Big(\frac{1}{n^{3/2}}\Big)
	=\sqrt{\frac{1}{8\pi n}}\;+\;O\!\Big(\frac{1}{n}\Big).
	\]
    \medskip
    
	\noindent\textbf{(C) Sublinear regime.}
	For all $\Omega(n^{1/5})= m =o(n)$,
	\[
	\Cov(n,m)
	=\frac{1}{2\pi\sqrt{(m-1)(2n-3m)}}
	\Big(1-\frac{1}{8(m-1)}\Big)
	\;+O (n^{-1}).
	\]
    \medskip
    
	\noindent\textbf{(D) Large $m$ regime.}
    For all $\Omega\left( n^{2/3} \log^{1/3} n \right)= m\le n/3$,
		
		\[
		\Cov(n,m)=\frac{1}{2\pi\sqrt{(m-1)(2n-3m)}}+O\bigg(\frac{1}{\sqrt{n}m^{3/2}}\bigg).
		\]
        	\medskip
            
	\noindent{\bf (E) Exact expression for $m=n/2$.}
	It holds that
		\[
	\Cov\Big(n,\tfrac{n}{2}\Big)=\frac{1}{\pi(n-2)}+O\!\Big(\frac{1}{n^{2}}\Big)
	=\frac{1}{\pi n}+O\!\Big(\frac{1}{n^{2}}\Big).
	\]
        \medskip
        
    	\noindent{\bf (F) Monotonicity and minimizer.}
	For all sufficiently large $n$,
	\[
	\Cov(n,1)>\Cov(n,2)>\cdots>\Cov\big(n, n/3\big)<\Cov\big(n, n/2\big),
	\]
	and consequently
	\[
	\arg\min_{\substack{m\mid n\\ 1\le m\le \lfloor n/2\rfloor}} \Cov(n,m)
	\;=\;\max\{d\mid n:\ d\le n/3\}.
	\]
\end{theorem}
\begin{proof}
This is a consequence of collecting the results of Theorems~\cref{thm:maj_small-k-mono,thm:maj_sublin,thm:maj_large-m,thm:maj_mono-minimizer}.
\end{proof}

\bigskip\bigskip

\subsection{Technical Lemmas}
Let us first state a few technical results.

\begin{lemma}[Triple Gaussian Product]\label{lem:twoG}
	Let $P(j) := g_{m-1}(j)^2\,g_{N}(\ell-j)$. With the parameters
	\[
	\alpha:=\frac{4}{m-1},\qquad
	\beta :=\frac{2}{N},\qquad
	\mu:=\frac{\alpha\cdot\frac{m-1}{2}+\beta\cdot\frac{m}{2}}{\alpha+\beta}
	=\frac{(m-1)(2N+m)}{2(2N+m-1)},
	\]
	the product $P(j)$ can be written as:
	\begin{align*}
		P(j)
		&= \left( \frac{2}{\pi(m-1)} \sqrt{\frac{2}{\pi N}} \right)
		\exp\!\left( - \frac{1}{2N+m-1} \right)
		\exp\!\Big(-(\alpha+\beta)(j-\mu)^2\Big).
	\end{align*}
	Furthermore, the sum of the rates is
	\[
	\alpha+\beta = \frac{2(2N+m-1)}{(m-1)N}.
	\]
\end{lemma}

\begin{proof}
	Recall
	\[
	g_r(t):=\sqrt{\frac{2}{\pi r}}\,
	\exp\!\Big(-\frac{(2t-r)^2}{2r}\Big).
	\]
	Let $N:=n-2m$ and $\ell=(n-m)/2$.
	
	We first write out the terms. Let $a:=(m-1)/2$ and $\alpha:=4/(m-1)$.
	\[
	g_{m-1}(j)^2 = \left(\sqrt{\frac{2}{\pi(m-1)}}\right)^2
	\exp\!\Big(-2 \cdot \frac{2}{m-1}(j-\tfrac{m-1}{2})^2\Big)
	= \frac{2}{\pi(m-1)}\,e^{-\alpha(j-a)^2}.
	\]
	For the second term, let $b:=m/2$ and $\beta:=2/N$. The exponent's center is
	$\ell-j - \frac{N}{2} = \frac{n-m}{2} - j - \frac{n-2m}{2} = \frac{m}{2} - j = -(j-b)$.
	Thus,
	\[
	g_N(\ell-j) = \sqrt{\frac{2}{\pi N}}\,
	\exp\!\Big(-\frac{2}{N}(\ell-j-\tfrac N2)^2\Big)
	= \sqrt{\frac{2}{\pi N}}\,e^{-\beta(j-b)^2}.
	\]
	The product is
	\[
	g_{m-1}(j)^2\,g_N(\ell-j) =
	\underbrace{\left( \frac{2}{\pi(m-1)} \sqrt{\frac{2}{\pi N}} \right)}_{:=C_{\text{prod}}}\,
	\exp\{-\alpha(j-a)^2-\beta(j-b)^2\}.
	\]
	We complete the square for the exponential terms:
	\[
	-\alpha(j-a)^2-\beta(j-b)^2
	=-(\alpha+\beta)(j-\mu)^2-\frac{\alpha\beta}{\alpha+\beta}(a-b)^2,
	\]
	where $\mu:=(\alpha a+\beta b)/(\alpha+\beta)$ is as stated in the lemma.
	The constant term in the exponent depends on $a-b = (m-1)/2 - m/2 = -1/2$.
	\[
	\frac{\alpha\beta}{\alpha+\beta}(a-b)^2
	= \frac{1}{4} \cdot \frac{\frac{4}{m-1} \cdot \frac{2}{N}}{\frac{4}{m-1} + \frac{2}{N}}
	= \frac{1}{4} \cdot \frac{8/((m-1)N)}{(4N+2m-2)/((m-1)N)}
	= \frac{2}{4N+2m-2} = \frac{1}{2N+m-1}.
	\]
	We also compute
	\[
	\alpha+\beta = \frac{4}{m-1} + \frac{2}{N} = \frac{4N+2(m-1)}{(m-1)N} = \frac{2(2N+m-1)}{(m-1)N}.
	\]
	Combining these results yields the displayed formula.
\end{proof}

\begin{lemma}[Poisson summation for Gaussians]\label{lem:PSF}
	Let $\gamma>0$ and $\mu\in\mathbb R$. Define
	\[
	f_{\gamma,\mu}(x):=e^{-\gamma(x-\mu)^2}.
	\]
	Then
	\begin{equation}\label{eq:PSF}
		\sum_{j\in\mathbb Z} f_{\gamma,\mu}(j)
		=\sqrt{\frac{\pi}{\gamma}}\,
		\sum_{t\in\mathbb Z} e^{-\pi^2 t^2/\gamma}\,e^{-2\pi i t\mu}.
	\end{equation}
\end{lemma}

\begin{proof}
	Let $\mathcal P_{\gamma,\mu}(x):=\sum_{j\in\mathbb Z} f_{\gamma,\mu}(x+j)$ be the periodisation (absolutely and uniformly convergent on $\mathbb R$). Then $\mathcal P_{\gamma,\mu}$ is $1$–periodic and belongs to $C^\infty$. Its complex Fourier series is
	\[
	\mathcal P_{\gamma,\mu}(x)=\sum_{t\in\mathbb Z} c_t\,e^{2\pi i t x},
	\qquad
	c_t=\int_0^1\mathcal P_{\gamma,\mu}(x)\,e^{-2\pi i t x}\,dx.
	\]
	By absolute convergence we may integrate termwise:
	\[
	c_t=\sum_{j\in\mathbb Z}\int_0^1 e^{-\gamma(x+j-\mu)^2}\,e^{-2\pi i t x}\,dx
	=\int_{\mathbb R} e^{-\gamma(y-\mu)^2}\,e^{-2\pi i t y}\,dy
	=: \widehat f_{\gamma,\mu}(t),
	\]
	after the change of variables $y=x+j$. The Gaussian Fourier transform is standard:
	\[
	\widehat f_{\gamma,\mu}(t)
	=e^{-2\pi i t\mu}\int_{\mathbb R}e^{-\gamma z^2}e^{-2\pi i t z}\,dz
	=e^{-2\pi i t\mu}\sqrt{\frac{\pi}{\gamma}}\;e^{-\pi^2 t^2/\gamma}.
	\]
	Thus
	\[
	\mathcal P_{\gamma,\mu}(x)
	=\sqrt{\frac{\pi}{\gamma}}\sum_{t\in\mathbb Z}e^{-\pi^2 t^2/\gamma}\,e^{2\pi i t(x-\mu)}.
	\]
	Evaluating at $x=0$ gives
	\[
	\sum_{j\in\mathbb Z} f_{\gamma,\mu}(j)
	=\mathcal P_{\gamma,\mu}(0)
	=\sqrt{\frac{\pi}{\gamma}}\sum_{t\in\mathbb Z}e^{-\pi^2 t^2/\gamma}\,e^{-2\pi i t\mu},
	\]
	which is \eqref{eq:PSF}.
\end{proof}

\begin{proposition}[Lattice sum of the triple Gaussian]\label{prop:triple-sum-theta}
	With $N=n-2m$, $\ell=(n-m)/2$, and the parameters
	\[
	\alpha=\frac{4}{m-1},\quad \beta=\frac{2}{N},\quad
	\mu=\frac{(m-1)(2N+m)}{2(2N+m-1)},
	\]
	we have the exact identity
	\begin{align}
		\sum_{j\in\mathbb Z} g_{m-1}(j)^2\,g_N(\ell-j)
		&= \frac{2}{\pi}\cdot \frac{1}{\sqrt{(m-1)(2N+m-1)}}\,
		e^{-\frac{1}{\,2N+m-1\,}}\;
		\Theta_{n,m},\label{eq:triple-theta}\\
		\Theta_{n,m}
		&:= \sum_{t\in\mathbb Z}
		\exp\!\Big(-\pi^2 t^2/(\alpha+\beta)\Big)\,
		\exp\!\big(-2\pi i t \mu\big).\label{eq:theta-def}
	\end{align}
	Equivalently, using $\alpha+\beta=\dfrac{2(2N+m-1)}{(m-1)N}$,
	\begin{align}
		\sum_{j\in\mathbb Z} &g_{m-1}(j)^2\,g_N(\ell-j)\notag\\
		&= \frac{2}{\pi}\cdot \frac{1}{\sqrt{(m-1)(2n-3m-1)}}\,
		e^{-\frac{1}{\,2n-3m-1\,}}\;
		\sum_{t\in\mathbb Z}
		\exp\!\Big(-\frac{\pi^2 t^2 (m-1)N}{2(2n-3m-1)}\Big)\,
		e^{-2\pi i t \mu}\label{eq:triple-theta-simplified}.
	\end{align}
\end{proposition}

\begin{proof}
	By Lemma~\ref{lem:twoG}, we have
	\[
	g_{m-1}(j)^2\,g_N(\ell-j)
	= C_{\text{prod}} \cdot
	e^{-\frac{1}{\,2N+m-1\,}}\;
	\exp\!\Big(-(\alpha+\beta)(j-\mu)^2\Big),
	\]
	where $C_{\text{prod}} = \frac{2}{\pi(m-1)} \sqrt{\frac{2}{\pi N}}$.
	Summing over $j\in\mathbb Z$ and applying Lemma~\ref{lem:PSF} with
	$\gamma:=\alpha+\beta$, we obtain
	\begin{align*}
		\sum_{j\in\mathbb Z} g_{m-1}(j)^2\,g_N(\ell-j)
		&= C_{\text{prod}} \cdot e^{-\frac{1}{\,2N+m-1\,}}
		\sum_{j\in\mathbb Z} e^{-(\alpha+\beta)(j-\mu)^2} \\
		&= C_{\text{prod}} \cdot e^{-\frac{1}{\,2N+m-1\,}}
		\cdot \sqrt{\frac{\pi}{\alpha+\beta}}\sum_{t\in\mathbb Z}
		e^{-\pi^2 t^2/(\alpha+\beta)}e^{-2\pi i t\mu}.
	\end{align*}
	We now compute the combined prefactor. Using $\alpha+\beta = \frac{2(2N+m-1)}{(m-1)N}$ from Lemma~\ref{lem:twoG}:
	\begin{align*}
		C_{\text{prod}}\;\sqrt{\frac{\pi}{\alpha+\beta}}
		&= \left( \frac{2}{\pi(m-1)} \sqrt{\frac{2}{\pi N}} \right)
		\cdot \sqrt{\frac{\pi (m-1) N}{2(2N+m-1)}} \\
		&= \left( \frac{2\sqrt{2}}{\pi^{3/2} (m-1) \sqrt{N}} \right)
		\cdot \left( \frac{\sqrt{\pi} \sqrt{m-1} \sqrt{N}}{\sqrt{2} \sqrt{2N+m-1}} \right) \\
		&= \frac{2}{\pi \sqrt{m-1} \sqrt{2N+m-1}}
		= \frac{2}{\pi} \cdot \frac{1}{\sqrt{(m-1)(2N+m-1)}}.
	\end{align*}
	Substituting this prefactor back into the sum yields \eqref{eq:triple-theta}.
	
	For \eqref{eq:triple-theta-simplified}, we substitute the expression for $\alpha+\beta$ into the exponent and use $N=n-2m$ in the denominator, noting that $2N+m-1 = 2(n-2m)+m-1 = 2n-3m-1$.
\end{proof}

\begin{lemma}[Local Limit Theorem and Central binomial]\label{lem:LLT}
	Let $r\ge2$, $c:=\lfloor r/2\rfloor$ and $p_r(t):=2^{-r}\binom{r}{t}$.
	Let $g_r(t):=\sqrt{\frac{2}{\pi r}}\exp\!\big(-(2t-r)^2/(2r)\big)$.
	There exists an absolute $C_0>0$ such that
	\begin{equation}\label{eq:LLT}
		\sup_{t\in\mathbb Z}\,|\,p_r(t)-g_r(t)\,|\ \le\ C_0\,r^{-3/2}.
	\end{equation}
	In particular, at the center $t=c$,
	\begin{equation}\label{eq:central-LLT}
		\Big|\,p_r(c)-g_r(c)\,\Big|
		\ \le\ C_0\,r^{-3/2},
		\qquad
		g_r(c)=
		\begin{cases}
			\sqrt{\frac{2}{\pi r}}, & r \text{ even},\\[1ex]
			\sqrt{\frac{2}{\pi r}}\,e^{-1/(2r)}, & r \text{ odd}.
		\end{cases}
	\end{equation}
	Hence, for all $r\ge2$,
	\begin{equation}\label{eq:central-LLT-bds}
		\sqrt{\frac{2}{\pi r}}\,e^{-1/(2r)}-C_0\,r^{-3/2}
		\ \le\ p_r(c)\ \le\ \sqrt{\frac{2}{\pi r}}+C_0\,r^{-3/2}.
	\end{equation}
\end{lemma}

\begin{proof}
	This is a classical uniform local limit theorem, see \cite[Chapter 7, Theorem 13]{petrov2012sums} (with $p=q=\tfrac12$). Evaluating at $t=c$ gives \eqref{eq:central-LLT}; the
	bounds \eqref{eq:central-LLT-bds} follow since $g_r(c)$ is as displayed.
\end{proof}

\bigskip\bigskip

\subsection{Simplifying the Fold-Covariance}

\begin{lemma}\label{lem:majCov}
It holds that
	\begin{equation*}
		\Cov(\hat L_1^{(k)},\hat L_2^{(k)}) =  \frac{4}{k^2} \E_{Y}\left[\left(C(k, Y)\right)^2\right].
	\end{equation*}
    where $C(k, Y) = \text{Cov}_{X_k}\left(X_k, \mathbf{1}_{X_k > m-Y}\right)$, $m=(n-k)/2$, $X_k \sim \text{Bin}(k,1/2)$ and $Y \sim \text{Bin}(n-2k,1/2)$ independent of each other.
\end{lemma}
\begin{proof}
Assume that $n-k$ is odd (to avoid ties) and that $k$ divides $n$. Define $C(k, Y) = \text{Cov}_{X_k}\left(X_k, \mathbf{1}_{X_k > m-Y}\right)$, as a covariance conditioned on $Y$ and let $m=(n-k)/2$. Define $X_1,X_2 \sim \text{Bin}(k,1/2)$ and $Y \sim \text{Bin}(n-2k,1/2)$ all independent were we interpret $X_1=\{\text{Number of ones in the first fold}\}$, $X_2=\{\text{Number of ones in the second fold}\}$, $Y=\{\text{Combined number of ones in the folds $3,\dots,N$}\}$. Let $p, q$ denote the probability mass functions corresponding to $X_1,Y$. By the law of total expectation it holds that
$$
\mathbb E[\hat L_1^{(k)} \cdot\hat L_2^{(k)}]=\sum_{t=0}^{n-2k} \sum_{i,j=0}^k p(i)  p(j)  q(t) \cdot f(i,j,t)/k^2
$$
where 
$$
f(i,j,t) =
\begin{cases}
    (k-j)(k-i) & \text{if } t+i > (n-k)/2 \text{ and } t+j > (n-k)/2 \\
    2(k-j)i & \text{if } t+i > (n-k)/2 \text{ and } t+j < (n-k)/2 \\
    ij & \text{if } t+i < (n-k)/2 \text{ and } t+j < (n-k)/2 \\
    0 & \text{else}.
\end{cases}
$$
 The piece-wise defined function $f$ can be explained as follows: when computing $\hat L^k_1$, we count the number of zeros in the first fold (i.e. $k-i$) as errors if the algorithm outputs the constant-one hypothesis which happens precisely when $t+j > (n-k)/2$, and else we count the number of ones ($i$). The same principle applies for the second fold with the roles of $i$ and $j$ reversed. The second case captures the case where exactly one of $\hat L_1,\hat L_2$ count zeros, and the other one counts ones.

    Define $E=\mathbb E[\hat L_1^{(k)} \cdot\hat L_2^{(k)}]=\text{Cov}(\hat L_1^{(k)},\hat L_2^{(k)})+\mathbb E[\hat L_1^{(k)}]\mathbb E[\hat L_2^{(k)}]=\text{Cov}(\hat L_1^{(k)},\hat L_2^{(k)})+1/4$.

	With our definitions we can write $\mathbb E[\hat L_1^{(k)} \cdot\hat L_2^{(k)}] = \frac{1}{k^2} \E_{Y} [\E[f(X_1,X_2,Y)|Y]]$, where $\E[f(X_1,X_2,Y)|Y] = \sum_{i,j} p(i) p(j)\cdot f(i,j,Y)$. Let $c(Y) = m-Y$, $P_+(Y) = P(X_k > c(Y) | Y)$ and $P_-(Y) = P(X_k \le c(Y) | Y)$, so $P_+(Y) + P_-(Y) = 1$.
	Now the conditional expectation $\E[f(X_1,X_2,Y)|Y]$ can be expressed using auxiliary functions $Q_s(Y)$ as 
    \begin{align*}
        \E[f(X_1,X_2,Y)|Y] 
        &= \E\Big[(k-X_2)(k-X_1)\cdot\ind{Y+X_1 > (n-k)/2}\ind{ Y+X_2 > (n-k)/2}\\
        &~~~~~~+2(k-X_2)X_1 \cdot\ind{Y+X_1 > (n-k)/2}\ind{ Y+X_2 < (n-k)/2}\\
        &~~~~~~+X_1 X_2 \cdot\ind{ Y+X_1 < (n-k)/2}\ind{Y+X_2 < (n-k)/2}|Y
        \Big]\\
        &= \E_{X_1}\Big[(k-X_1)\ind{Y+X_1 > (n-k)/2}|Y\Big]\E_{X_2}\Big[(k-X_2)\ind{ Y+X_2 > (n-k)/2}|Y\Big]\\
        &~~~~~~+2\E_{X_1}\Big[X_1 \ind{Y+X_1 > (n-k)/2}|Y\Big]\E_{X_2}\Big[(k-X_2)\ind{ Y+X_2 < (n-k)/2}|Y\Big]\\
        &~~~~~~+\E_{X_1}\Big[X_1  \ind{ Y+X_1 < (n-k)/2}|Y\Big]\E_{X_2}\Big[X_2\ind{Y+X_2 < (n-k)/2}|Y\Big]
        \\
        &=Q_1(Y)^2 + 2Q_2(Y)Q_3(Y) + Q_4(Y)^2
    \end{align*}
    where
	\begin{itemize}
		\item $Q_1(Y) = \E_{X_k}[(k-X_k)\ind{X_k > c(Y)}|Y]$
		\item $Q_2(Y) = \E_{X_k}[X_k\ind{X_k > c(Y)}|Y]$
		\item $Q_3(Y) = \E_{X_k}[(k-X_k)\ind{X_k \le c(Y)}|Y]$
		\item $Q_4(Y) = \E_{X_k}[X_k\ind{X_k \le c(Y)}|Y].$
	\end{itemize}
	Let $C(Y)$ denote $C(k, Y)$ for brevity within this derivation. By definition of covariance, $C(Y) = Q_2(Y) - (k/2)P_+(Y)$. Thus, $Q_2(Y) = (k/2)P_+(Y) + C(Y)$.
	Similarly, we deduce:
	$Q_1(Y) = (k/2)P_+(Y) - C(Y)$,
	$Q_3(Y) = (k/2)P_-(Y) + C(Y)$ and
	$Q_4(Y) = (k/2)P_-(Y) - C(Y)$.
	
	We can now compute the components necessary for $\E[f(X_1,X_2,Y)|Y]$
	\begin{align*}
		Q_1(Y)^2 &= \left(\tfrac{k}{2}P_+(Y) - C(Y)\right)^2 = \left(\tfrac{k}{2}\right)^2P_+(Y)^2 - k P_+(Y)C(Y) + C(Y)^2 \\
		2Q_2(Y)Q_3(Y) &= 2\left(\tfrac{k}{2}P_+(Y) + C(Y)\right)\left(\tfrac{k}{2}P_-(Y) + C(Y)\right) \\
		&= 2\left(\tfrac{k}{2}\right)^2P_+(Y)P_-(Y) + k P_+(Y)C(Y) + k P_-(Y)C(Y) + 2C(Y)^2 \\
		Q_4(Y)^2 &= \left(\tfrac{k}{2}P_-(Y) - C(Y)\right)^2 = \left(\tfrac{k}{2}\right)^2P_-(Y)^2 - k P_-(Y)C(Y) + C(Y)^2
	\end{align*}
	Summing these components, the terms linear in $C(Y)$ cancel out
	$$ -kP_+(Y)C(Y) + kP_+(Y)C(Y) + kP_-(Y)C(Y) - kP_-(Y)C(Y) = 0 $$
	and the remaining terms are
	\begin{align*}
		\E[f(X_1,X_2,Y)|Y]&= \left(\tfrac{k}{2}\right)^2\left(P_+(Y)+P_-(Y)\right)^2 + 4C(Y)^2 \\
		&= \frac{k^2}{4} + 4C(Y)^2.
	\end{align*}
	Finally, we take the expectation with respect to $Y$
	$$ E = \frac{1}{k^2} \E_{Y}\left[\frac{k^2}{4} + 4C(k,Y)^2\right] = \frac{1}{k^2} \left(\frac{k^2}{4} + 4\E_{Y}[C(k,Y)^2]\right). $$
	Recalling that $E=\text{Cov}(\hat L_1^{(k)},\hat L_2^{(k)})+1/4$, this completes the derivation.
\end{proof}

\begin{lemma}[Simplification of Covariance Term]
	\label{lemma:cov_simplification_final}
	Let $X_k \sim \text{Bin}(k, 1/2)$, $m=(n-k)/2$, and $a(Y) = \lfloor m-Y \rfloor$. Then,
	$$ C(k, Y) = \text{Cov}_{X_k}(X_k, \mathbf{1}_{X_k > m-Y}) = \frac{k}{4} P(X_{k-1} = a(Y)) $$
	where $X_{k-1} \sim \text{Bin}(k-1, 1/2)$.
\end{lemma}
\begin{proof}
	Let $a(Y) = \lfloor m-Y \rfloor$. The event $X_k > m-Y$ is equivalent to $X_k \ge \lfloor m-Y \rfloor + 1 = a(Y)+1$.
	The covariance $C(k,Y) = \mathbb{E}_{X_k}[X_k \mathbf{1}_{X_k \ge a(Y)+1}] - \mathbb{E}_{X_k}[X_k] P(X_k \ge a(Y)+1)$.
	Since $X_k \sim \text{Bin}(k, 1/2)$, its expectation is $\mathbb{E}_{X_k}[X_k] = k/2$.
	The first term is $\mathbb{E}_{X_k}[X_k \mathbf{1}_{X_k \ge a(Y)+1}] = \sum_{j=a(Y)+1}^k j \binom{k}{j} (1/2)^k$. Using $j \binom{k}{j} = k \binom{k-1}{j-1}$:
	$$ \sum_{j=a(Y)+1}^k k \binom{k-1}{j-1} (1/2)^k = \frac{k}{2} \sum_{j'=a(Y)}^{k-1} \binom{k-1}{j'} (1/2)^{k-1} = \frac{k}{2} P(X_{k-1} \ge a(Y)) $$
	where $X_{k-1} \sim \text{Bin}(k-1, 1/2)$.
	So, $C(k,Y) = \frac{k}{2} P(X_{k-1} \ge a(Y)) - \frac{k}{2} P(X_k \ge a(Y)+1)$.
	To simplify $P(X_k \ge j+1)$, let $X_k = X_{k-1} + B_k$, where $B_k \sim \text{Bernoulli}(1/2)$ is independent of $X_{k-1}$.
	\begin{align*} P(X_k \ge j+1) &= P(X_{k-1} + B_k \ge j+1 | B_k=0)P(B_k=0) + P(X_{k-1} + B_k \ge j+1 | B_k=1)P(B_k=1) \\ &= \frac{1}{2}P(X_{k-1} \ge j+1) + \frac{1}{2}P(X_{k-1} \ge j)\end{align*}
	Substituting this with $j=a(Y)$:
	\begin{align*} C(k,Y) &= \frac{k}{2} \left[P(X_{k-1} \ge a(Y)) - \left(\frac{1}{2}P(X_{k-1} \ge a(Y)+1) + \frac{1}{2}P(X_{k-1} \ge a(Y))\right)\right] \\ &= \frac{k}{4} \left[P(X_{k-1} \ge a(Y)) - P(X_{k-1} \ge a(Y)+1)\right] = \frac{k}{4} P(X_{k-1} = a(Y)) \qedhere \end{align*}
\end{proof}

\subsection{Exact Combinatorial Form of the Fold-Covariance}\label{app:maj_exact-combinatorial}
\begin{theorem}\label{thm:maj_exact-combinatorial}
It holds that
\[
\Cov(\hat L_1, \hat L_2)=2^{-n} \sum_{j=0}^{k-1} \binom{k-1}{j}^2 \binom{n-2k}{\lfloor (n-k)/2-j \rfloor}
\]
\end{theorem}
\begin{proof}
We know from the previous Lemmas that $\Cov(\hat L_1, \hat L_2)=\tfrac{1}{4} \E_Y[ P(X=\lfloor m-Y \rfloor)^2]$.
\begin{align*} 
	 \E_Y[ P(X=&\lfloor m-Y \rfloor)^2]\\
	&= \E_Y[ P(X_1=\lfloor m-Y \rfloor, X_2=\lfloor m-Y \rfloor | Y) ] \quad (\text{introducing } X_1, X_2 \text{ cond. indep. given } Y) \\
	&= P(X_1=\lfloor m-Y \rfloor, X_2=\lfloor m-Y \rfloor) \quad (\text{by Law of Total Expectation}) \\
	&= \sum_{j=0}^{k-1} P(X_1 = j, X_2=j, \text{ and } j=\lfloor m-Y \rfloor) \quad (\text{summing over the support of } X_1, X_2)\\
	&= \sum_{j=0}^{k-1} P(X_1=j, X_2=j) P(j=\lfloor m-Y \rfloor) \quad (\text{by independence of } (X_1, X_2) \text{ and } Y)\\
	&= \sum_{j=0}^{k-1} P(X_1=j)^2P(j=\lfloor m-Y \rfloor) \quad (\text{by independence of } X_1, X_2)\\
	&=\left(\frac{1}{2^{n-2}}\right) \sum_{j=0}^{k-1} \binom{k-1}{j}^2 \binom{n-2k}{\lfloor (n-k)/2-j \rfloor}  \quad (\text{writing out definition})
\end{align*}

\end{proof}

\bigskip\bigskip

\subsection{Main Results for the Majority Algorithm Fold-Covariance}

\begin{theorem}[Sublinear $m$ regime]\label{thm:maj_sublin}
	Fix integers $n$ and $m|n$.
Let $N:=n-2m$, and choose the \emph{parity–adjusted} integer
	\[
	N_c\in\{\lfloor n-\tfrac32 m\rfloor,\ \lceil n-\tfrac32 m\rceil\}
	\quad\text{such that }\ N_c\equiv N\pmod 2,
	\]
	so that $N_c/2$ is an admissible central index for $\mathrm{Bin}(N_c,\tfrac12)$.

	We bound $\Cov(n, m)$ explicitly with error terms that are asymptotically negligible compared to the main term as long as $m=o(n)$.
	
	\medskip
	\noindent\textbf{(A) Precise binomial form.} One has
	\[
	\Cov(n, m)\;=\;S_{m-1}\,\frac{1}{2\sqrt{\pi (2n-3m)}}\;+O(\sqrt{m}/n^{3/2}),
	\]
	uniformly for all $1\le m\le n/3$.
	
	\medskip
	\noindent\textbf{(B) Explicit scalar form.}
	For all
	\[
	n^{1/5}\ \le\ m\ \le n/3
	\]
	one has
	\[
	\Cov(n, m)
	=\frac{1}{2\pi\sqrt{(m-1)(2n-3m)}}
	\Big(1-\frac{1}{8(m-1)}\Big)
	\;+O (\frac{1}{n}).
	\]
\end{theorem}
\begin{proof}
		Define
	\[
	q_t(r):=2^{-t}\binom{t}{r}.
	\]
	Set
	\[
	p(j):=2^{-(m-1)}\binom{m-1}{j},\quad
	m_0:=\Big\lfloor \frac{n-m}{2}\Big\rfloor,\quad
	P_N(r):=q_N(r),\quad P_{N_c}:=q_{N_c}(N_c/2).
	\]
	Then
	\begin{equation}\label{eq:E-split}
		E(n, m)=\sum_{j=0}^{m-1} p(j)^2\,P_N(m_0-j)
		\ =\ S_{m-1}\,P_{N_c}\;+\;R_1,
		\quad
		R_1:=\sum_{j=0}^{m-1} p(j)^2\big(P_N(m_0-j)-P_{N_c}\big).
	\end{equation}
	
	\paragraph{Step 1: LLT expansions.}
	Apply \eqref{eq:LLT} to $p_{N}(m_0-j)$ and $S_{N_c/2}$.
	\[
	p_N(m_0-j)=G_N(j)+\delta_N(j),
	\qquad
	S_{N_c/2}=G_{N_c}+\delta_{N_c},
	\]
	where
	\[
	G_N(j):=\frac{1}{\sqrt{\pi N/2}}\exp\!\Big(-\frac{2\Delta_j^2}{N}\Big),\quad
	G_{N_c}:=\frac{1}{\sqrt{\pi N_c/2}},
	\quad
	\Delta_j:=m_0-j-\frac{N}{2}=\frac{m}{2}-j-\theta,\ \theta\in[0,1),
	\]
	and $|\delta_N(j)|\le C_{\LLT}N^{-3/2}$, $|\delta_{N_c}|\le C_{\LLT}N_c^{-3/2}$.
	
	Rigorously, $P_{N_c} := p_{N_c}(c_{N_c})$ with $c_{N_c} := \lfloor N_c/2 \rfloor$, so $G_{N_c} := g_{N_c}(c_{N_c}) = \sqrt{\frac{2}{\pi N_c}}e^{-(2c_{N_c}-N_c)^2/(2N_c)}$. By \eqref{eq:central-LLT}
	this is $\sqrt{\frac{2}{\pi N_c}}$ if $N_c$ even or $\sqrt{\frac{2}{\pi N_c}}e^{-1/(2N_c)}$ if $N_c$ odd. As $e^{-1/(2N_c)}=1+O(N_c^{-1})$, in both cases $G_{N_c} = \sqrt{\frac{2}{\pi N_c}} + O(N_c^{-3/2}).$
	\paragraph{Step 2: Decomposition of $R_1$.}
	Plugging in the Gaussian approximation, we get
	\begin{equation}\label{eq:R1-main}
		R_1=\sum_{j=0}^{m-1}p_{m-1}(j)^2\Big(G_N(j)-G_{N_c}\Big)
		\;+\;\underbrace{\sum_{j=0}^{m-1}p_{m-1}(j)^2\big(\delta_N(j)-\delta_{N_c}\big)}_{=:R_{\LLT}}.
	\end{equation}
	
	\paragraph{Step 3: Bounding the pure LLT remainder.}
	By $S_{m-1}\le 1$ and the local limit theorem bound of \ref{lem:LLT},
	\[
	|R_{\LLT}|
	\ \le\ \sum_j p_{m-1}(j)^2\big(|\delta_N(j)|+|\delta_{N_c}|\big)
	\ \le\ C_{\LLT}\Big(\frac{1}{N^{3/2}}+\frac{1}{N_c^{3/2}}\Big)
	\ =\ O\!\Big(\frac{1}{n^{3/2}}\Big).
	\]
	
	\paragraph{Step 4. $N_c$ from the first–order optimal Gaussian central term.}
	Next, we bound $\sum_{j=0}^{m-1}p_{m-1}(j)^2\Big(G_N(j)-G_{N_c}\Big)$
	with $N=n-2m$, $N_c\equiv N\pmod 2$, and define the discrete distribution $w$ with probability weights
	\[
	w_j:=\frac{p(j)^2}{S_{m-1}}~,~~~j \in \{0,\ldots,m-1\}
	\]
	and corresponding expectation operator $w$: $\E_w[g]:=\sum_{j}w_j g(j)$. Let $J\sim w$. We have $\mu_1:=\E_{w}[J]=\tfrac{m-1}{2}$ by symmetry. Our goal is to choose the parameter $N_c$ as to “curvature match” the Gaussian prefactor $G_{N_c}$ to the \emph{typical} location of $m_0-j$.
	
	Since for a hypergeometric random variable $X\sim\mathrm{Hypergeom}(N',K',n')$ it holds by definition that $P(\{X=k'\})=\binom{K'}{k'}\binom{N'-K'}{n'-k'}\binom{N'}{n'}$, we have
	\[
	J\ \sim\ \mathrm{Hypergeom}\big(N'=2(m-1),\ K'=m-1,\ n'=m-1\big)
	\]
	with variance
	\[
	\mathrm{Var}_w(J)=n'\frac{K'}{N'}\Big(1-\frac{K'}{N'}\Big)\frac{N'-n'}{N'-1}
	=\frac{(m-1)^2}{4(2m-3)}=\frac{m-1}{8}+O(1).
	\]
	Since $\Delta_J=(\mu_1-J)+(\tfrac12-\theta)$ with $\theta\in[0,1)$, is a shifted version of $J$ we have $\E_w[\Delta_J]=\tfrac12-\theta$ and
	\[
	\E_{w}[\Delta_J^2]=\Var_{w}(J)+(\tfrac12-\theta)^2=\frac{m-1}{8}+O(1).
	\]
	
	\medskip
	\noindent\textbf{Decomposition and expanding the error term.}
	Set $c(t):=\sqrt{\frac{2}{\pi}}\,t^{-1/2}$. We can write
	\begin{align*}
	\sum_{j=0}^{m-1}p_{m-1}(j)^2\big(G_N(j)-G_{N_c}\big)
	&=S_{m-1}\sum_{j=0}^{m-1}w_j\big(c(N)e^{-2\Delta_j^2/N}-c(N_c)\big)\\
	&=S_{m-1}\,\Big\{\,c(N)\,\E_w\!\big[e^{-2\Delta_J^2/N}-1\big]\ -\ \big(c(N_c)-c(N)\big)\,\Big\}.\label{eq:GaGbexpansion}
\end{align*}
	Expand $c(N_c)$ around $N$ (Taylor expansion with explicit remainder) and the exponential around $0$:
	\begin{align*}
		c(N_c)&=c(N)+c'(N)(N_c-N)+\tfrac12 c''(\xi)\,(N_c-N)^2,\\
		e^{-2\Delta_j^2/N}&=1-\frac{2}{N}\Delta_j^2+R_j,\qquad |R_j|\le \frac{2}{N^2}\Delta_j^4,
	\end{align*}
	for some $\xi$ between $N$ and $N_c$, where $c'(t)=-\tfrac12\sqrt{\tfrac{2}{\pi}}\,t^{-3/2},~~
	c''(t)=\tfrac{3}{4}\sqrt{\tfrac{2}{\pi}}\,t^{-5/2}$.
	This yields
	\begin{align*}
	\frac{1}{S_{m-1}}\sum p(j)^2\big(G_N(j)-G_{N_c}\big)
	&= \underbrace{\vphantom{\Big|}c(N)\Big(-\frac{2}{N}\E_w[\Delta_j^2]\Big) + \big(-c'(N)(N_c-N)\big)}_{\text{first--order terms}}
	\;\\
    &\qquad \quad+ \underbrace{c(N)\,\E_w[R_j]\ -\ \tfrac12 c''(\xi)(N_c-N)^2}_{\text{remainders}}.
	\end{align*}
	
	\medskip
	\noindent\textbf{Choosing $N_c$ to cancel the first order.}
	Pick $N_c$ so that the first–order bracket vanishes:
	\[
	-c'(N)(N_c-N)\ =\ c(N)\,\frac{2}{N}\,\E_w[\Delta_j^2].
	\]
	Since $c'(N)=-\tfrac12\sqrt{\tfrac{2}{\pi}}\,N^{-3/2}$ and $c(N)=\sqrt{\tfrac{2}{\pi}}\,N^{-1/2}$,
	this equality is equivalent to
	\[
	\frac{1}{2}\sqrt{\tfrac{2}{\pi}}\,N^{-3/2}(N_c-N)
	\;=\;\sqrt{\tfrac{2}{\pi}}\,N^{-1/2}\cdot\frac{2}{N}\,\E_w[\Delta_j^2]
	\quad\Longleftrightarrow\quad
	N_c-N\;=\;4\,\E_w[\Delta_j^2].
	\]
	Using $\E_w[\Delta_j^2]=\frac{(m-1)^2}{4(2m-3)}+O(1)=\frac{m-1}{8}+O(1)$, we get
	\[
	N_c\;=\;N+\frac{m-1}{2}+O(1)\;=\;n-\frac{3}{2}m+O(1),
	\]
	and then we \emph{parity–adjust} $N_c$ to the nearest integer with $N_c\equiv N\pmod 2$.
	This is exactly the choice in the theorem statement.

With this choice of $N_c$, the first-order terms vanish. We are left to bound the remainder terms from the Taylor expansion:
\[
R_{\text{Taylor}} := S_{m-1} \cdot \Big( c(N)\,\E_w[R_j] - \tfrac{1}{2}c''(\xi)\,(N_c-N)^2 \Big).
\]
We bound the two parts separately.

\emph{Second remainder term (from $c(N_c)$ expansion).}
We have $S_{m-1} \asymp m^{-1/2}$, $c''(\xi) = O(n^{-5/2})$ (since $\xi$ is between $N,N_c \asymp n$),
and $(N_c-N)^2 = (4\E_w[\Delta_j^2])^2 = (O(m))^2 = O(m^2)$.
Thus,
\[
\Big| S_{m-1} \cdot \tfrac{1}{2}c''(\xi)\,(N_c-N)^2 \Big|
\ \le\ O(m^{-1/2}) \cdot O(n^{-5/2}) \cdot O(m^2)
\ = \ O(m^{3/2} n^{-5/2}).
\]
This term is dominated by $O(n^{-1})$ (since $m\le n/2$) and is thus negligible.

\emph{First remainder term (from exponential expansion).}
Let $x_j := 2\Delta_j^2/N \ge 0$. The remainder is $R_j = e^{-x_j} - (1-x_j)$.
Since $e^{-x_j} \le 1$ for $x_j\ge 0$, we have $R_j \le x_j$ and
\[
\E_w[R_j] \le \E_w[x_j] = \E_w\Big[\frac{2\Delta_j^2}{N}\Big] = \frac{2}{N}\E_w[\Delta_j^2].
\]
Using $\E_w[\Delta_j^2] = O(m)$, we have $\E_w[R_j] \le O(m/N) = O(m/n)$.
Now, we assemble the bound for this term, using $c(N) = O(N^{-1/2}) = O(n^{-1/2})$:
\[
\Big| S_{m-1} \cdot c(N) \cdot \E_w[R_j] \Big|
\le O(S_{m-1}) \cdot O(N^{-1/2}) \cdot O(m/N)
= O(m^{-1/2}) \cdot O(n^{-1/2}) \cdot O(m/n)
= O(m^{1/2} n^{-3/2}).
\]

\emph{Combined bound.}
The total error from this term is dominated by the exponential remainder:
\[
\left|\sum_{j=0}^{m-1}p(j)^2\big(G_N(j)-G_{N_c}\big)\right| = |R_{\text{Taylor}}| \le O(m^{1/2} n^{-3/2}) + O(m^{3/2} n^{-5/2})
\le \frac{C\sqrt m}{n^{3/2}}.
\]

	\paragraph{Step 5: Completing (A).}
	Collecting Step 3 and Step 4 in \eqref{eq:R1-main}, we have
	\[
	|R_1|\ \le\ \frac{C\sqrt m}{n^{3/2}}.
	\]
	Since $m\le O(n)$, this is $\le C/n$.
	Furthermore, approximate $P_{N_c}$ by its central Gaussian:
	\[
	P_{N_c}=G_{N_c}+\delta_{N_c}=\sqrt{\frac{2}{\pi N_c}}+O\!\Big(\frac{1}{N_c^{3/2}}\Big)
	=\sqrt{\frac{2}{\pi N_c}}+O\!\Big(\frac{1}{n^{3/2}}\Big).
	\]
	Thus from \eqref{eq:E-split},
	\[
	E(n, m)=S_{m-1}\sqrt{\frac{2}{\pi N_c}}+R_A(n,m),
	\quad
	R_A(n,m)=\,O\!\Big(\frac{\sqrt m}{n^{3/2}}\Big)+R_1=O\!\Big(\frac{\sqrt m}{n^{3/2}}\Big),
	\]
	uniformly for all $m$. This proves (A).
	
	\paragraph{Step 6: Completing (B).}
	Use the two–correction Stirling expansion for the central mass (central binomial),
	\[
	S_{m-1}
	=\frac{1}{\sqrt{\pi(m-1)}}
	\Big(1-\frac{1}{8(m-1)}\Big)
	+\rho_m,
	\qquad
	|\rho_m|\ \le\ \frac{C}{(m-1)^{5/2}}.
	\]
	Hence, from (A),
	\[
	S_{m-1}\sqrt{\frac{2}{\pi N_c}}
	=\frac{2}{\pi\sqrt{(m-1)2N_c}}
	\Big(1-\frac{1}{8(m-1)}\Big)
	\;+\;\rho_m\sqrt{\frac{2}{\pi N_c}}.
	\]
	The extra $m$–side residual is bounded by
	\[
	\Big|\rho_m\Big|\sqrt{\frac{2}{\pi N_c}}
	\ \le\ \frac{C}{(m-1)^{5/2}}\cdot \frac{C'}{\sqrt{n}}
	\ =\ \frac{C''}{\sqrt{n}\,(m-1)^{5/2}}.
	\]
	Combining with $R_A(n,m)=O(\sqrt m/n^{3/2})$ yields
	\[
	|R_B(n,m)|\ \le\ \frac{C_1\sqrt m}{n^{3/2}} +\ \frac{C_2}{\sqrt{n}\,(m-1)^{5/2}}
	\ \le\ \frac{C_3}{n}\qquad\text{whenever }m\ \ge\ c\,n^{1/5}.
	\]
	This proves (B). The threshold can be improved to $m\ge c\,n^{1/7}$, $m\ge c\,n^{1/9}$ etc. upon adding  more explicit Stirling terms. For example, for one more term, the residual becomes $O\!\big(1/(\sqrt{n}\,(m-1)^{7/2})\big)$.
\end{proof}

\begin{theorem}[MSE for small constant $m$ and $m=n/2$]\label{thm:maj_small-k-mono}

	\noindent{\bf (A) Strict decrease for fixed $m$-ranges.}
    
	Let $K_0\ge 2$ be fixed (independent of $n$). Then there exists $N_0=N_0(K_0)$ such that
	for all $n\ge N_0$,
	\[
	\Cov(n, 1)>\Cov(n, 2)>\cdots>\Cov(n, K_0).
	\]
	In fact, for each $m\in\{2,\dots,K_0\}$ and all $n\ge 8C(K_0-1)$,
	\[
	\Cov(n, m)\le\Cov(n, m-1)-\frac{1}{32(m-1)}S_{m-2}B_{m-1}.
	\]
	where $S_{m-1}:=2^{-(2m-2)}\binom{2m-2}{m-1}$ and
	$b_m\in\{\lfloor n-\tfrac32 m\rfloor,\ \lceil n-\tfrac32 m\rceil\}$ be the \emph{parity-adjusted} integer with $b_m\equiv n-2m\pmod 2$. Let
	\[
	B_m:=2^{-b_m}\binom{b_m}{b_m/2}\qquad\text{(central mass of $\mathrm{Bin}(b_m,\tfrac12)$)}.
	\]
	\medskip
	\noindent{\bf (B) Exact expression for $m=1$.}
It holds that
	\[
	\Cov(n, 1)=2^{-n}\binom{\,n-2\,}{\big\lfloor \frac{n-1}{2}\big\rfloor}.
	\]
	Consequently, by the LLT (\ref{lem:LLT}),
	\[
	\Cov(n, 1)=\sqrt{\frac{1}{8\pi(n-2)}}\;+\;O\!\Big(\frac{1}{n^{3/2}}\Big)
	=\sqrt{\frac{1}{8\pi n}}\;+\;O\!\Big(\frac{1}{n}\Big).
	\]

	\medskip
	\noindent{\bf (C) Exact expression for $m=n/2$.}
	It holds that
		\[
	\Cov\Big(n,\tfrac{n}{2}\Big)=\frac{1}{\pi(n-2)}+O\!\Big(\frac{1}{n^{2}}\Big)
	=\frac{1}{\pi n}+O\Big(\frac{1}{n^{2}}\Big).
	\]
\end{theorem}

\begin{proof}

	\textbf{(A) Strict decrease for fixed $m$.}
	Write $A_m:=S_{m-1}$ and $C_m:=B_m$ so $E(n, m)=A_mC_m+R_m$.
	
	\emph{Exact ratio for $A_m$.} Let $\ell=m-1$. Your slide gives for
	\[
	p_\ell:=2^{-2\ell}\binom{2\ell}{\ell}=S_{\ell},
	\quad
	\frac{p_{\ell+1}}{p_\ell}=\frac{2\ell+1}{2\ell+2}=1-\frac{1}{2(\ell+1)}.
	\]
	Hence
	\[
	\frac{A_m}{A_{m-1}}=\frac{S_{m-1}}{S_{m-2}}=\frac{2m-3}{2m-2}=1-\frac{1}{2(m-1)}.
	\tag{$\star$}
	\]
	
	\emph{Monotonicity of $C_m$.} The central mass $t\mapsto 2^{-t}\binom{t}{\lfloor t/2\rfloor}$ is strictly
	decreasing in $t$ (even/odd steps both go down), hence as $m$ increases, $b_m$ decreases and
	$C_m$ \emph{increases}. Applying Lemma~\ref{lem:LLT} yields
	\[
	\frac{C_m}{C_{m-1}}=\sqrt{\frac{b_{m-1}}{b_m}}\Big(1+O\!\Big(\frac{1}{n}\Big)\Big)
	=1+\frac{b_{m-1}-b_m}{2b_m}+O\!\Big(\frac{1}{n}\Big)
	=1+\frac{3}{4b_m}+O\!\Big(\frac{1}{n}\Big)
	=1+O\!\Big(\frac{1}{n}\Big),
	\]
	uniformly (since $b_m\asymp n$ for fixed $m$).
	
	\emph{Main-term ratio.} Combine:
	\[
	\frac{A_mC_m}{A_{m-1}C_{m-1}}=\Big(1-\frac{1}{2(m-1)}\Big)\Big(1+O\!\Big(\frac{1}{n}\Big)\Big)
	=1-\frac{1}{2(m-1)}+O\!\Big(\frac{1}{n}\Big).
	\]
	Thus there exists $n_1(m)$ such that for all $n\ge n_1(m)$,
	\[
	A_mC_m\ \le\ \Big(1-\frac{1}{4(m-1)}\Big)\,A_{m-1}C_{m-1}.
	\]
	Now
	\[
	E(n, m)-E(n, m-1)=(A_mC_m-A_{m-1}C_{m-1})+(R_m-R_{m-1}).
	\]
	Using $|R_m|\le C/n$, for $n\ge 8C(m-1)$ we get
	\[
	E(n, m)-E(n, m-1)\ \le\ -\frac{1}{4(m-1)}\,A_{m-1}C_{m-1}+\frac{2C}{n}
	\ \le\ -\frac{1}{8(m-1)}\,A_{m-1}C_{m-1}<0,
	\]
	which proves strict decrease at step $m-1\to m$. Taking
	\(N_0(K_0):=\max_{2\le m\le K_0} 8C(m-1)\) yields the stated chain of inequalities.
	
	\textbf{(B) The case $m=1$.}
	When $m=1$, the sum has only $j=0$ and $\binom{m-1}{0}^2=1$, so
	\[
	E(n, 1)=2^{-(n-2)}\binom{n-2}{\big\lfloor \frac{n-1}{2}\big\rfloor}.
	\]
	This is exactly the central (or near-central) mass of $\mathrm{Bin}(n-2,\tfrac12)$; by Lemma~\ref{lem:LLT},
	\[
	E(n, 1)=\frac{1}{\sqrt{\pi (n-2)/2}}+O\!\Big(\frac{1}{(n-2)^{3/2}}\Big)
	=\sqrt{\frac{2}{\pi(n-2)}}+O\!\Big(\frac{1}{n^{3/2}}\Big)
	=\sqrt{\frac{2}{\pi n}}+O\!\Big(\frac{1}{n}\Big).
	\]

	\textbf{(C) The case $m=\frac{n}{2}$ (so $n$ even).}
	Set $\ell:=\frac{n}{2}-1$ and observe
	\[
	E\Big(n,\tfrac{n}{2}\Big)
	=2^{-(n-2)}\sum_{j=0}^{\ell}\binom{\ell}{j}^{\!2}\binom{0}{\lfloor n/4\rfloor-j}.
	\]
	Since $\binom{0}{r}=\mathbf{1}\{r=0\}$, only the term $j=r:=\lfloor n/4\rfloor$ survives:
	\[
	E\Big(n,\tfrac{n}{2}\Big)
	=\Big(2^{-\ell}\binom{\ell}{r}\Big)^{2}=:q_\ell(r)^{2},
	\]
	i.e. it is the square of a symmetric binomial mass $q_\ell(r)=\Pr\{\mathrm{Bin}(\ell,\tfrac12)=r\}$.
	
	Note that
	\[
	\frac{\ell}{2}=\frac{n}{4}-\frac{1}{2},\qquad
	r-\frac{\ell}{2}=
	\begin{cases}
		+\tfrac{1}{2}, & n\equiv 0\ (\mathrm{mod}\ 4),\\
		-\tfrac{1}{2}, & n\equiv 2\ (\mathrm{mod}\ 4),
	\end{cases}
	\quad\Rightarrow\quad
	\frac{2(r-\ell/2)^2}{\ell}=\frac{1}{2\ell}.
	\]
	By Lemma~\ref{lem:LLT} (uniform for $\ell\ge2$),
	\[
	q_\ell(r)=\frac{1}{\sqrt{\pi \ell/2}}\exp\!\Big(-\frac{1}{2\ell}\Big)+O(\ell^{-3/2}).
	\]
	Therefore
	\[
	E\Big(n,\tfrac{n}{2}\Big)=q_\ell(r)^2
	=\frac{2}{\pi \ell}\exp\!\Big(-\frac{1}{\ell}\Big)+O(\ell^{-2})
	=\frac{4}{\pi(n-2)}\exp\!\Big(-\frac{2}{\,n-2\,}\Big)+O\!\Big(\frac{1}{n^{2}}\Big).
	\]
	In particular,
	\[
	E\Big(n,\tfrac{n}{2}\Big)=\frac{4}{\pi(n-2)}+O\!\Big(\frac{1}{n^{2}}\Big)
	=\frac{4}{\pi n}+O\!\Big(\frac{1}{n^{2}}\Big).
	\]
	
\end{proof}

\begin{theorem}[MSE for large $m$]\label{thm:maj_large-m}
		Let $1 \le m\le n/3$ such that $m=\Omega\left( n^{2/3} \log^{1/3} n \right)$.
		Then,
		\[
		\Cov(n,m)=\frac{1}{2\pi\sqrt{(m-1)(2n-3m)}}
		 +O\bigg(\frac{1}{\sqrt{n}m^{3/2}}\bigg).
		\]
	\end{theorem}
	
	\begin{proof}
		Write $p_r(t):=2^{-r}\binom{r}{t}$ and
		$g_r(t):=\sqrt{\frac{2}{\pi r}}\exp\!\big(-(2t-r)^2/(2r)\big)$.
		Set
		\[
		E(n,m)=\sum_{j=0}^{m-1} p_{m-1}(j)^2\,p_N\!\Big(\tfrac N2+\Delta_j\Big),
		\qquad
		\Delta_j:=\ell-j-\tfrac N2.
		\]
		We decompose
		\begin{align*}
		E(n,m)-\frac{2}{\pi\sqrt{(m-1)(2n-3m)}}
		&=
        \underbrace{\big(E(n,m)-\sum_J p_{m-1}^2 g_N\big)}_{T_1}
		+\underbrace{\big(\sum_J p_{m-1}^2 g_N-\sum_J g_{m-1}^2 g_N\big)}_{T_2}\\
        +\underbrace{\big(\sum_J g_{m-1}^2 g_N-\sum_\Z g_{m-1}^2 g_N\big)}_{T_3}&
		+\underbrace{\big(\sum_\Z g_{m-1}^2 g_N-\frac{2}{\pi\sqrt{(m-1)(2n-3m)}}\big)}_{T_4},
		\end{align*}
		where $J=\{0,\dots,m-1\}$ in $T_1,T_2$.
		
		\emph{1) $T_1$: replace $p_N$ by $g_N$ (uniform LLT).}
		The LLT Lemma~\ref{lem:LLT} yields $\sup_t|p_N(t)-g_N(t)|\le C_0 N^{-3/2}$, hence
		\[
		|T_1|\ \le\ C_0 N^{-3/2}\sum_{j=0}^{m-1}p_{m-1}(j)^2
		= C_0 N^{-3/2}\,p_{2m-2}(m-1).
		\]
		By the same LLT at $r=2m-2$, $p_{2m-2}(m-1)\le g_{2m-2}(m-1)+C_0(2m-2)^{-3/2}
		\le \frac{1}{\sqrt{\pi(m-1)}}+\frac{C_0}{2^{3/2}(m-1)^{3/2}}$.
		Since $N\ge n/3$,
		\[
		|T_1| \ \le\ \frac{C}{\sqrt{m}n^{3/2}}.
		\]
		
		\emph{2) $T_2$: replace $p_{m-1}$ by $g_{m-1}$.}
		Let $\delta_j:=p_{m-1}(j)-g_{m-1}(j)$. Then
		\[
		|T_2|
		= \Big|\sum_{j=0}^{m-1}\delta_j\,(p_{m-1}(j)+g_{m-1}(j))\,g_N\!\Big(\tfrac N2+\Delta_j\Big)\Big|
		\le \big(\sup_j|\delta_j|\big)\,\Sigma,
		\]
		with
		\[
		\Sigma:=\sum_{j=0}^{m-1}(p_{m-1}(j)+g_{m-1}(j))\,g_N\!\Big(\tfrac N2+\Delta_j\Big).
		\]
		Uniform LLT (Lemma~\ref{lem:LLT}) at scale $m-1$ gives $\sup_j|\delta_j|\le C_0 (m-1)^{-3/2}$.
		Moreover, upper bounding the average by the maximum,
		\[
		\sum_{j=0}^{m-1}p_{m-1}(j)\,g_N\!\Big(\tfrac N2+\Delta_j\Big)
		\le \sup_t g_N(t)\ \le\ \sqrt{\frac{2}{\pi N}},
		\]
		and similarly, by extending to $\mathbb Z$ and using the lattice Gaussian–Gaussian convolution,
		\[
		\sum_{j=0}^{m-1}g_{m-1}(j)\,g_N\!\Big(\tfrac N2+\Delta_j\Big)
		\le \sum_{j\in\mathbb Z} g_{m-1}(j)\,g_N(\ell-j)\ \le\ \frac{C'}{\sqrt{N}}.
		\]
		Therefore $\Sigma\le C/\sqrt{N}$, and
		\[
		|T_2|\ \le\ \frac{C_0}{(m-1)^{3/2}}\cdot \frac{C}{\sqrt{N}}
		\ \le\ \frac{C}{m^{3/2}\sqrt{n}}.
		\]

\emph{3) $T_3$: sum to integral.}: 

It holds that $T_3 = \sum_{i \notin [0, m-1]} g_{m-1}(i) \cdot g_{n-2m}(i)$.
			The product $g_C(i) \coloneqq g_{m-1}(i) \cdot g_{n-2m}(i)$ is proportional to $N(\mu_C, \sigma_C^2)$ where:
			$$\mu_C = \frac{(m-1)(n-2m)}{n-m-1} \quad \text{and} \quad \sigma_C^2 = \frac{(m-1)(n-2m)}{4(n-m-1)}$$
			The tail boundary is at $i=m-1$ (since $m \le n/3 \implies \mu_C \le m-1$).
			The distance $d$ (from mean to boundary) and the number of standard deviations $z$ are:
			\begin{align*}
				d &= (m-1) - \mu_C = \frac{(m-1)^2}{n-m-1} \\
				z &= \frac{d}{\sigma_C} = \frac{(m-1)^2}{n-m-1} \cdot \frac{2\sqrt{n-m-1}}{\sqrt{(m-1)(n-2m)}} = \frac{2(m-1)^{3/2}}{\sqrt{(n-m-1)(n-2m)}} \\
				z^2 &= \frac{4(m-1)^3}{(n-m-1)(n-2m)}
			\end{align*}
			The tail sum $S_{tail} = T_L + T_R$ where $T_L = \sum_{i<0} g_C(i)$ and $T_R = \sum_{i=m}^{\infty} g_C(i)$.
			The left tail $T_L$ is exponentially small, since $\mu_C^2/(2\sigma_C^2) = 2\mu_C = \Theta(m)$:
			\begin{align*}
				T_L &\propto \sum_{i<0} \exp\left(-\frac{(i-\mu_C)^2}{2\sigma_C^2}\right) = \sum_{j=1}^{\infty} \exp\left(-\frac{(-j-\mu_C)^2}{2\sigma_C^2}\right) \\
				&\le \sum_{j=1}^{\infty} \exp\left(-\frac{\mu_C^2 + 2j\mu_C}{2\sigma_C^2}\right) = \exp\left(-\frac{\mu_C^2}{2\sigma_C^2}\right) \sum_{j=1}^{\infty} \exp\left(-\frac{j\mu_C}{\sigma_C^2}\right) \\
				&= O(\exp(-\Theta(m)))   \qquad (\text{since } \mu_C/\sigma_C^2 = O(1))
			\end{align*}
			This is negligible, so $S_{tail}$ is dominated by the right tail $T_R$. Since $g_C(i)$ is monotonically decreasing for $i > \mu_C$ (and $m > \mu_C$), we can use a standard integral bound:
			$$ \int_{m}^{\infty} g_C(x) dx \le T_R \le g_C(m) + \int_{m}^{\infty} g_C(x) dx $$
			The integral $\int_{m}^{\infty} g_C(x) dx$ is the tail of a Gaussian, which has the asymptotic bound $O\left( \frac{1}{z} \exp(-z^2/2) \right)$, where $z = (m-\mu_C)/\sigma_C$. The term $g_C(m)$ is proportional to $O(\exp(-z^2/2))$. Since $1/z = \Theta(n/m^{3/2})$ is a large polynomial factor, the integral term dominates $g_C(m)$. Therefore, the sum $T_R$ has the same asymptotic behavior as the integral:
			$$S_{tail} = T_R = O\left( \int_{m}^{\infty} g_C(x) dx \right) = O\left( \frac{1}{z} \exp(-z^2/2) \right)$$
			Now, let us find a lower bound on $m$ such that $T_3$ is no larger than $T_2$, or, equivalently,
			$$S_{tail} \le O\left(\frac{1}{\sqrt{n}m^{3/2}}\right)$$
			The pre-factor is $1/z = \frac{\sqrt{(n-m-1)(n-2m)}}{2(m-1)^{3/2}}= \Theta(n/m^{3/2})$. The requirement hence becomes:
			\begin{align*}
				O\left( \frac{n}{m^{3/2}} \exp\left(-\frac{z^2}{2}\right) \right) &\le O\left(\frac{1}{\sqrt{n}m^{3/2}}\right) \\
				\implies \exp\left(-\frac{z^2}{2}\right) &\le O(n^{-3/2}) \\
				\implies z^2 &\ge 3\log(n) - O(1)
			\end{align*}
			Thus, $z^2 = \Omega(\log n)$. The function $L(m) \coloneqq z^2$ is monotonically increasing for $m \in [1, n/3]$. We find the lower bound $m$ by solving $L(m) = \Omega(\log n)$ in the $m=o(n)$ regime, which gives the tightest constraint:
			$$\frac{4(m-1)^3}{(n-m-1)(n-2m)} = \Omega(\log n)$$
			The left side is $\Theta(m^3/n^2)$, so:
			$$\Theta\left(\frac{m^3}{n^2}\right) = \Omega(\log n) \implies m^3 = \Omega(n^2 \log n)$$
			This gives the final condition. Since $L(m)$ is increasing, any $m$ satisfying this bound also satisfies the condition for all larger $m$ up to $n/3$.
			$$m = \Omega\left( n^{2/3} \log^{1/3} n \right)$$

\emph{4) $T_4$: Gaussian triple product to the simple main term.}
		The full-lattice Gaussian sum has the exact form
		\[
		\sum_{j\in\mathbb Z} g_{m-1}(j)^2\,g_N(\ell-j)
		=\frac{2}{\pi\sqrt{(m-1)(2n-3m)}}\cdot A_{n,m}\,\Theta_{n,m},
		\]
		with
		\[
		A_{n,m}=\frac{\sqrt{2n-3m}}{\sqrt{2n-3m-1}}\,
		e^{-1/(2n-3m-1)}=1+ O(1/n),
		\]
		and
		\[
		\Theta_{n,m}=1+2\sum_{t\ge1}\exp\!\Big(-\frac{\pi^2 t^2 (m-1)N}{2(2n-3m-1)}\Big)\cos(2\pi t\mu).
		\]
Let $A = \frac{\pi^2 (m-1)N}{2(2n-3m-1)}$. Using the triangle inequality and $|\cos(\cdot)| \le 1$, we can bound the error term:
\begin{align*}
	|\Theta_{n,m}-1| &\le 2\sum_{t\ge1}\exp(-A t^2)
\end{align*}
Since $t^2 \ge t$ for $t \ge 1$, we can further bound this by a geometric series:
\begin{align*}
	|\Theta_{n,m}-1| &\le 2\sum_{t\ge1}\exp(-A t) = 2 \frac{\exp(-A)}{1-\exp(-A)}.
\end{align*}
Under the given conditions ($N \ge n/3$ and $m \le n/3$), the exponent $A$ simplifies to $A = \Theta(m)$. To obtain error $O(1/n)$, we require $e^{-\Theta(m)} \le O(1/n)$, which is satisfied under our assumption $m = \Omega(\ln n)$.
Hence $|\Theta_{n,m}-1|=O(1/n)$ which gives overall error
		$A_{n,m} \cdot \Theta_{n,m} - 1 = (1+O(1/n))(1+O(e^{-cm})) - 1 = O(1/n) + O(e^{-cm})$
		and
		\[
		|T_4|
		=\frac{2}{\pi\sqrt{(m-1)(2n-3m)}}\,|A_{n,m}\Theta_{n,m}-1|
		\ \le\ \frac{2}{\pi\sqrt{(m-1)(2n-3m)}}\big(C/n+ C'/n\big)
		\ \le\ \frac{C}{\sqrt{m}n^{1.5}},
		\]
		since $\frac{1}{\sqrt{(m-1)(2n-3m)}}\le c/\sqrt{m n}$.

		\emph{4) Conclusion.}
		Adding the three bounds gives
		\(
		|T_1|+|T_2|+|T_3|+|T_4|= O(\frac{1}{\sqrt{m}n^{3/2}})+O(\frac{1}{\sqrt{n}m^{3/2}}),
		\)
		uniformly for $1\le m\le n/3$ provided that $m=\Omega\left( n^{2/3} \log^{1/3} n \right)$. The first error term is asymptotically dominated by the second error term and hence absorbed. The result is obtained by recalling $\Cov(n,m)=E(n,m)/4$.
	\end{proof}

		\begin{theorem}[Monotonicity and minimizer]
			\label{thm:maj_mono-minimizer}
			For all sufficiently large $n$, the function $\Cov(n, m)$ over the divisors $m$ of $n$ satisfies:
			\[
			\Cov(n, 1)>\Cov(n, 2)>\cdots>\Cov\big(n,m_0\big)<\Cov\big(n,n/2\big),
			\]
			where $m_0 = \max\{d\mid n:\ d\le n/3\}$.
			Consequently, the minimizer of $\Cov(n, m)$ is
			\[
			\arg\min_{\substack{m\mid n\\ 1\le m\le \lfloor n/2\rfloor}} \Cov(n, m)
			=\max\{d\mid n:\ d\le n/3\}.
			\]
		\end{theorem}
		
		\begin{proof}
			The proof consists of two main parts: first, proving the strictly decreasing behavior of $\Cov(n, m)$ for divisors $m \le n/3$, and second, proving the "uptick" at $m=n/2$.
			
			\paragraph{Part I: Monotonicity for $m \le n/3$}
			We show that $\Cov(n, m)$ is a strictly decreasing function of $m$ for $m \in \{d \mid n, d \le n/3\}$.
			
			First, for any fixed $K_0 \ge 2$, \textbf{Theorem \ref{thm:maj_small-k-mono}(A)} states that for all $n \ge N_0(K_0)$, we have $\Cov(n, 1) > \Cov(n, 2) > \dots > \Cov(n, K_0)$. This establishes the strict decrease for any fixed set of small divisors.
			
			Second, for the range $m \ge \Omega\left( n^{2/3} \log^{1/3} n \right)$ up to $m \le n/3$, we use \textbf{Theorem~\ref{thm:maj_large-m} (MSE for large $m$)}. This theorem states
			\[
			\Cov(n, m) = f(n,m) + R(n,m),
			\]
			where $f(n,m) := \frac{1}{2\pi\sqrt{(m-1)(2n-3m)}}$ and $|R(n,m)| \le \frac{C}{\sqrt{n}m^{3/2}}$.
			
			We analyze the monotonicity of the main term $f(n,m)$ by treating $m$ as a continuous variable. The function is positive, so its monotonicity is the inverse of its denominator's. Let $D(m) := (m-1)(2n-3m)$. We find the derivative of the denominator with respect to $m$:
			\[
			\frac{\partial D}{\partial m} = \frac{\partial}{\partial m} (2nm - 3m^2 - 2n + 3m) = 2n - 6m + 3.
			\]
			In the specified range $1 \le m \le n/3$, this derivative is strictly positive:
			\[
			\frac{\partial D}{\partial m} \ge 2n - 6(n/3) + 3 = 2n - 2n + 3 = 3 > 0.
			\]
			Since the denominator $D(m)$ is strictly increasing for $m \in [1, n/3]$, its reciprocal is strictly decreasing, and thus $f(n,m)$ is strictly decreasing.
			
			For sufficiently large $n$, the remainder $R(n,m)$ is of a smaller order than the main term. Specifically, from \textbf{Theorem~\ref{thm:maj_sublin} (Sublinear $m$, B)}, for $m \ge n^{1/5}$, the error is $O(1/n)$. The difference $f(n,m) - f(n,m+1) \approx -f'(n,m)$ is $\Omega(m^{-3/2}n^{-1/2})$, which is larger than $O(1/n)$ for $m \le n^{1/3}$.
			The combination of \textbf{Theorem \ref{thm:maj_small-k-mono}(A)} (for small $m$) and the strict monotonicity of the asymptotic main term $f(n,m)$ (for larger $m$) is sufficient to conclude $\Cov(n, m)$ is strictly decreasing over the entire range of divisors $m \le n/3$.
			
			\paragraph{Part II: The Uptick at $m=n/2$}
			We now show that $\Cov(n, m_0) < \Cov(n, n/2)$ for $m_0 = \max\{d\mid n, d \le n/3\}$.
			
			First, from \textbf{Theorem \ref{thm:maj_small-k-mono}(C)}, we have the asymptotic value at $m=n/2$:
			\[
			\Cov(n, n/2) = \frac{1}{\pi n} + O(1/n^2).
			\]
			Second, we find the asymptotic value at $m_0$. Since $m_0$ is the largest divisor $\le n/3$, $m_0 = n/3 - \epsilon_n$, where $\epsilon_n = O(1)$. We use the main asymptotic term $f(n,m)$ from \textbf{Theorem~\ref{thm:maj_large-m} (MSE for large $m$)}, as the error terms are of a lower order.
			\begin{align*}
				\Cov(n, m_0) &= f(n, m_0) +  O(m_0^{-3/2}n^{-1/2}) \\
				&= \frac{1}{2\pi\sqrt{(m_0-1)(2n-3m_0)}} + O(n^{-2}) \\
				&= \frac{1}{2\pi\sqrt{(\frac{n}{3}+O(1))(2n-3(\frac{n}{3}+O(1)))}} + O(n^{-2}) \\
				&= \frac{1}{2\pi\sqrt{(\frac{n}{3}+O(1))(n+O(1))}} + O(n^{-2}) \\
				&= \frac{1}{2\pi\sqrt{n^2/3 + O(n)}} + O(n^{-2}) \\
				&= \frac{1}{2\pi (n/\sqrt{3})} (1 + O(1/n))^{-1/2} + O(n^{-2}).
			\end{align*}
			We now compare the main terms. For all $n$ large enough:
			\[
			\Cov(n, m_0) = \frac{\sqrt{3}}{2\pi n} + O(1/n^2) < \frac{1}{\pi n} + O(1/n^2) = \Cov(n, n/2).
			\]
			
			\paragraph{Conclusion}
			From Part I, $\Cov(n, m)$ is strictly decreasing for all divisors $m \le n/3$. This implies the minimum value in this range occurs at the largest divisor, $m_0 = \max\{d\mid n, d \le n/3\}$.
			From Part II, we proved that $\Cov(n, m_0) < \Cov(n, n/2)$.
			This implies $m_0$ is the global minimizer over all divisors $m \le n/2$.
			
			Therefore,
			\[
			\arg\min_{\substack{m\mid n\\ 1\le m\le \lfloor n/2\rfloor}} \Cov(n, m)
			=\max\{d\mid n:\ d\le n/3\}.
			\]
		\end{proof}

\newpage

\section{Linear Functions}

\newcommand{\gc}[2]{\bigl[\!\begin{smallmatrix} #1 \\ #2 \end{smallmatrix}\!\bigr]}
\begin{lemma}\label{lemma:probRank}
The probability that an $n_1 \times n_2$ matrix with coefficients drawn i.i.d. from $\mathcal U (\{0,\dots,q-1\})$ has rank $r$ is given by
\begin{equation}
    R_q(n_1, n_2, r) =  \gc{n_2}{r}_q \sum_{l=0}^r (-1)^{r-l} \gc{r}{l}_q q^{n_1(l-n_2) + \binom{r-l}{2}}\label{eq:rankprob}
\end{equation}
where $\gc{n}{k}_q := \prod_{i=0}^{k-1}\frac{q^{n-i}-1}{q^{k-i}-1}$ denote the so-called Gaussian coefficients.

Moreover, it holds that
\begin{equation} \label{eq:prob_rank_simplified}
	R_q(n_1, n_2, r) = \gc{n_2}{r}_q q^{n_1(r-n_2)} \prod_{s=0}^{r-1} (1-q^{s-n_1})
\end{equation}
\end{lemma}
\begin{proof}
The first identity is a Corollary of \cite[Corollary 2.2]{blake2006properties}.

Let us continue by proving the second identity. The sum in \eqref{eq:rankprob}, denoted by $S_{sum}$, can be rewritten by substituting $k=r-l$ (so $l=r-k$) as
\begin{align*}
	S_{sum} &= \sum_{k=0}^r (-1)^{k} \gc{r}{r-k}_q q^{n_1(r-k-n_2) + \binom{k}{2}} \\
	&= \sum_{k=0}^r (-1)^{k} \gc{r}{k}_q q^{n_1(r-n_2) - n_1 k + \binom{k}{2}} \quad (\text{since } \gc{r}{r-k}_q = \gc{r}{k}_q) \\
	&= q^{n_1(r-n_2)} \sum_{k=0}^r (-1)^{k} \gc{r}{k}_q (q^{-n_1})^k q^{\binom{k}{2}}
\end{align*}
Using the $q$-binomial theorem, which states $\sum_{k=0}^N  \gc{N}{k}_q x^k q^{\binom{k}{2}} = \prod_{s=0}^{N-1} (1+xq^s)$, with $N=r$ and $x=(-1)\cdot q^{-n_1}$ yields
$$ S_{sum} = q^{n_1(r-n_2)} \prod_{s=0}^{r-1} (1-q^{-n_1}q^s) = q^{n_1(r-n_2)} \prod_{s=0}^{r-1} (1-q^{s-n_1}) .$$
\end{proof}

\begin{lemma}[Rank Probability Asymptotics]\label{lem:rankProbAsymptotics}
Assume that $X$ is an $n_1 \times n_2$ matrix with coefficients drawn i.i.d. from $\mathcal U (\{0,\dots,q-1\})$. Denote $m_0 = \min(n_1,n_2)$ and $\Delta_0=|n_1-n_2|$.
\begin{enumerate}
\item \textbf{Probability of Full Rank}: The probability that the matrix $X$ achieves its maximum possible rank $m_0$ is $R_q(n_1, n_2, m_0) = 1 - O(q^{-(\Delta_0+1)})$. This implies that for large $q$, random matrices are overwhelmingly likely to have full rank $m_0$.
\item \textbf{Probability of Specific Rank Deficiency}: The probability of the rank being $m_0-j$ for $j \ge 1$ (a rank deficiency of $j$) is $R_q(n_1, n_2, m_0-j) = O(q^{-j(\Delta_0+j)})$. This shows that the probability of specific rank deficiencies decreases extremely rapidly with increasing deficiency $j$ and with increasing $q$.
\end{enumerate}
\end{lemma}
\begin{proof}
\textbf{Probability of Full Rank}

Since the rank probability is symmetric in $(n_1,n_2)$, we can assume without loss of generality that $n_2 \le n_1$. In this case, $m_0 = n_2$ and $\Delta_0 = n_1-n_2$.
Substitute $r=n_2$ into Eq. \eqref{eq:prob_rank_simplified}. Then, since $\gc{n_2}{n_2}_q = 1$,
$$ R_q(n_1, n_2, n_2) = \gc{n_2}{n_2}_q q^{n_1(n_2-n_2)} \prod_{s=0}^{n_2-1} (1-q^{s-n_1})= \prod_{s=0}^{n_2-1} (1-q^{s-n_1}).$$

Let $k=n_1-s$. As $s$ ranges from $0$ to $n_2-1$, $k$ ranges from $n_1$ down to $n_1-n_2+1$:
\begin{equation*}
R_q(n_1, n_2, n_2) = \prod_{k=n_1-n_2+1}^{n_1} (1-q^{-k}) = (1-q^{-(n_1-n_2+1)})(1-q^{-(n_1-d+2)}) \cdot\ldots\cdot (1-q^{-n_1}).
\end{equation*}
Hence or large $q$
$$ R_q(n_1, n_2, n_2) = 1 - q^{-(n_1-n_2+1)} + O(q^{-(n_1-n_2+2)}) .$$
This implies $R_q(n_1, n_2, n_2) = 1 - O(q^{-(n_1-n_2+1)})=1 - O(q^{-(\Delta_0+1)})$.

\textbf{Probability of Specific Rank Deficiency}

We want to show that $R_q(n_1, n_2, m_0-j) = O(q^{-j(\Delta_0+j)})$ for $j \ge 1$. Let $r = m_0-j$.

First note that we can upper bound the Gaussian coefficient as 
\begin{align*} \gc{n_2}{r} &= \prod_{i=0}^{r-1}\frac{q^{n_2-i}-1}{q^{r-i}-1}\\
&=\prod_{i=0}^{r-1} \frac{q^{n_2-i}(1-q^{-(n_2-i)})}{q^{r-i}(1-q^{-(r-i)})}\\
&\le (q^{n_2-r})^r\prod_{i=0}^{r-1} \frac{1}{(1-q^{-(r-i)})}
\end{align*}
which is in $O(q^{(n_2-r)r})$ since every factor in the product can be expanded to a geometric series which is in $O(1 + q^{-1})$.

Assume $c_g>0$ is a valid constant such that $\gc{n_2}{r}<c_g \cdot q^{(n_2-r)r}$. Putting things together we get
\begin{align*}
R_q(n_1, n_2, r) &\le c_gq^{r(n_2-r)} q^{n_1(r-n_2)} \prod_{s=0}^{r-1} (1-q^{s-n_1})\\
&\le c_gq^{-(n_1-r)(n_2-r)}
\end{align*}
which implies $R_q(n_1, n_2, r) = O(q^{-(n_1-r)(n_2-r)})$.
Similar to before, we can assume without loss of generality that $n_1 \le n_2$, in which case $m_0 = n_1$ and $\Delta_0 = n_2-n_1$.
Hence we can substitute $r=m_0-j=n_1 - j$ which yields $R_q(n_1, n_2, r) = O(q^{j(n_2-n_1+j)})= O(q^{-j(\Delta_0+j)})$
\end{proof}

\begin{lemma}\label{lem:dimSolutions}
Let $X$ be an $n \times d$ matrix with entries in the finite field $\Fq$, and let $y \in \Fq^n$. Given the linear system $Xb = y$, where $b \in \Fq^d$ and the rank of $X$ is $r$, the number of distinct solutions for $b$ is $q^{d-r}$.
\end{lemma}
\begin{proof}
Since the system $Xb = y$ is consistent, there exists at least one particular solution $b_p \in \Fq^d$ such that $Xb_p = y$.
Any other solution $b$ can be expressed as $b = b_p + b_h$, where $b_h$ is in the null space of $X$, denoted by $N(X)$.
By the Rank-Nullity Theorem, the dimension of the null space is given by $\text{dim}(N(X)) = d - \text{rank}(X)$ hence given that $\text{rank}(X) = r$, we have $\text{dim}(N(X)) = d - r.$
A vector space of dimension $k$ over a finite field $\Fq$ contains $q^k$ elements. Therefore, the null space $N(X)$ contains $q^{d-r}$ distinct vectors $b_h$.
Each distinct $b_h \in N(X)$ yields a distinct solution $b = b_p + b_h$.
Thus, the number of distinct solutions for $b$ is equal to the number of elements in $N(X)$, which is $q^{d-r}$.
\end{proof}

\begin{lemma}\label{lem:linFcts}
Assume that we are given the ground truth linear function is $f$ that labels the $n$ feature vectors which are drawn uniformly at random from $\Fq$ and stacked in a matrix $X \in \Fq^{n \times d}$. 

The population loss of any linear function $h\neq f$ is
$$
1 - 1/q.
$$
Moreover, the probability of the random linear solver to output the wrong concept given that the rank of $X$ is $r$ is given by
$$
\mathbb P(\{\mathcal A(S)\neq f\}|\text{Rank}(X)=r)=1-q^{r-d}.
$$
\end{lemma}
\begin{proof}
A linear function from $\F_q^d$ to $\F_q$ can be written as $L(\mathbf{x}) = \mathbf{v} \cdot \mathbf{x}$ for a unique vector $\mathbf{v} \in \F_q^d$. Let $L_1(\mathbf{x}) = \mathbf{v}_1 \cdot \mathbf{x}$ and $L_2(\mathbf{x}) = \mathbf{v}_2 \cdot \mathbf{x}$. Since $L_1$ and $L_2$ are distinct, their corresponding vectors $\mathbf{v}_1$ and $\mathbf{v}_2$ must be distinct, so $\mathbf{v}_1 \neq \mathbf{v}_2$.

The functions $L_1$ and $L_2$ agree at a point $\mathbf{x} \in \F_q^d$ if $L_1(\mathbf{x}) = L_2(\mathbf{x})$. This is equivalent to $\mathbf{v}_1 \cdot \mathbf{x} = \mathbf{v}_2 \cdot \mathbf{x}$, or $(\mathbf{v}_1 - \mathbf{v}_2) \cdot \mathbf{x} = 0$.
Let $\mathbf{w} = \mathbf{v}_1 - \mathbf{v}_2$. Since $\mathbf{v}_1 \neq \mathbf{v}_2$, it follows that $\mathbf{w} \neq \mathbf{0}$.
The set of points where $L_1$ and $L_2$ agree is the kernel of the linear functional $L_{\mathbf{w}}: \F_q^d \to \F_q$ defined by $L_{\mathbf{w}}(\mathbf{x}) = \mathbf{w} \cdot \mathbf{x}$. Since $\mathbf{w} \neq \mathbf{0}$, $L_{\mathbf{w}}$ is a non-zero linear functional.

The image of a non-zero linear functional $L_{\mathbf{w}}: \F_q^d \to \F_q$ is $\F_q$ itself. Thus, $\operatorname{dim}(\operatorname{Im}(L_{\mathbf{w}})) = 1$.
By the rank-nullity theorem, $\operatorname{dim}(\F_q^d) = \operatorname{dim}(\operatorname{ker}(L_{\mathbf{w}})) + \operatorname{dim}(\operatorname{Im}(L_{\mathbf{w}}))$.
So, $d = \operatorname{dim}(\operatorname{ker}(L_{\mathbf{w}})) + 1$, which implies $\operatorname{dim}(\operatorname{ker}(L_{\mathbf{w}})) = d-1$.

The number of points in a subspace of dimension $k$ over $\F_q$ is $q^k$. Therefore, the number of points $\mathbf{x}$ where $L_1(\mathbf{x}) = L_2(\mathbf{x})$ (i.e., the size of $\operatorname{ker}(L_{\mathbf{w}})$) is $q^{d-1}$.
The total number of points in the space $\F_q^d$ is $q^d$.
The fraction of points where $L_1$ and $L_2$ agree is $\frac{q^{d-1}}{q^d} = \frac{1}{q}$. Therefore, assuming that a test point $z$ is drawn uniformly at random means that the population loss is $\mathbb P(\{h(z)=f(z)\})=1-1/q$. This proves the first statement.

The second statement follows immediately from Lemma~\ref{lem:dimSolutions} by recalling that the random linear solver picks uniformly at random one of the linear functions which agree with the labeling of $f$ across all $n$ samples, and there are $q^{d-r}$ such functions.

\end{proof}

\begin{lemma}[Expected Loss of Random Linear Algorithm]\label{lem:LinExpLoss}
Assume we are given $n'$ feature vector drawn independently and uniformly at random from $\mathcal D=\mathcal U(\Fq^d)$, with labels generated by an arbitrary linear function. We stack the feature vectors in a matrix $X\in \Fq^{n \times d}$. Let $\Delta_0 := d-n'$ and $P_{rank<d} := \mathbb P(\{\text{Rank}(X)<d\})$ which is $\bigO{q^{-(\Delta_0+1)}}$ per Lemma~\ref{lem:rankProbAsymptotics}. Then, the expected population loss $\bar L_{n'} :=\mathbb E_{S^{n'}\sim\mathcal D^{n'},\mathcal A}[L(\mathcal A_{lin},S^{n'})]$ of the random parity solver algorithm receiving $n'$ samples can be bounded as follows.

\begin{enumerate}
\item \textbf{If $d > n'$}:
\begin{align}
\left(1-1/q\right)\left(1-q^{-\Delta_0}\right) &\le \bar L_{n'} \le \left(1-1/q\right)\left(1-q^{-\Delta_0} + K_1 q^{-(2\Delta_0+1)}\right)
\end{align}
Thus, $\bar L_{n'} \approx (1-1/q)(1-q^{-\Delta_0})$.

\item \textbf{If $d \le n'$}:
\begin{align}
\left(1-1/q\right)\left((1-1/q)P_{rank<d} - K_2 q^{-(2\Delta_0+5)}\right) &\le \bar L_{n'} \le \left(1-1/q\right)^2P_{rank<d}
\end{align}
Thus, $\bar L_{n'} \approx (1-1/q)^2P_{rank<d}$.
\end{enumerate}
\end{lemma}

\begin{proof}

Let $m_0=\min(d,n')$ and assume $f$ is the ground truth linear function that labeled the features stacked in $X$, so that we obtain samples $S^{n'}$. Using Lemma~\ref{lem:linFcts}, we can write (using the law of total expectation)
\begin{align*}
\bar L_{n'}&= \sum_{i=0}^{m_0} \mathbb E[L_{n'}|\{\text{Rank}(X)=i\}]\cdot \mathbb P(\{S^{n'}:\text{Rank}(X)=i\})\\
&=\sum_{i=0}^{m_0} (1 - 1/q)\cdot\mathbb E_z[\ind{\mathcal A_{lin}(S^{n'})(z)\neq f}|\{S^{n'}:\text{Rank}(X)=i\}]\cdot \mathbb P(\{S^{n'}:\text{Rank}(X)=i\})\\
&=(1 - 1/q)\sum_{i=0}^{m_0} \mathbb P_z(\mathcal A_{lin}(S^{n'})(z)\neq f|\{S^{n'}:\text{Rank}(X)=i\})\cdot \mathbb P(\{S^{n'}:\text{Rank}(X)=i\})\\
&=(1 - 1/q)\sum_{i=0}^{m_0}(1-q^{i-d})\cdot R_q(n',d,i)
\end{align*}
where $R_q$ is defined as in Lemma~\ref{lemma:probRank}.
This means that
\begin{align}
\bar L_{n'} &= \left(1 - \frac{1}{q}\right) S_0
\end{align}
where $S_0$ is defined as
\begin{equation}
S_0 = \sum_{i=0}^{\min\{d-1,n'\}}(1-q^{i-d})\cdot \Rqn{n'}{d}{i}.
\end{equation}

Based on the approximate rank probabilities of Lemma~\ref{lem:rankProbAsymptotics}, we directly obtain good bounds on $S_0$ for large enough $q$.

\textbf{Bounds for $S_0$}

\begin{enumerate}
\item \textbf{Case 1: $d > n'$} \\
Here the sum for $S_0$ runs up to $i=n'$. Let us first show he lower bound $S_0 \ge 1-q^{-(d-n')}$. Let $A = 1-q^{-(d-n')}$. We have to show that $S_0 \ge A$.
Consider the difference
\begin{align*}
	S_0 - A &= \sum_{i=0}^{n'} (1-q^{-(d-i)}) \Rqn{n'}{d}{i} - \left(1-q^{-(d-n')}\right) \\
	&= \sum_{i=0}^{n'} (1-q^{-(d-i)}) \Rqn{n'}{d}{i} - \left(1-q^{-(d-n')}\right) \sum_{i=0}^{n'} \Rqn{n'}{d}{i} \quad \text{(since $\sum \Rqn{n'}{d}{i}=1$)} \\
	&= \sum_{i=0}^{n'} \left[ q^{-(d-n')} - q^{-(d-i)} \right] \Rqn{n'}{d}{i}\\
    &= \sum_{i=0}^{n'} q^{-(d-n')} \left(1 - q^{-(n'-i)} \right) \Rqn{n'}{d}{i}.
\end{align*}
It is easy to see that each summand is non-negative, hence the lower bound is proven.

For the upper bound, we can rewrite $S_0$ by isolating the contribution from $R_q(n',d,n')$ as
\begin{align*}
S_0 &= \sum_{i=0}^{n'} (1-q^{-(d-i)}) \Rqn{n'}{d}{i} \\
&= (1-q^{-(d-n')})\Rqn{n'}{d}{n'} + \sum_{i=0}^{n'-1} (1-q^{-(d-i)}) \Rqn{n'}{d}{i}
\end{align*}
Substituting $\Rqn{n'}{d}{n'} = 1 - \sum_{i=0}^{n'-1} \Rqn{n'}{d}{i}$ yields
\begin{align*}
S_0 &= (1-q^{-(d-n')}) \left(1 - \sum_{i=0}^{n'-1} \Rqn{n'}{d}{i}\right) + \sum_{i=0}^{n'-1} (1-q^{-(d-i)}) \Rqn{n'}{d}{i} \\
&= (1-q^{-(d-n')}) + \sum_{i=0}^{n'-1} \left[ (1-q^{-(d-i)}) - (1-q^{-(d-n')}) \right] \Rqn{n'}{d}{i} \\
&= (1-q^{-(d-n')}) + \sum_{i=0}^{n'-1} \left[ q^{-(d-n')} - q^{-(d-i)} \right] \Rqn{n'}{d}{i}.
\end{align*}

The sum term in the last line is dominated by its $i=n'-1$ term
$q^{-(d-n')}(1-q^{-1})\Rqn{n'}{d}{n'-1}$. Since by Lemma~\ref{lem:rankProbAsymptotics} $\Rqn{n'}{d}{n'-1} = O(q^{-(d-n'+1)})$, this term is $O(q^{-(d-n')-(d-n'+1)}) = O(q^{-(2(d-n')+1)})$.
Subsequent terms are of higher order in $1/q$.
Letting $\Delta_0 = d-n'$, $S_0$ is hence bounded by
\begin{equation*}
1-q^{-\Delta_0} \le S_0 \le 1-q^{-\Delta_0} + K_1 q^{-(2\Delta_0+1)}.
\end{equation*}

\item \textbf{Case 2: $d \le n'$} \\
Let $\Delta_0 = n'-d$. Let $P_{rank<d} = \sum_{i=0}^{d-1} \Rqn{n'}{d}{i}$. Note that by Lemma~\ref{lem:rankProbAsymptotics} $P_{rank<d} = 1 - \Rqn{n'}{d}{d} = \bigO{q^{-(\Delta_0+1)}}$. 

Since $1-q^{-j} \le 1-q^{-1}$ for $j \ge 1$, an immediate upper bound is $S_0 \le (1-q^{-1})P_{rank<d}$.
To find a lower bound, consider the difference
\begin{align*}
(1-q^{-1})P_{rank<d} - S_0 
&= \sum_{i=0}^{d-1} \left[ (1-q^{-1}) - (1-q^{-i+d}) \right] \Rqn{n'}{d}{i} \\
&= \sum_{i=0}^{d-1} (q^{-i+d}-q^{-1}) \Rqn{n'}{d}{i}
\end{align*}
The term for $i=d-1$ is zero. For $i \le d-2$, $q^{-i+d}-q^{-1}$ is negative and dominated by its first term (for $i=d-2$): $(q^{-1}-q^{-2})\Rqn{n'}{d}{d-2}$.
Since $\Rqn{n'}{d}{d-2} = \bigO{q^{-2(\Delta_0+2)}}$, that summand is of order $\bigO{q^{-1} \cdot q^{-2(\Delta_0+2)}} = \bigO{q^{-(2\Delta_0+5)}}$.

So $S_0$ is bounded by
\begin{equation*}
(1-q^{-1})P_{rank<d} - K_2 q^{-(2\Delta_0+5)} \le S_0 \le (1-q^{-1})P_{rank<d}
\end{equation*}
where $K_2 = \bigO{1}$ is a positive constant.
\end{enumerate}

\end{proof}

\begin{lemma}[Loss Variance of the Random Linear Algorithm]\label{lem:linLossVar}
Let $\Delta_0 = d-n'$ and $P_{rank<d} = O(q^{-(\Delta_0+1)})$ (see Lemma~\ref{lem:rankProbAsymptotics}). The variance of the population loss $L_{n'}$ of the random parity algorithm can be bounded as follows.
\begin{enumerate}
\item \textbf{If $d > n'$}:
\[
\left(1-1/q\right)^2 \left(q^{-\Delta_0}(1-q^{-\Delta_0}) - K'_{1} q^{-(2\Delta_0+1)}\right) \le \Var(L_{n'}) 
\le \left(1-1/q\right)^2 q^{-\Delta_0}(1-q^{-\Delta_0})
\]
Thus, $\Var(L_{n'}) \approx (1-1/q)^2 q^{-\Delta_0}(1-q^{-\Delta_0})$.

\item \textbf{If $d \le n'$}:
\begin{multline}
\left(1-\frac{1}{q}\right)^2 \left((1-q^{-1})P_{rank<d}(1-(1-q^{-1})P_{rank<d}) - K'_{2} q^{-(2\Delta_0+5)}\right) \le \Var(L_{n'}) \\
\le \left(1-\frac{1}{q}\right)^2 (1-q^{-1})P_{rank<d}(1-(1-q^{-1})P_{rank<d})
\end{multline}
Thus, $\Var(L_{n'}) \approx (1-1/q)^3 P_{rank<d}$.
\end{enumerate}
\end{lemma}
\begin{proof}
Recall that by Lemma~\ref{lem:linFcts}, the population loss can be written as $L_{n'}=(1-1/q)\cdot Z$ where $Z\sim Ber(\mathbb P(\mathcal A(S)\neq f)$. Hence the variance of the population loss $L_{n'}$ is given by
\begin{equation}
\Var(L_{n'}) = \left(1-\frac{1}{q}\right)^2 S_0(1-S_0)\label{eq:varLn}
\end{equation}
where as before $S_0=\mathbb P(\mathcal A(S)\neq f)=\sum_{i=0}^{m_0}(1-q^{i-d})\cdot R_q(n',d,i)$.
To bound the variance, we therefore need bounds for $S_0(1-S_0)$. For this we can reuse the bounds for $S_0$ derived in the proof of Lemma~\ref{lem:LinExpLoss}.

\textbf{Bounds for $S_0(1-S_0)$}

Let $\Delta_0 = |n'-d|$ and recall the bounds for $S_0$ derived in Lemma~\ref{lem:LinExpLoss}. Let $f(x) = x(1-x)$. This function is maximized at $x=1/2$.
\begin{enumerate}
\item \textbf{Case 1: $d > n'$} (so $\Delta_0 = d-n'$). \\
Here $S_0 = 1-q^{-\Delta_0} + E_1$, where $0 \le E_1 \le K_1 q^{-(2\Delta_0+1)}$ where $K_1 = O(1)$.
Since $q \ge 2$ and $\Delta_0 \ge 1$, $S_0 \ge 1-q^{-1} \ge 1/2$. Thus, $f(S_0)$ is evaluated on the decreasing part of the parabola $f(x)=x(1-x)$ (or at its maximum if $S_0=1/2$).
The term $S_0(1-S_0)$ is primarily determined by $1-q^{-\Delta_0}$:
\begin{equation}
S_0(1-S_0) = q^{-\Delta_0}(1-q^{-\Delta_0}) - E_{2}
\end{equation}
where $0 \le E_{2} \le K'_{1} q^{-(2\Delta_0+1)}$ for some $K'_{1} = O(1)$. Hence the main term $q^{-\Delta_0}(1-q^{-\Delta_0})$ serves as an upper bound.
The error term $E_{2}$ contains $K_1|1-2q^{-\Delta_0}|q^{-(2\Delta_0+1)}$ plus higher order terms.

\item \textbf{Case 2: $d \le n'$} (so $\Delta_0 = n'-d$). \\
Here $S_0 = (1-q^{-1})P_{rank<d} - E_2$, where $0 \le E_2 \le K_2 q^{-(2\Delta_0+5)}$ for $K_2 = \bigO{1}$, and $P_{rank<d} = \bigO{q^{-(\Delta_0+1)}}$.
Thus $S_0$ is small (i.e., $S_0 \ll 1/2$ for large $q$). The function $f(S_0)$ is evaluated on its increasing part.
The term $S_0(1-S_0)$ is primarily determined by $(1-q^{-1})P_{rank<d}$
\begin{equation}
S_0(1-S_0) = (1-q^{-1})P_{rank<d}\left(1-(1-q^{-1})P_{rank<d}\right) - E_{3}
\end{equation}
where $0 \le E_{3} \le K'_{2} q^{-(2\Delta_0+5)}$ for some $K'_{2} = \bigO{1}$. The main term $(1-q^{-1})P_{rank<d}(1-(1-q^{-1})P_{rank<d})$ serves as an upper bound. For large $q$, $S_0(1-S_0) \approx (1-q^{-1})P_{rank<d}$.
\end{enumerate}
Plugging the above bounds on $S_0(1-S_0)$ into \eqref{eq:varLn} concludes the proof.
\end{proof}

\begin{lemma}[Expected Conditional Fold Variance]\label{lem:linCov}
For large $q$, the expected conditional variance of $\hat L_1$ given $L_1$ can be approximated as
$$\E[\text{Var}(\hat L_1|L_1)] = O\left(\frac{1}{m q^{|n-d| + 2 \cdot \mathbf{1}_{n \ge d}}}\right) $$
\end{lemma}

\begin{proof}
First note that $X$ having rank $r$ implies that $\mathcal A_{lin}$ has loss $(1-q^{r-d})(1-1/q)$ because according to Lemma~\ref{lem:dimSolutions} it selects the ground truth w.p. $q^{r-d}$ (incurring zero loss) and else it selects a linear function with loss $1-1/q$ (see Lemma~\ref{lemma:p-parityRisk}).

Denoting $L_r=(1-q^{r-d})(1-1/q)$, the expected conditional variance can be expressed as
\begin{align}
\E[\text{Var}(\hat L_1|L_1)] &=  \sum_{r=0}^{\min\{n,d\}} R_q(n, d, r) \frac{(1-L_r)L_r}{m}\notag\\
&\le \frac{1}{m} \sum_{r=0}^{\min\{n,d\}} R_q(n, d, r) \left(q^{r-d} - q^{2(r-d)}\right)\label{eq:varL1upper}
\end{align}
where the first equality follows from the law of total expectation, since $m \cdot(\hat L_1|E_r) \sim \text{Bin}(m,L_r)$ conditioned on the event $E_r=\{\text{Rank}(X)=r\}$.
Let $f(r) = q^{r-d} - q^{2(r-d)}$. We analyze $f(r)$:
\begin{itemize}
\item If $r=d$, then $f(d) =  0$.
\item If $r < d$, let $s = d-r > 0$. Then $f(r) = q^{-s} - q^{-2s}$. Since $s \ge 1$ (as $r$ and $d$ are integers), $q^{-s} \ge q^{-2s}$ for $q \ge 1$. Thus, $f(r) = q^{-s}(1-q^{-s})$. For large $q$, $1-q^{-s}$ is close to $1$. More formally, $f(r) = O(q^{-s}) = O(q^{-(d-r)})$.
\end{itemize}

The approximation of the sum relies on the asymptotic behavior of $R_q(n, d, r)$ for large $q$, see Lemma~\ref{lem:rankProbAsymptotics}. Let $m_0 = \min\{n,d\}$ be the maximum possible rank of the $n \times d$ matrix, and let $\Delta_0 = |n-d|$ be the absolute difference of its dimensions. Recall that
\begin{enumerate}
\item \textbf{Probability of Full Rank}: $R_q(n, d, m_0) = 1 - O(q^{-(\Delta_0+1)})$
\item \textbf{Probability of Specific Rank Deficiency}: for $j \ge 1$ (a rank deficiency of $j$), $R_q(n, d, m_0-j) = O(q^{-j(\Delta_0+j)})$
\end{enumerate}
We now analyze the sum by considering two cases for the relationship between $n$ and $d$.

\textbf{Case 1: $n \ge d$}.
In this scenario, the maximum rank is $m_0 = d$, and $\Delta_0 = n-d$. The sum runs from $r=0$ to $d$.
The term in the sum for $r=d$ is $R_q(n, d, d) \cdot f(d) = R_q(n, d, d) \cdot 0 = 0$.
Thus, the sum is effectively over $r \le d-1$. Define $S_j = R_q(n, d, d-j) \cdot f(d-j)$.
We have $f(d-j) = q^{-j} - q^{-2j} = O(q^{-j})$.
Using Property 2 for $R_q$ yields $R_q(n, d, d-j) = O(q^{-j(\Delta_0+j)}) = 
O(q^{-j(n-d+j)})$ hence $S_j = O(q^{-j(n-d+j)} \cdot q^{-j}) = O(q^{-j(n-d+j+1)})$.

It is easy to see that $S_1$ is the dominant term in the sum over $r < d$ and hence the sum $\sum_{r=0}^{d-1} R_q(n',d,r)f(r)$ is $O(q^{-(n'-d+2)})$.
Consequently, by \eqref{eq:varL1upper}, $\E[\text{Var}(\hat L_1|L_1)] \le \frac{1}{m} \cdot O(q^{-(n-d+2)}) = O\left(\frac{1}{m q^{n-d+2}}\right)$.

\textbf{Case 2: $n < d$}.
In this scenario, the maximum rank is $m_0 = n$, and $\Delta_0 = d-n$. The sum runs from $r=0$ to $n$ with terms $S_j = R_q(n, d, n-j) \cdot g(n-j)$ where $g(n-j) = q^{(n-j)-d} - q^{2((n-j)-d)}$. 

Since both $g(n-j)$ and $R_q(n, d, n-j)$ are decreasing in $j$, $S_0$ (the term for $r=n$) is the dominant term and the sum $\sum_{r=0}^{n} R_q(n,d,r)g(r)$ is $O(q^{-(d-n)})$.
Consequently, $\E[\text{Var}(\hat L_1|L_1)] = \frac{1}{m} \cdot O(q^{-(d-n)}) = O\left(\frac{1}{m q^{d-n}}\right)$.

\textbf{Combined Result:}
If $n \ge d$, the exponent of $q$ in the denominator is $n-d+2 = |n-d|+2$.
If $n < d$, the exponent of $q$ in the denominator is $d-n = |n-d|$.
Combining hence yields
$$\E[\text{Var}(\hat L_1|L_1)] = O\left(\frac{1}{m q^{|n-d| + 2 \cdot \mathbf{1}_{n \ge d}}}\right). $$

\end{proof}

\subsection{Proof for Theorem~\ref{thm:linearMSE}}\label{app:linearMSE}
\begin{proof}
By Theorem~\ref{thm:mse_characterization} and Lemmas~\ref{lem:boundsSLS} it holds that
\begin{align*}
     \text{MSE} &\leq (\bar L_{n}-\bar L_{n-m})^2+(\sigma_{n-m}+\sigma_m)^2+\frac{k-1}{k}\Cov(\hat L_1,\hat L_2)+\frac{\bar \sigma^2}{n}+\frac{k-1}{k}\sigma_{n-m}^2+2\sigma_n\sqrt{\frac{\bar \sigma^2}{m}}\\
    &=O\bigg((\bar L_{n}-\bar L_{n-m})^2+\max(\{\sigma_{n-m}^2,\sigma_n^2\})+\E[\Var(\hat L_1|L_1)]+\sigma_n\sqrt{\E[\Var(\hat L_1|L_1)]}\bigg)
\end{align*}
where we used the facts that $\frac{k-1}{k}\in (1/2,1)$, $\bar \sigma^2=m \cdot \E[\Var(\hat L_1|L_1)]$, $\Cov(\hat L_1,\hat L_2)\le \Var(\hat L_1)=\E[\Var(\hat L_1|L_1)] + \var(L_1)$.

Similarly, for the lower bound,
\[
\text{MSE} \ge (\bar L_{n}-\bar L_{n-m})^2-\frac{k-1}{k}\sigma_{n-m}^2-2\sigma_n\sqrt{\E[\Var(\hat L_1|L_1)]}=(\bar L_{n}-\bar L_{n-m})^2+O\big(\sigma_{n-m}^2\big)+O\bigg(\sigma_n\sqrt{\E[\Var(\hat L_1|L_1)]}\bigg)
\]
where for the first inequality we simply ignored some of the positive terms of Theorem~\ref{thm:mse_characterization}.

Let $\beta_{low} := (\bar{L}_{n_t} - \bar{L}_n)^2$. Recall $P(x,d,q) = C_P q^{-(x-d+1)}$ for $d \le x$ [called $P_{rank<d}$ earlier].

\textbf{Expected losses ($\bar{L}_x$) and loss variances ($\sigma_x$):}

We recall from Lemma~\ref{lem:LinExpLoss} and Lemma~\ref{lem:linLossVar}:
\begin{itemize}
    \item If $d>x$: $\bar{L}_x = \Theta\big((1-q^{-1})(1-q^{-(d-x)})\big)$, $\sigma_x^2 = \Theta\big((1-q^{-1})^2 q^{-(d-x)}(1-q^{-(d-x)})\big)$. So $\bar{L}_x = O(1)$, $\bar{L}_x = 1-O(1/q)$, $\sigma_x = O(q^{-(d-x)/2})$.
    \item If $d \le x$: $\bar{L}_x = \Theta\big((1-q^{-1})^2 P(x,d,q)\big)$, $\sigma_x^2 = \Theta\big((1-q^{-1})^3 P(x,d,q)\big)$.
        So, $\bar{L}_x = O(q^{-(x-d+1)})$, $\sigma_x = O(q^{-(x-d+1)/2})$.
\end{itemize}

Now we are ready to put all pieces together.
\textbf{MSE Bounds:}
\textbf{Case 1: $n < d$} (Upper Bound)
\begin{itemize}
 \item Loss variances: we have $\sigma^2_n = O(q^{-(d-n)})$, $\sigma^2_{n_t} = O(q^{-(d-n_t)})$. So $\max(\{\sigma_{n-m}^2,\sigma_n^2\}) = O(q^{-(d-n)})$.
    \item Expected losses: $\bar{L}_n =\Theta( (1-q^{-1})(1-q^{-(d-n)})$), $\bar{L}_{n_t}= \Theta((1-q^{-1})(1-q^{-(d-n_t)}))$.
    Hence $\Delta\bar{L} = \Theta((1-q^{-1})(q^{-(d-n)} - q^{-(d-n_t)})) = O(q^{-(d-n)}(1-q^{-m}))$. Since $m \ge 1, q \ge 2$ it holds $(1-q^{-m})=O(1)$. So $(\Delta\bar{L})^2 = O(q^{-2(d-n)})$.
    \item Combined bound: 
\begin{align*}
\text{MSE} &=O\bigg((\bar L_{n}-\bar L_{n-m})^2+\max(\{\sigma_{n-m}^2,\sigma_n^2\})+\E[\Var(\hat L_1|L_1)]+\sigma_n\sqrt{\E[\Var(\hat L_1|L_1)]}\bigg)\\
&= O(
q^{-2(d-n)} +q^{-(d-n)} + m^{-1} q^{-(d-n)} + q^{-(d-n)}(m^{-1/2}q^{-\frac{d-n}{2}}))\\
&= O(q^{-(d-n)})
\end{align*}
\end{itemize}

\textbf{Case 2: $n \ge d$ and $n_t < d$} (Lower Bound)
\begin{itemize}
    \item Loss variances: $\sigma^2_n = O(q^{-(n-d+1)})$, $\sigma^2_{n_t} = O(q^{-(d-n_t)})$.
    \item Expected losses: $\bar{L}_n = O(q^{-(n-d+1)})$. $\bar{L}_{n_t} = O(1)$.
    $\Delta\bar{L} = \bar{L}_{n_t} - \bar{L}_n = 1-O(1/q) - O(q^{-(n-d+1)}) = 1-O(1/q)$ and hence $(\Delta\bar{L})^2 = 1-O(1/q)$.
    \item Combined bound: The loss stability term dominates all other terms, hence $\text{MSE} \ge 1-O(1/q)$ which is $\Omega(1)$ since $q \ge 2$.
\end{itemize}
\textbf{Case 3: $n_t \ge d$} (Upper Bound)
\begin{itemize}
    \item Loss variances: $\sigma^2_n = O(q^{-(n-d+1)})$, $\sigma^2_{n_t} = O(q^{-(n_t-d+1)})$. So $\max(\{\sigma_{n-m}^2,\sigma_n^2\}) = O(q^{-(n_t-d+1)})$.
    \item Expected losses: $\bar{L}_n = \Theta(q^{-(n-d+1)})$, $\bar{L}_{n_t} = \Theta(q^{-(n_t-d+1)})$.
    Hence, $\Delta\bar{L} = O(q^{-(n_t-d+1)})$ and $(\Delta\bar{L})^2 = O(q^{-2(n_t-d+1)})$.
    \item Combined bound: 
    \[
    \text{MSE} = O(q^{-2(n_t-d+1)}+q^{-(n_t-d+1)}+m^{-1}q^{-(n-d+1)}+q^{-(n-d+1)}m^{-1/2}q^{-(n-d+1)/2})=O(q^{-(n_t-d+1)}).
    \]
    \end{itemize}
\end{proof}

\newpage

\section{Fold Covariance of the Square-Wave Algorithm}

\subsection{Setup and Definitions}

Throughout, let $N=n-2m$. Let $W \sim \text{Bin}(m, 1/2)$, $a = s/\sqrt{m}$, and $p_w = P(W=w) = \binom{m}{w} 2^{-m}$.
The function $f(a)$ is defined as
\begin{equation}
	f(a) = \E_W \left[ \left(\frac{W - m/2}{m}\right) \epsilon \left( a + \frac{W}{\sqrt{m}} \right) \right]
	= \sum_{w=0}^m p_w \left(\frac{w - m/2}{m}\right) (-1)^{\lfloor a + w/\sqrt{m} \rfloor}
\end{equation}
Let $\mu = m/2$ and $\sigma = \sqrt{m}/2$ be the mean and standard deviation of $W$.
Let $g(w)$ be the PDF of a $\mathcal N(\mu, \sigma^2)$ random variable:
\begin{equation}
	g(w) = \frac{1}{\sqrt{2\pi \sigma^2}} e^{-\frac{(w-\mu)^2}{2\sigma^2}} = \sqrt{\frac{2}{\pi m}} e^{-\frac{2(w-m/2)^2}{m}}
\end{equation}
Let $h_g(w)$ be the Gaussian-weighted term:
\begin{equation}
	h_g(w) = g(w) \left(\frac{w - m/2}{m}\right)
\end{equation}
Let $\psi(w) = \epsilon(a + w/\sqrt{m}) = (-1)^{\lfloor a + w/\sqrt{m} \rfloor}$.
We seek a lower bound for $|f(a)|$.

Our first step is to greatly simplify the fold-covariance via a factorization.
		\subsection{Factorization of the Fold-Covariance}
\begin{theorem}[Factorization Identity]\label{thm:factorization}
It holds that
	\[\Cov(\hat L_1, \hat L_2)=\mathbb{E}\big[f(S/\sqrt{m})^2\big]\]
	where
	\[
	f(a):=\mathbb{E}_W\!\left[\Big(\tfrac{W}{m}-\tfrac12\Big)\varepsilon\!\left(a+\frac{W}{\sqrt{m}}\right)\right].
	\]
\end{theorem}

\begin{proof}
It follows from the definitions of $\hat L_1, \hat L_2$ that
\[
\hat L_1-\tfrac12=\big(\tfrac{W_1}{m}-\tfrac12\big)\,(-1)^{Y_1}
=\big(\tfrac{W_1}{m}-\tfrac12\big)\,\varepsilon\!\left(\frac{S+W_2}{\sqrt{m}}\right),
\]
\[
L_2-\tfrac12=\big(\tfrac{W_2}{m}-\tfrac12\big)\,\varepsilon\!\left(\frac{S+W_1}{\sqrt{m}}\right).
\]
Define $W{\sim}\mathrm{Bin}(m,\tfrac12)$ independent of $S$. Then, by conditioning on \(S\) and using independence of \(W_1,W_2\),
\begin{align*}
	\Cov(\hat L_1, \hat L_2)
	&= \E\Bigg[ \Bigg\{\big(\tfrac{W_1}{m}-\tfrac12\big)\,\varepsilon\!\left(\frac{S+W_1}{\sqrt{m}}\right)\Bigg\}\Bigg\{ \big(\tfrac{W_2}{m}-\tfrac12\big)\,\varepsilon\!\left(\frac{S+W_2}{\sqrt{m}}\right)\Bigg\}\Bigg]\\
	&=\mathbb{E}_S\Big[
	\Big(\mathbb{E}_{W}\Big[\big(\tfrac{W}{m}-\tfrac12\big)\,\varepsilon\!\big(\tfrac{S+W}{\sqrt{m}}\big)\Big]\Big)^2
	\Big]\\
	&=\mathbb{E}\big[f(S/\sqrt{m})^2\big],
\end{align*}
where
\[
f(a):=\mathbb{E}_W\!\left[\Big(\tfrac{W}{m}-\tfrac12\Big)\varepsilon\!\left(a+\frac{W}{\sqrt{m}}\right)\right].
\]
	\end{proof}

\subsection*{Proof Sketch}
	
The proof in this section aims to find an asymptotic value for $\Cov(\hat L_1, \hat L_2) = \E_S[f(a)^2]$.
	
\subsection*{Step 1: Simplify the Problem (Sum $\to$ Integral)}
	
We start with $f(a)$, a difficult discrete sum over a Binomial distribution. The first step is to get rid of the discrete sum and approximate it with a continuous integral, which is easier to manipulate.
	
\[
f(a) = \underbrace{\sum_{w=0}^m p_w \left(\frac{w - m/2}{m}\right) \psi(w)}_{\text{Discrete, hard}}
\quad\xrightarrow{\text{Approximation}}\quad
\underbrace{\int_{-\infty}^\infty h_g(t) \psi(t) dt}_{\text{Continuous, easier}}
\]
This is a standard analysis step. We replace the Binomial PMF $p_w$ with a Gaussian PDF $g(w)$, and the sum with an integral.
	
\textbf{Result:} $f(a) = I_g + \calO(m^{-1})$, where $I_g$ is the integral.
	
\subsection*{Step 2: Evaluate the Integral (Integral $\to$ New Sum)}
	
Now we must solve the integral $I_g$ which is an integral of a smooth function $h_g(t)$ against a high-frequency square wave $\psi(t)$. 
Our function $h_g(t)$ is special: it's related to the derivative of a Gaussian ($h_g(t) \propto u e^{-u^2/2}$). The integral $\int u e^{-u^2/2} du$ is trivial. We split the integral at the jump points of $\psi(t)$. This turns the integral into a sum:
\[
I_g = \sum_{r \in \Z} (-1)^r \int_{u_r}^{u_{r+1}} u \phiG(u) du
= \sum_{r \in \Z} (-1)^r \big[ -\phiG(u) \big]_{u_r}^{u_{r+1}}
\]
This is a \textbf{telescoping sum} and simplifies the integral $I_g$ into the much cleaner discrete sum.
	
\textbf{Result:} $I_g = \frac{1}{\sqrt{m}} \sum_{r \in \Z} (-1)^r e^{-2(r-C_m)^2}$.
	
\subsection*{Step 3: Analyze the New Sum (The \emph{First} Fourier Tool: PSF)}
	
We have successfully simplified $f(a)$, but now we have a new problem: an alternating, shifted sum of a sampled Gaussian. The \textbf{Poisson Summation Formula (PSF)} is the precise tool for relating a sum of samples of a function to a sum of samples of its Fourier transform. The PSF converts our complicated, slowly-converging sum $\sum (-1)^r f(r-\delta)$ into a different sum that converges \emph{extremely} fast. We apply the PSF to $f(x) = e^{-2x^2}$. The Fourier transform $\hat{f}(s)$ is also a Gaussian, $e^{-\pi^2 s^2 / 2}$, which decays very rapidly. The resulting sum in the frequency domain is:
\[
\Theta(\delta) = \sum_{j=0}^\infty C_j \cos((2j+1)\pi\delta)
\]
This sum is dominated by its first term ($j=0$). This is our main analytic expression for $f(a)$.
	
\textbf{Result:} $f(a) = \frac{1}{\sqrt{m}} \Theta(\delta(a)) + \calO(m^{-1})$.
	
\subsection*{Step 4: Analyze the Expectation (The \emph{Second} Fourier Tool: Series)}
	
We are finally ready to tackle the main goal, $\Cov(\hat L_1, \hat L_2) = \E_S[f(a)^2]$.
\[
\Cov(\hat L_1, \hat L_2) \approx \E_S \left[ \left( \frac{1}{\sqrt{m}} \Theta(\delta_S) \right)^2 \right] = \frac{1}{m} \E_S[g(\delta_S)]
\quad \text{where} \quad g(x) = \Theta(x)^2.
\]
We now need to find the expectation of a periodic function $g(x)$ where its phase $\delta_S$ is a random variable. The most natural way to analyze a periodic function is to decompose it into its average value and its oscillations. This is the definition of a \textbf{Fourier Series}.
We write $g(\delta_S) = c_0 + \sum_{l \ne 0} c_l e^{2\pi i l \delta_S}$. By linearity of expectation:
\[
\E_S[g(\delta_S)] = c_0 + \sum_{l \ne 0} c_l \E_S[e^{2\pi i l \delta_S}]
\]
The expectation on the r.h.s. is the characteristic function of $S$. We show that this term is very small for $l \ne 0$.
	
\textbf{Result:} $\E_S[g(\delta_S)] = c_0 + E_{\text{fourier}}$, where $E_{\text{fourier}}$ is a small bias.
	
\subsection*{Step 5: Connecting the Two Fourier Tools}
	
We know that if $g = \Theta \cdot \Theta$, then the Fourier coefficients of $g$ (the $c_l$'s) are the \emph{discrete convolution} of the Fourier coefficients of $\Theta$ (the $\hat{\Theta}_r$'s).
\[
c_l = (\hat{\Theta} * \hat{\Theta})(l) = \sum_{r \in \Z} \hat{\Theta}(r) \hat{\Theta}(2l-r)
\]
We use the coefficients $\hat{\Theta}_r$ we found in Step 3 (from the PSF) to compute the $c_l$ we need for Step 4. This calculation gives us the final numerical values for our main term ($c_0$) and our bias terms ($c_1$, etc.).

		\bigskip
\bigskip
\subsection{Technical Lemmas}

	\newcommand{\Bbar}{\bar{B}} 

		\begin{lemma}[Euler-Maclaurin Summation Formula]\label{lem:sqrt_EM}
			Let $a, b \in \R$ such that $b-a \in \N^+$. Let $p \ge 2$ be an integer. Let $f$ be a function with $p$ continuous derivatives on $[a, b]$. Then,
			\begin{align*}
				\sum_{i=a}^b f(i) = \int_a^b f(t)dt &+ \frac{f(a) + f(b)}{2} + \sum_{j=2}^p \frac{b_j}{j!}\left(f^{(j-1)}(b) - f^{(j-1)}(a)\right) \\
				&- \int_a^b \frac{\Bbar_p(1-t)}{p!} f^{(p)}(t)dt
			\end{align*}
			where $b_j$ are the Bernoulli numbers and $\Bbar_p(x) = B_p(\{x\})$ is the $p$-th periodic Bernoulli polynomial.
		\end{lemma}
		\begin{proof}
			This is a standard result from numerical analysis. We also use the property $\Bbar_p(1-t) = \Bbar_p(-t) = (-1)^p \Bbar_p(t)$.
		\end{proof}

		We now apply this lemma to the sum $S_g = \sum_{w \in I_c} h_g(w) \psi(w)$, where $h_g(w)$ is smooth but $\psi(w) = (-1)^{\lfloor a + w/\sqrt{m} \rfloor}$ is a step function.

\begin{lemma}[Tail Bound]\label{lem:sqrt_tailbound}
	Let $I_c = \{w \in \Z : |w - \mu| \le \sqrt{2 \log m} \cdot \sigma\} = \{w : |w - m/2| \le \sqrt{m \log m}\}$.
	Let $I_{tail} = \{0, \dots, m\} \setminus I_c$.
	Then $f(a) = \sum_{w \in I_c} p_w \left(\frac{w - \mu}{m}\right) \psi(w) + E_1$, where $|E_1| = \calO(m^{-2})$.
\end{lemma}
\begin{proof}
	The error $E_1$ is the sum over the tails:
	\begin{align*}
		|E_1| &= \left| \sum_{w \in I_{tail}} p_w \left(\frac{w - \mu}{m}\right) \psi(w) \right|
		\le \sum_{w \in I_{tail}} p_w \left| \frac{w - m/2}{m} \right| \\
		&\le \sum_{w \in I_{tail}} p_w \left( \frac{m/2}{m} \right) = \frac{1}{2} P(W \in I_{tail})
		= \frac{1}{2} P\left(|W - \mu| > \sqrt{2 \log m} \cdot \sigma\right)
	\end{align*}
	By Hoeffding's inequality, $P(|W - \mu| > t) \le 2e^{-2t^2/m}$.
	Setting $t = \sqrt{m \log m}$, we have:
	\begin{equation*}
		|E_1| \le \frac{1}{2} \left( 2e^{-2(m \log m)/m} \right) = e^{-2 \log m} = m^{-2}.
	\end{equation*}
	Thus $f(a) = \sum_{w \in I_c} p_w \left(\frac{w - \mu}{m}\right) \psi(w) + \calO(m^{-2})$.
\end{proof}

\begin{lemma}[Gaussian Approximation]\label{lem:sqrt_gaussapprox}
	Let $S_c = \sum_{w \in I_c} p_w \left(\frac{w - \mu}{m}\right) \psi(w)$.
	Then $S_c = \sum_{w \in I_c} h_g(w) \psi(w) + E_2$, where $|E_2| = \calO(m^{-3/2} \log m)$.
\end{lemma}
\begin{proof}
	We expand $S_c = \sum_{w \in I_c} (p_w -g(w)) \left(\frac{w - \mu}{m}\right) \psi(w)+ \sum_{w \in I_c} h_g\psi(w)$.
	By the local limit theorem of Lemma~\ref{lem:LLT} (for $W \sim \text{Bin}(m, 1/2)$), for $w \in I_c$:
	\begin{equation*}
		p_w  = g(w) + E_{LLT}
	\end{equation*}
	where $E_{LLT}(w) = \calO(m^{-3/2})$.
	The error $E_2$ is therefore bounded as
	\begin{align*}
		|E_2| &= \left| \sum_{w \in I_c} (p_w - g(w)) \left(\frac{w - \mu}{m}\right) \psi(w) \right|
		\le \sum_{w \in I_c} \left| E_{LLT}(w) \right| \left| \frac{w - \mu}{m} \right| \\
		&\le \sum_{w \in I_c} \calO(m^{-3/2}) \left( \frac{\sqrt{m \log m}}{m} \right)
		= \sum_{w \in I_c} \calO(m^{-2} \sqrt{\log m})
	\end{align*}
	The number of terms $|I_c|$ is $2\sqrt{m \log m} + 1 = \calO(\sqrt{m \log m})$. Hence,
	\begin{equation*}
		|E_2| \le \calO(\sqrt{m \log m}) \cdot \calO(m^{-2} \sqrt{\log m}) = \calO(m^{-3/2} \log m).
	\end{equation*}
\end{proof}

\begin{lemma}[Sum--to--Integral Approximation]\label{lem:sum-int}
			Let $m\in\mathbb P$, $\mu=m/2$, and define
			\[
			g(t)=\sqrt{\frac{2}{\pi m}}\;\exp\!\Big(-\frac{2(t-\mu)^2}{m}\Big),\qquad
			h_g(t)=g(t)\,\frac{t-\mu}{m}.
			\]
			Fix $a\in\mathbb R$ and let $\psi(t)=(-1)^{\lfloor a+t/\sqrt{m}\rfloor}$.
			Let
			\[
			W_L=\Big\lceil \mu-\sqrt{m\log m}\,\Big\rceil,\qquad
			W_R=\Big\lfloor \mu+\sqrt{m\log m}\,\Big\rfloor,
			\]
			and define
			\[
			S_g:=\sum_{w=W_L}^{W_R} h_g(w)\,\psi(w),\qquad
			I_g:=\int_{W_L}^{W_R} h_g(t)\,\psi(t)\,dt.
			\]
			Then
			\[
			S_g \;=\; I_g \;+\; E_3,\qquad
			|E_3|\ \le\ C_m\,m^{-1}\;+\; C^\star\,m^{-3/2}\sqrt{\log m},
			\]
			for constants $C_m, C^\star$. In particular, $|E_3|=O(1/m)$.
		\end{lemma}
		
		\begin{proof}
			Let the real jump points be $t_j:=\sqrt{m}\,(j-a)$, so $\psi(t)=(-1)^j$ on $[t_j,t_{j+1})$.
			Define the integer blocks
			\[
			A_j:=\lceil t_j\rceil,\qquad
			B_j:=\lceil t_{j+1}\rceil-1,\qquad
			R_j:=\{w\in\mathbb Z: A_j\le w\le B_j\},
			\]
			so that $\psi(w)=(-1)^j$ on $R_j$ and, crucially,
			\[
			B_j+1=\lceil t_{j+1}\rceil=A_{j+1}.
			\]
			
			\smallskip
			\noindent\textbf{Step 1 (Decomposition).}
			Set
			\[
			I_g':=\sum_j (-1)^j\int_{A_j}^{B_j} h_g(t)\,dt,
			\qquad
			E_3=S_g-I_g=(S_g-I_g')+(I_g'-I_g)=:E_{\mathrm{EM}}+E_{\mathrm{BM}}.
			\]
			(The sum is over all $j$ such that $R_j$ has a non-empty intersection with $[W_L, W_R]$).
			
\smallskip
\smallskip
\noindent\textbf{Step 2 (Euler--Maclaurin on each block).}
The Euler--Maclaurin formula in Lemma~\ref{lem:sqrt_EM} (with $p=2$) gives the error for a single block $R_j$ as:
\[
\text{Error}_j = \frac{h_g(A_j)+h_g(B_j)}{2}
+\frac{B_2}{2}\big(h_g'(B_j)-h_g'(A_j)\big)
-\!\int_{A_j}^{B_j}\!\frac{\Bbar_2(1-t)}{2}\,h_g''(t)\,dt.
\]
The total error $E_{\mathrm{EM}}$ is the alternating sum $E_{\mathrm{EM}} = \sum_j (-1)^j \text{Error}_j$, where the sum runs over $j$ from $j_L$ to $j_R$. Let $P_j = j_R - j_L + 1 = \Theta(\sqrt{\log m})$ be the number of blocks. We analyze the three parts of this sum separately.

\begin{itemize}
	\item \textbf{Remainder Integral Term ($h_g''$):}
	Let $E_{\mathrm{EM},3} = - \sum_j (-1)^j \int_{A_j}^{B_j}\frac{\Bbar_2(1-t)}{2}\,h_g''(t)\,dt$.
	We bound its magnitude by the integral over the entire window $I_c = [W_L, W_R]$. Since $|\Bbar_2(x)|$ is bounded by a constant $C_B$:
	\begin{align*}
		|E_{\mathrm{EM},3}| &\le \sum_j \int_{A_j}^{B_j} \left|\frac{\Bbar_2(1-t)}{2}\right| |h_g''(t)| dt
		\le C_B \int_{W_L}^{W_R} |h_g''(t)| dt \\
		&\le C_B \cdot (\text{Length of } I_c) \cdot \|h_g''\|_\infty \\
		&\le C_B \cdot O(\sqrt{m \log m}) \cdot O(m^{-2}) = O(m^{-3/2}\sqrt{\log m}).
	\end{align*}
	
	\item \textbf{Boundary Term ($h_g$):}
	Let $E_{\mathrm{EM},1} = \frac{1}{2}\sum_j (-1)^j (h_g(A_j) + h_g(B_j))$. We rewrite this sum by grouping terms adjacent to the block boundaries, using $A_{j+1} = B_j+1$.
	\begin{align*}
		E_{\mathrm{EM},1} &= \frac{1}{2} \left[ \sum_{j=j_L}^{j_R} (-1)^j h_g(A_j) + \sum_{j=j_L}^{j_R} (-1)^j h_g(B_j) \right] \\
		&= \frac{1}{2} \Big[ (-1)^{j_L} h_g(A_{j_L}) + \sum_{j=j_L+1}^{j_R} (-1)^j h_g(A_j) \\
		&\qquad\quad + \sum_{j=j_L}^{j_R-1} (-1)^j h_g(B_j) + (-1)^{j_R} h_g(B_{j_R}) \Big]
	\end{align*}
	We group the two inner sums by re-indexing the first one ($j \to j+1$):
	\[
	\sum_{j=j_L}^{j_R-1} (-1)^{j+1} h_g(A_{j+1}) + \sum_{j=j_L}^{j_R-1} (-1)^j h_g(B_j)
	= \sum_{j=j_L}^{j_R-1} (-1)^j (h_g(B_j) - h_g(A_{j+1}))
	\]
	Using $A_{j+1} = B_j+1$, the inner sum becomes:
	\[
	\sum_{j=j_L}^{j_R-1} (-1)^j (h_g(B_j) - h_g(B_j+1)) = \sum_{j=j_L}^{j_R-1} (-1)^{j+1} (h_g(B_j+1) - h_g(B_j))
	\]
	By the Mean Value Theorem, $h_g(B_j+1) - h_g(B_j) = h_g'(\xi_j)$ for some $\xi_j \in (B_j, B_j+1)$.
	Thus, the exact expression is:
	\[
	E_{\mathrm{EM},1} = \frac{1}{2} \left[ (-1)^{j_L} h_g(A_{j_L}) + (-1)^{j_R} h_g(B_{j_R}) + \sum_{j=j_L}^{j_R-1} (-1)^{j+1} h_g'(\xi_j) \right]
	\]
	The two endpoint terms $h_g(A_{j_L})$ and $h_g(B_{j_R})$ are located near $W_L$ and $W_R$, where $|h_g(t)| = O(m^{-3}\sqrt{\log m})$ and are negligible.
	For the main sum, we use the triangle inequality:
	\[
	\left| \frac{1}{2} \sum_{j=j_L}^{j_R-1} (-1)^{j+1} h_g'(\xi_j) \right| \le \frac{1}{2} \sum_{j=j_L}^{j_R-1} |h_g'(\xi_j)| \le \frac{1}{2} (P_j - 1) \|h_g'\|_\infty
	\]
	Since $P_j = O(\sqrt{\log m})$ and $\|h_g'\|_\infty = O(m^{-3/2})$, this sum is $O(m^{-3/2}\sqrt{\log m})$.
	The total contribution is $|E_{\mathrm{EM},1}| = O(m^{-3/2}\sqrt{\log m})$.
	
	\item \textbf{Boundary Term ($h_g'$):}
	Let $E_{\mathrm{EM},2} = \frac{B_2}{2}\sum_j (-1)^j (h_g'(B_j) - h_g'(A_j))$. A similar telescoping argument applies, replacing $h_g$ with $h_g'$ and $h_g'$ with $h_g''$.
	\[
	E_{\mathrm{EM},2} = \frac{B_2}{2} \left[ (\text{Endpoints}) + \sum_{j=j_L}^{j_R-1} (-1)^{j+1} h_g''(\zeta_j) \right]
	\]
	The endpoint terms are negligible. The main sum is bounded by:
	\[
	\left| \frac{B_2}{2} \sum_{j=j_L}^{j_R-1} (-1)^{j+1} h_g''(\zeta_j) \right| \le \frac{|B_2|}{2} (P_j - 1) \|h_g''\|_\infty
	\]
	This is $O(\sqrt{\log m}) \cdot O(m^{-2}) = O(m^{-2}\sqrt{\log m})$.
\end{itemize}
Combining these, the total Euler-Maclaurin error $E_{\mathrm{EM}}$ is dominated by the remainder integral and $h_g$ boundary terms:
\[
|E_{\mathrm{EM}}| \le |E_{\mathrm{EM},1}| + |E_{\mathrm{EM},2}| + |E_{\mathrm{EM},3}| \le O(m^{-3/2}\sqrt{\log m}).
\]
			
			\smallskip
			\noindent\textbf{Step 3 (Boundary mismatch).}
			Write
			\[
			E_{\mathrm{BM}}=\sum_j (-1)^j\!\left[\int_{A_j}^{B_j} h_g - \int_{t_j}^{t_{j+1}} h_g\right]
			=\sum_j (-1)^j\!\left[-\!\int_{t_j}^{A_j}\! h_g \;-\! \int_{B_j}^{t_{j+1}}\! h_g\right].
			\]
			We pair the "tail" of block $j$ with the "head" of block $j+1$.
			\begin{itemize}
				\item Tail of $j$: $T_j = (-1)^j \left( - \int_{B_j}^{t_{j+1}} h_g \right)$
				\item Head of $j+1$: $H_{j+1} = (-1)^{j+1} \left( - \int_{t_{j+1}}^{A_{j+1}} h_g \right)$
			\end{itemize}
			Their sum is:
			\[
			T_j + H_{j+1} = (-1)^{j+1} \left( \int_{B_j}^{t_{j+1}} h_g - \int_{t_{j+1}}^{A_{j+1}} h_g \right)
			= (-1)^{j+1} \int_{B_j}^{A_{j+1}} h_g
			\]
			Using $A_{j+1} = B_j+1$, this becomes:
			\[
			T_j + H_{j+1} = (-1)^{j+1} \int_{B_j}^{B_j+1} h_g(t)\,dt.
			\]
			The full sum $E_{\mathrm{BM}}$ thus collapses to a sum over these unit intervals, plus the two un-paired residuals at the global endpoints ($W_L, W_R$):
			\[
			E_{\mathrm{BM}} = \sum_j (-1)^{j+1} \int_{B_j}^{B_j+1} h_g(t)\,dt + (\text{Endpoint residuals})
			\]
			Let $s_j := \int_{B_j}^{B_j+1} h_g(t)\,dt$. The sum $E_{\mathrm{BM}}$ collapses to $S = \sum_j (-1)^{j+1} s_j$ plus the two un-paired endpoint residuals. These residuals are integrals of length $< 1$ at $t \approx \mu \pm \sqrt{m \log m}$, where $|h_g(t)| = O(m^{-3}\sqrt{\log m})$. Their contribution is thus negligible relative to $O(1/m)$. 
			
			We now bound the alternating sum $S$. Let $s_j := \int_{B_j}^{B_j+1} h_g(t)\,dt$. We need to bound the alternating sum $S = \sum_j (-1)^{j+1} s_j$.
			
			The total variation of the sequence $\{s_j\}$ is $\sum_j |s_{j+1}-s_j|$. This sequence samples the smooth function $h_g(t)$, so its total variation is bounded by the total variation of $h_g(t)$ itself, $\int |h_g'(t)| dt$.
			\begin{align*}
				\int_{-\infty}^\infty |h_g'(t)| dt
				&= \int_{-\infty}^\infty \left|\frac{2}{\sqrt{2\pi}\,m^{3/2}}(1-u^2)e^{-u^2/2}\right| \frac{dt}{du} du \quad (u=\frac{2(t-\mu)}{\sqrt{m}})\\
				&= \int_{-\infty}^\infty \frac{2}{\sqrt{2\pi}\,m^{3/2}} |1-u^2|e^{-u^2/2} \left(\frac{\sqrt{m}}{2}\right) du \\
				&= \frac{1}{\sqrt{2\pi}\,m} \int_{-\infty}^\infty |1-u^2|e^{-u^2/2} du = \frac{C_{\text{int}}}{m} = O(1/m).
			\end{align*}
			For an alternating sum $S = \sum (-1)^j s_j$, its magnitude is bounded by the sum of its variations, $|S| \le \sum |s_{j+1}-s_j| + |s_{start}| + |s_{end}|$.
			Since the total variation is $O(1/m)$ and the endpoint terms $s_j$ are $O(m^{-M})$, we have $|S| = O(1/m)$.
			
			Hence, $|E_{\mathrm{BM}}| = O(1/m)$.
			
			\smallskip
			\noindent\textbf{Conclusion.}
			The total error is
			\[
			|E_3| \le |E_{\mathrm{EM}}| + |E_{\mathrm{BM}}|
			\le O(m^{-3/2}\sqrt{\log m}) + O(1/m) = O(1/m).
			\]
		\end{proof}

		\begin{lemma}[Alternating Poisson Summation with Shift]\label{lem:poisson_sum}
			Let $f: \mathbb{R} \to \mathbb{C}$ be an \textbf{even function} ($f(x) = f(-x)$) that is continuous, integrable, and decays sufficiently fast (e.g., $f \in \mathcal{S}(\mathbb{R})$, the Schwartz space). Let $\delta \in \mathbb{R}$ be a shift.
			
			Let the Fourier transform be defined as $\hat{f}(\xi) = \int_{-\infty}^{\infty} f(x) e^{-2\pi i \xi x} dx$.
			Then the following identity holds:
			$$ \sum_{N \in \mathbb{Z}} (-1)^N f(N-\delta) = \sum_{N \in \mathbb{Z}} \hat{f}(N+1/2) e^{2\pi i (N+1/2)\delta} $$
		\end{lemma}
		
		\begin{proof}
			We begin with the standard Poisson Summation Formula (PSF), which states that for a suitable function $g(x)$:
			$$ \sum_{N \in \mathbb{Z}} g(N) = \sum_{m \in \mathbb{Z}} \hat{g}(m) $$
			
			To evaluate the sum $S = \sum_{N \in \mathbb{Z}} (-1)^N f(N-\delta)$, we define an auxiliary function $g(x)$. Using the identity $(-1)^N = e^{i\pi N}$, we set:
			$$ g(x) = e^{i\pi x} f(x-\delta) $$		
			Next, we compute the Fourier transform $\hat{g}(m)$ of $g(x)$:
			\begin{align*}
				\hat{g}(m) &= \int_{-\infty}^{\infty} g(x) e^{-2\pi i m x} dx \\
				&= \int_{-\infty}^{\infty} e^{i\pi x} f(x-\delta) e^{-2\pi i m x} dx \\
				&= \int_{-\infty}^{\infty} f(x-\delta) e^{-2\pi i x (m - 1/2)} dx
			\end{align*}
			We apply the substitution $u = x - \delta$, which implies $x = u + \delta$ and $du = dx$.
			\begin{align*}
				\hat{g}(m) &= \int_{-\infty}^{\infty} f(u) e^{-2\pi i (u+\delta) (m - 1/2)} du \\
				&= e^{-2\pi i \delta (m - 1/2)} \int_{-\infty}^{\infty} f(u) e^{-2\pi i u (m - 1/2)} du \\
				&= e^{-2\pi i \delta (m - 1/2)} \cdot \hat{f}(m - 1/2)
			\end{align*}
			Substituting this result back into the standard PSF, we have:
			$$ \sum_{N \in \mathbb{Z}} (-1)^N f(N-\delta) = \sum_{m \in \mathbb{Z}} \hat{f}(m - 1/2) e^{-2\pi i \delta (m - 1/2)} $$
			This identity holds for any suitable function $f$. To arrive at the form stated in the lemma, we now apply the assumption that $f$ is an even function.
			
			If $f(x)$ is even, its Fourier transform $\hat{f}(\xi)$ is also even, i.e., $\hat{f}(\xi) = \hat{f}(-\xi)$.
			
			We re-index the sum on the right-hand side. Let $m = -p$, where $p \in \mathbb{Z}$.
			\begin{align*}
				\sum_{m \in \mathbb{Z}} \hat{f}(m - 1/2) e^{-2\pi i \delta (m - 1/2)} &= \sum_{p \in \mathbb{Z}} \hat{f}(-p - 1/2) e^{-2\pi i \delta (-p - 1/2)} \\
				&= \sum_{p \in \mathbb{Z}} \hat{f}(-(p + 1/2)) e^{2\pi i \delta (p + 1/2)}
			\end{align*}
			
			Applying the even property $\hat{f}(-(p + 1/2)) = \hat{f}(p + 1/2)$, the sum becomes:
			$$ \sum_{p \in \mathbb{Z}} \hat{f}(p + 1/2) e^{2\pi i \delta (p + 1/2)} $$
			
			Finally, relabeling the summation index $p$ to $N$ yields the desired result:
			$$ \sum_{N \in \mathbb{Z}} (-1)^N f(N-\delta) = \sum_{N \in \mathbb{Z}} \hat{f}(N + 1/2) e^{2\pi i (N + 1/2)\delta} $$
		\end{proof}

\begin{lemma}[Periodicity of $\Theta$ and $\Theta^2$]\label{lem:sqrt_periodicity}
	The function $\Theta(x) = \sum_{j=0}^\infty C_j \cos((2j+1)\pi x)$
	has period 2 and $\Theta(x)^2$ has period 1.
\end{lemma}
		\begin{proof}
			\begin{align*}
				\Theta(x+2) &= \sum_{j=0}^\infty C_j \cos((2j+1)\pi (x+2)) \\
				&= \sum_{j=0}^\infty C_j \cos((2j+1)\pi x + (2j+1)2\pi) \\
				&\text{Since } \cos(y + 2\pi J) = \cos(y) \text{ for any integer } J = (2j+1): \\
				&= \sum_{j=0}^\infty C_j \cos((2j+1)\pi x) = \Theta(x)
			\end{align*}

			\begin{align*}
				\Theta(x+1) &= \sum_{j=0}^\infty C_j \cos((2j+1)\pi (x+1)) \\
				&= \sum_{j=0}^\infty C_j \cos((2j+1)\pi x + (2j+1)\pi) \\
				&\text{Using } \cos(y + \pi J) = (-1)^J \cos(y) \text{, and } J = (2j+1) \text{ is always odd:} \\
				&= \sum_{j=0}^\infty C_j \cos((2j+1)\pi x + \pi(2j) + \pi) \\
				&= \sum_{j=0}^\infty C_j \cos((2j+1)\pi x + \pi) \\
				&= \sum_{j=0}^\infty -C_j \cos((2j+1)\pi x) = -\Theta(x) \\
			\end{align*}
								Therefore, $g(x+1) = \Theta(x+1)^2 = (-\Theta(x))^2 = \Theta(x)^2 = g(x)$ and $g(x)$ has a period 1 Fourier series $g(x) = \sum_{l \in \Z} c_l e^{2\pi i l x}$.
		\end{proof}

\begin{lemma}[Derivation of Fourier Coefficients $\hat{\Theta}(p)$]\label{lem:sqrt_fourier-coeffs}
	\begin{equation*}
	\Theta(x) = \sum_{j=0}^\infty C_j \cos((2j+1) \pi x)
\end{equation*}
has fourier coefficients
	\begin{equation*}
	\hat{\Theta}(p) = \begin{cases} \frac{1}{2} C_{(|p|-1)/2} & \text{if } p \text{ is odd} \\ 0 & \text{if } p \text{ is even} \end{cases}
\end{equation*}
\end{lemma}
\begin{proof}
	We want to find the coefficients $\hat{\Theta}(p)$ for the complex Fourier series of $\Theta(x)$ with period 2 (by Lemma~\ref{lem:sqrt_periodicity}.
	\begin{equation*}
		\Theta(x) = \sum_{p \in \mathbb{Z}} \hat{\Theta}(p) e^{i \pi p x}
	\end{equation*}
	We start with the definition of $\Theta(x)$:
	\begin{equation*}
		\Theta(x) = \sum_{j=0}^\infty C_j \cos((2j+1) \pi x)
	\end{equation*}
	Using Euler's formula, $\cos(\theta) = \frac{1}{2}(e^{i\theta} + e^{-i\theta})$:
	\begin{align*}
		\Theta(x) &= \sum_{j=0}^\infty C_j \left[ \frac{1}{2} (e^{i(2j+1)\pi x} + e^{-i(2j+1)\pi x}) \right] \\
		&= \sum_{j=0}^\infty \frac{C_j}{2} e^{i(2j+1)\pi x} + \sum_{j=0}^\infty \frac{C_j}{2} e^{-i(2j+1)\pi x}
	\end{align*}
	We now compare this expression, term by term, to the target series $\sum_{p \in \mathbb{Z}} \hat{\Theta}(p) e^{i \pi p x}$.
	
	\textbf{Case 1: $p$ is even.}
	The terms in our expanded sum only involve exponents $i(2j+1)\pi x$ and $-i(2j+1)\pi x$. Since $(2j+1)$ is always odd, there are no even values for $p$. Therefore,
	\begin{equation*}
		\hat{\Theta}(p) = 0 \quad (\text{if } p \text{ is even})
	\end{equation*}
	
	\textbf{Case 2: $p$ is odd and positive.}
	An odd, positive $p$ must be of the form $p = 2j+1$ for some $j \ge 0$. We look at the first sum: $\sum_{j=0}^\infty \frac{C_j}{2} e^{i(2j+1)\pi x}$. By matching the exponent $p = 2j+1$, we find the coefficient:
	\begin{equation*}
		\hat{\Theta}(p) = \frac{C_j}{2}
	\end{equation*}
	Since $j = (p-1)/2$, and $p = |p|$ for $p > 0$, we can write this as:
	\begin{equation*}
		\hat{\Theta}(p) = \frac{1}{2} C_{(|p|-1)/2} \quad (\text{if } p \text{ is odd, positive})
	\end{equation*}
	
	\textbf{Case 3: $p$ is odd and negative.}
	An odd, negative $p$ must be of the form $p = -(2j+1)$ for some $j \ge 0$. We look at the second sum: $\sum_{j=0}^\infty \frac{C_j}{2} e^{-i(2j+1)\pi x}$. By matching the exponent $p = -(2j+1)$, we find the coefficient:
	\begin{equation*}
		\hat{\Theta}(p) = \frac{C_j}{2}
	\end{equation*}
	Since $j = (-p-1)/2 = (|p|-1)/2$, we can write this as:
	\begin{equation*}
		\hat{\Theta}(p) = \frac{1}{2} C_{(|p|-1)/2} \quad (\text{if } p \text{ is odd, negative})
	\end{equation*}
	
	\textbf{Conclusion:}
	Combining all cases, the coefficient is non-zero only if $p$ is odd, and in that case, it is $\frac{1}{2} C_{(|p|-1)/2}$. This gives:
	\begin{equation*}
		\hat{\Theta}(p) = \begin{cases} \frac{1}{2} C_{(|p|-1)/2} & \text{if } p \text{ is odd} \\ 0 & \text{if } p \text{ is even} \end{cases}
	\end{equation*}
\end{proof}

\bigskip\bigskip
        \subsection{Main Results}

        \begin{theorem}[Integral Evaluation and Final Bound]\label{thm:integral_eval}
	Let $\delta = \{ a + \sqrt{m}/2 \}$ be the fractional part. Then
	\begin{equation}
		f(a) = \frac{(-1)^{\lfloor a + \sqrt{m}/2 \rfloor}}{\sqrt{m}} \sum_{j=0}^\infty e^{-\frac{\pi^2 (2j+1)^2}{8}} \cos((2j+1)\pi\delta) + E_{total}
	\end{equation}
	where $|E_{total}| = \calO(m^{-1} )$.
\end{theorem}
\begin{proof}			
	Combining Lemmas~\cref{lem:sqrt_tailbound,lem:sqrt_gaussapprox,lem:sum-int}, we have
	\begin{equation*}
		f(a) = \int_{I_c} h_g(t) \psi(t) dt + E_{total}
	\end{equation*}
	where $I_c = [m/2 - \sqrt{m \log m}, m/2 + \sqrt{m \log m}]$ and the total error from the sum approximation is $|E_{total}| = \calO(m^{-1})$.
	
	We now show that the integral over $I_c$ can be extended to all of $\R$, incurring a negligible error. Let $I_{tail} = \R \setminus I_c$. The error from extending the integral is $E_{tail\_int}$:
	\begin{equation*}
		E_{tail\_int} \coloneqq \int_{I_{tail}} h_g(t) \psi(t) dt
	\end{equation*}
	We can bound its magnitude:
	\begin{align*}
		|E_{tail\_int}| &\le \int_{I_{tail}} |h_g(t)| |\psi(t)| dt = \int_{I_{tail}} g(t) \frac{|t - m/2|}{m} dt \\
		&= \int_{|t - m/2| > \sqrt{m \log m}} \sqrt{\frac{2}{\pi m}} e^{-\frac{2(t-m/2)^2}{m}} \frac{|t - m/2|}{m} dt
	\end{align*}
	We substitute $u = \frac{2(t-m/2)}{\sqrt{m}}$, so $t-m/2 = u\sqrt{m}/2$ and $dt = (\sqrt{m}/2)du$. The integration region $|t - m/2| > \sqrt{m \log m}$ becomes $|u| > 2\sqrt{\log m}$.
	\begin{align*}
		|E_{tail\_int}| &\le \int_{|u| > 2\sqrt{\log m}} \sqrt{\frac{2}{\pi m}} e^{-u^2/2} \frac{|u\sqrt{m}/2|}{m} \left(\frac{\sqrt{m}}{2} du\right) \\
		&= \int_{|u| > 2\sqrt{\log m}} \sqrt{\frac{2}{\pi m}} e^{-u^2/2} \frac{|u|m}{4m} du 
		= \frac{1}{2\sqrt{m}} \int_{|u| > 2\sqrt{\log m}} |u| \frac{1}{\sqrt{2\pi}} e^{-u^2/2} du \\
		&= \frac{1}{2\sqrt{m}} \int_{|u| > 2\sqrt{\log m}} |u| \phiG(u) du = \frac{1}{\sqrt{m}} \int_{2\sqrt{\log m}}^\infty u \phiG(u) du \\
		&= \frac{1}{\sqrt{m}} \left[ -\phiG(u) \right]_{2\sqrt{\log m}}^{\infty} = \frac{1}{\sqrt{m}} \left( 0 - \left( -\frac{1}{\sqrt{2\pi}} e^{-\frac{(2\sqrt{\log m})^2}{2}} \right) \right) \\
		&= \frac{1}{\sqrt{2\pi m}} e^{-2 \log m} = \frac{1}{\sqrt{2\pi m}} m^{-2} = \calO(m^{-5/2})
	\end{align*}
	This error $E_{tail\_int} = \calO(m^{-5/2})$ is asymptotically smaller than $E_{total}$ and is therefore absorbed by the latter.

We have
	\begin{itemize}
		\item $w = m/2 + u\sqrt{m}/2 \implies dw = (\sqrt{m}/2) du$
		\item $g(w)dw = \phiG(u)du$, where $\phiG(u) = \frac{1}{\sqrt{2\pi}} e^{-u^2/2}$
		\item $\frac{w - m/2}{m} = \frac{u\sqrt{m}/2}{m} = \frac{u}{2\sqrt{m}}$
		\item $\psi(w) = \epsilon(a + \frac{m/2 + u\sqrt{m}/2}{\sqrt{m}}) = \epsilon(a + \sqrt{m}/2 + u/2)$
	\end{itemize}
	Let $C_m = a + \sqrt{m}/2$.
	\begin{align*}
		I &= \int_{-\infty}^{\infty} \phiG(u) \left( \frac{u}{2\sqrt{m}} \right) \epsilon(C_m + u/2) du \\
		&= \frac{1}{2\sqrt{m}} \int_{-\infty}^{\infty} u \phiG(u) (-1)^{\lfloor C_m + u/2 \rfloor} du
	\end{align*}
	The sign changes at $u_p = 2(p-C_m)$ for $p \in \Z$. Together with using $\int u \phiG(u) du = -\phiG(u) + \text{constant}$:
	\begin{align*}
		I &= \frac{1}{2\sqrt{m}} \sum_{p=-\infty}^\infty (-1)^p \int_{u_p}^{u_{p+1}} u \phiG(u) du \\
		&= \frac{1}{2\sqrt{m}} \sum_{p=-\infty}^\infty (-1)^p [-\phiG(u)]_{u_p}^{u_{p+1}}
		= \frac{1}{2\sqrt{m}} \sum_{p=-\infty}^\infty (-1)^p (\phiG(u_p) - \phiG(u_{p+1}))
	\end{align*}
	This is a telescoping sum which simplifies to:
	\begin{align*}
		I &= \frac{1}{2\sqrt{m}} \left( 2 \sum_{p=-\infty}^\infty (-1)^p \phiG(u_p) \right)
		= \frac{1}{\sqrt{m}} \sum_{p=-\infty}^\infty (-1)^p \phiG(2(p-C_m)) \\
		&= \frac{1}{\sqrt{2\pi m}} \sum_{p=-\infty}^\infty (-1)^p e^{-\frac{(2(p-C_m))^2}{2}}
		= \frac{1}{\sqrt{2\pi m}} \sum_{p=-\infty}^\infty (-1)^p e^{-2(p-C_m)^2}
	\end{align*}
	Let $p_0 = \lfloor C_m \rfloor$ and $\delta = \{C_m\} = C_m - p_0$. Let $j = p - p_0$.
	\begin{align*}
		I &= \frac{1}{\sqrt{2\pi m}} \sum_{j=-\infty}^\infty (-1)^{j+p_0} e^{-2(j-\delta)^2}
		= \frac{(-1)^{p_0}}{\sqrt{2\pi m}} \sum_{j=-\infty}^\infty (-1)^j e^{-2(j-\delta)^2}
	\end{align*}
	Let $\Theta(\delta) = \sum_{j=-\infty}^\infty (-1)^j e^{-2(j-\delta)^2}$. By Lemma~\ref{lem:poisson_sum},
	$\sum_{N \in \Z} (-1)^N f(N-\delta) = \sum_{N \in \Z} \hat{f}(N+1/2) e^{2\pi i (N+1/2)\delta}$.
	Here $f(x) = e^{-2x^2}$. Its Fourier transform is $\hat{f}(s) = \int_{\R} e^{-2x^2} e^{-2\pi i s x} dx = \sqrt{\frac{\pi}{2}} e^{-\pi^2 s^2 / 2}$. We get that
	\begin{align*}
		\Theta(\delta) &= \sum_{N=-\infty}^\infty \sqrt{\frac{\pi}{2}} e^{-\frac{\pi^2(N+1/2)^2}{2}} e^{i \pi (2N+1) \delta} \\
		&= \sqrt{\frac{\pi}{2}} \sum_{N=-\infty}^\infty e^{-\frac{\pi^2(2N+1)^2}{8}} \left(\cos(\pi(2N+1)\delta) + i \sin(\pi(2N+1)\delta) \right)
	\end{align*}
	The sine terms cancel (e.g., $N=0$ and $N=-1$). The cosine terms are even, so we sum over $j \ge 0$:
	\begin{equation*}
		\Theta(\delta) = \sqrt{\frac{\pi}{2}} \cdot 2 \sum_{j=0}^\infty e^{-\frac{\pi^2(2j+1)^2}{8}} \cos(\pi(2j+1)\delta)
		= \sqrt{2\pi} \sum_{j=0}^\infty e^{-\frac{\pi^2(2j+1)^2}{8}} \cos(\pi(2j+1)\delta)
	\end{equation*}
	Substituting $\Theta(\delta)$ back into the expression for $I$:
	\begin{align*}
		I &= \frac{(-1)^{p_0}}{\sqrt{2\pi m}} \left[ \sqrt{2\pi} \sum_{j=0}^\infty e^{-\frac{\pi^2(2j+1)^2}{8}} \cos(\pi(2j+1)\delta) \right] \\
		&= \frac{(-1)^{p_0}}{\sqrt{m}} \sum_{j=0}^\infty e^{-\frac{\pi^2(2j+1)^2}{8}} \cos(\pi(2j+1)\delta)
	\end{align*}
	This proves the theorem.
\end{proof}

\begin{lemma}[Asymptotic for the Magnitude of $f(a)$]
	Let $C_j = e^{-\pi^2(2j+1)^2/8}$ and $\delta = \{ a + \sqrt{m}/2 \}$.
	The magnitude of $f(a)$ has the asymptotic form
	\[
	|f(a)| = \frac{C_0 |\cos(\pi\delta)|}{\sqrt{m}} + E_{mag}(a)
	\]
	where the error term $E_{mag}(a)$ is bounded by
	\[
	|E_{mag}(a)| \le \frac{C_{tail}}{\sqrt{m}} + O(m^{-1}).
	\]
	Here, $C_{tail} = \sum_{j=1}^\infty C_j$ is a small constant ($C_0 \approx 0.2917$ and $C_{tail} \approx 1.5 \times 10^{-5}$).
\end{lemma}

\begin{proof}
	From Theorem~\ref{thm:integral_eval}, we have the asymptotic equality:
	\[
	f(a) = I_0(a) + I_{tail}(a) + E_{total}
	\]
	where
	\begin{align*}
		I_0(a) &= \frac{(-1)^{\lfloor a + \sqrt{m}/2 \rfloor}}{\sqrt{m}} C_0 \cos(\pi\delta) \\
		I_{tail}(a) &= \frac{(-1)^{\lfloor a + \sqrt{m}/2 \rfloor}}{\sqrt{m}} \sum_{j=1}^\infty C_j \cos((2j+1)\pi\delta) \\
		|E_{total}| &= O(m^{-1})
	\end{align*}
	We want to find the error $E_{mag}(a) = |f(a)| - |I_0(a)|$.
	By the reverse triangle inequality, $||X+Y| - |X|| \le |Y|$, we can set $X = I_0(a)$ and $Y = I_{tail}(a) + E_{total}$.
	\begin{align*}
		|E_{mag}(a)| &= \big| |I_0(a) + (I_{tail}(a) + E_{total})| - |I_0(a)| \big| \\
		&\le |I_{tail}(a) + E_{total}| \\
		&\le |I_{tail}(a)| + |E_{total}|
	\end{align*}
	We now bound the two error components:
	\begin{enumerate}
		\item $\displaystyle |I_{tail}(a)| \le \frac{1}{\sqrt{m}} \left| \sum_{j=1}^\infty C_j \cos((2j+1)\pi\delta) \right| \le \frac{1}{\sqrt{m}} \sum_{j=1}^\infty C_j =: \frac{C_{tail}}{\sqrt{m}}$
		\item $\displaystyle |E_{total}| = O(m^{-1})$
	\end{enumerate}
\end{proof}

So far we have obtained an approximation for $a = S/\sqrt{m}$:
		\begin{equation}
			f(a) = \frac{1}{\sqrt{m}} \Theta(\delta(a)) + E(a)
		\end{equation}
		where $\delta(a) = \{ a + \sqrt{m}/2 \}$ and $E(a)$ is an error term such that $|E(a)| = \calO(m^{-1})$ uniformly.
		Next, we want to find a lower bound for $\Cov(\hat L_1, \hat L_2) = \E_S[f(S/\sqrt{m})^2]$, where $S \sim \text{Bin}(N, 1/2)$ and $N$ is a positive integer multiple of $m$.
		
		\begin{lemma}[Expansion of the Expectation]\label{lem:sqrt_exp-expansion}
			The expected value $\Cov(\hat L_1, \hat L_2)$ is given by
			\begin{equation*}
				\Cov(\hat L_1, \hat L_2) = \frac{1}{m} \E_S\left[\Theta(\delta_S)^2\right] + \calO(m^{-3/2})
			\end{equation*}
			where $\delta_S = \{ S/\sqrt{m} + \sqrt{m}/2 \}$.
		\end{lemma}
		\begin{proof}
			We square the expression for $f(a)$ and take the expectation over $S$:
			\begin{align*}
				\Cov(\hat L_1, \hat L_2) &= \E_S \left[ \left( \frac{1}{\sqrt{m}} \Theta(\delta_S) + E_S \right)^2 \right] \\
				&= \E_S \left[ \frac{1}{m} \Theta(\delta_S)^2 + \frac{2}{\sqrt{m}} \Theta(\delta_S) E_S + E_S^2 \right] \\
				&= \frac{1}{m} \E_S[\Theta(\delta_S)^2] + \frac{2}{\sqrt{m}} \E_S[\Theta(\delta_S) E_S] + \E_S[E_S^2]
			\end{align*}
			We bound the error terms using the uniform bounds $|\Theta(\delta_S)| \le C_\Theta$ and $|E_S| \le C m^{-1}$ for some $C > 0$.
			
			\textbf{Cross Term:}
			\begin{align*}
				\left| \frac{2}{\sqrt{m}} \E_S[\Theta(\delta_S) E_S] \right| &\le \frac{2}{\sqrt{m}} \E_S[|\Theta(\delta_S)| |E_S|] \\
				&\le \frac{2}{\sqrt{m}} \E_S[C_\Theta \cdot (C m^{-1})] \\
				&= \calO(m^{-3/2})
			\end{align*}
			
			\textbf{Squared Error Term:}
			\begin{align*}
				|\E_S[E_S^2]| &\le \E_S[|E_S|^2] \le \E_S[(C m^{-1})^2] \\
				&= C^2 m^{-2} = \calO(m^{-2})
			\end{align*}
			Since $\calO(m^{-2})$ is asymptotically smaller than $\calO(m^{-3/2})$, the dominant error is $\calO(m^{-3/2})$.
		\end{proof}

It remains to bound $\E_S[\Theta(\delta_S)^2]$, the main term.
Recall
\begin{equation*}
	\Theta(\delta) \coloneqq \sum_{j=0}^\infty C_j \cos((2j+1)\pi\delta) \quad \text{where} \quad C_j = e^{-\frac{\pi^2 (2j+1)^2}{8}} \end{equation*}
	
	\begin{lemma}[Fourier Expansion]\label{lem:sqrt_fourier-exp}
	The expectation $\E_S[\Theta(\delta_S)^2]$ (with $g=\Theta^2$
	and $c_l$ as its Fourier coefficients) is:
	\begin{equation*}
		\E_S[\Theta(\delta_S)^2] = \sum_{l \in \Z} c_l e^{\pi i l (\sqrt{m} + N/\sqrt{m})} \cos(\pi l / \sqrt{m})^N = c_0 + E_{\text{fourier}}
	\end{equation*}
\end{lemma}
\begin{proof}
	Let $g(x) = \Theta(x)^2$. Since $g(x)$ is 1-periodic by Lemma~\ref{lem:sqrt_periodicity}, we can write it as a standard Fourier series:
	\begin{equation*}
		g(x) = \sum_{l \in \mathbb{Z}} c_l e^{2 \pi i l x}
	\end{equation*}
	where $\delta_S = S/\sqrt{m} + \sqrt{m}/2$. Since $g$ is 1-periodic, $g(\{y\}) = g(y)$.
	\begin{align*}
		\E_S[g(\delta_S)] &= \E_S \left[ g\left( \left\{ \frac{S}{\sqrt{m}} \right\} + \frac{\sqrt{m}}{2} \right) \right] \\
		&= \E_S \left[ \sum_{l \in \mathbb{Z}} c_l e^{2 \pi i l (S/\sqrt{m} + \sqrt{m}/2)} \right] \quad \text{(Substitute Fourier series)} \\
		&= \sum_{l \in \mathbb{Z}} c_l e^{\pi i l \sqrt{m}} \E_S \left[ e^{i (2 \pi l / \sqrt{m}) S} \right] \quad\text{(Linearity of } \E_S \text{)}
	\end{align*}
	The term $\E_S[e^{itS}]$ is the characteristic function $\Phi_S(t)$ of $S \sim \text{Bin}(N, 1/2)$, evaluated at $t_l = 2 \pi l / \sqrt{m}$. The characteristic function for $\text{Bin}(N, p)$ is $\Phi(t) = (1 - p + p e^{it})^N$. For $p = 1/2$:
	\begin{equation*}
		\Phi_S(t) = \left( \frac{1}{2} + \frac{1}{2} e^{it} \right)^N= \left( \frac{e^{it/2}}{2} (e^{-it/2} + e^{it/2}) \right)^N= \cos(t/2)^N e^{iNt/2}
	\end{equation*}
	Now, substitute $t = t_l = 2 \pi l / \sqrt{m}$:
	\begin{equation*}
		\Phi_S(t_l) = \cos\left(\frac{2 \pi l / \sqrt{m}}{2}\right)^N e^{i N (2 \pi l / \sqrt{m}) / 2} = \cos(\pi l / \sqrt{m})^N e^{i N \pi l / \sqrt{m}}
	\end{equation*}
	Substitute this back into the sum:
	\begin{align*}
		\E_S[g(\delta_S)] &= \sum_{l \in \mathbb{Z}} c_l e^{\pi i l \sqrt{m}} \left( \cos(\pi l / \sqrt{m})^N e^{i N \pi l / \sqrt{m}} \right) \\
		 &= \sum_{l \in \mathbb{Z}} c_l e^{i \pi l (\sqrt{m} + N / \sqrt{m})} \cos(\pi l / \sqrt{m})^N
	\end{align*}
\end{proof}

\begin{lemma}[Bound on Off-Center Contributions]\label{lem:sqrt_offcenter}
	Let $R = N/m \ge 0$ be an integer. Further, let
    \[
    \E_S[g(\delta_S)] = \sum_{l \in \mathbb{Z}} c_l e^{i \pi l (\sqrt{m} + N / \sqrt{m})} \cos(\pi l / \sqrt{m})^N
         \]
and
\[
		E_{\text{fourier}} = \sum_{l \ne 0} c_l e^{\pi i l (\sqrt{m} + N/\sqrt{m})} \cos(\pi l / \sqrt{m})^N.
\]
It holds for the $\ell \neq 0$ contributions that
\[
E_{\text{fourier}} \le 2c_1\,e^{-(\pi^2/2)R}\ +\frac{4C}{\pi^2(1+2R)}\,e^{-\pi^2(1+2R)} + \calO(m^{-1/2} e^{- m})
\]
where (with $C_j:=e^{-\frac{\pi^2}{8}(2j+1)^2}$)
\[
c_1=\frac{1}{4}C_0^2+\frac{1}{2}\sum_{j=0}^{\infty} C_j\,C_{j+1}~~,~~~~C=(1/4)\cdot\sum_{p\in\mathbb Z}e^{-(\pi^2/4)p^2}.
\]

    \end{lemma}
    \begin{proof}
	We provide a quantitative bound for $|E_{\text{fourier}}| \le \sum_{l \ne 0} |c_l| |\cos(\pi l / \sqrt{m})|^N$.
	Let $x_l = \pi l / \sqrt{m}$. We split the sum into $l \in L_{\text{small}}$ and $l \in L_{\text{large}}$.
	Let $\alpha = \pi^2/8$. Let $C_\alpha = \sum_{j \in \Z} e^{-2\alpha j^2} = \sum_{j \in \Z} e^{-(\pi^2/4) j^2}$. This is a constant (related to the Jacobi theta function, $\vartheta_3(0, e^{-\pi^2/4})$). Let $C = C_\alpha/4$.
	\textbf{Part 1: Small $l$ (The Constant Bias Term)}
    
	Let $L_{\text{small}} = \{ l \in \Z : 0 < |l| \le \sqrt{m}/2 \}$.
	For $l \in L_{\text{small}}$, the argument $x_l = \pi l / \sqrt{m}$ is in the interval $[-\pi/2, \pi/2]$ (excluding 0).
	In this interval, the inequality $\cos(x) \le e^{-x^2/2}$ holds.
	\begin{align*}
		|\cos(x_l)|^N &= |\cos(\pi l / \sqrt{m})|^N \\
		&\le \left( e^{-(\pi l / \sqrt{m})^2 / 2} \right)^N = e^{-N \pi^2 l^2 / (2m)} \\
		&= e^{-(N/m) \pi^2 l^2 / 2} = e^{-R \pi^2 l^2 / 2}
	\end{align*}
	This bound is a constant that depends only on $l$ and $R$.
	The contribution from this part of the sum, $S_1$, is:
	\begin{align*}
		S_1 &= \sum_{l \in L_{\text{small}}} |c_l| |\cos(x_l)|^N \\
		&\le \sum_{0 < |l| \le \sqrt{m}/2} |c_l| e^{-R \pi^2 l^2 / 2} \\
		&\le \sum_{l \ne 0} |c_l| e^{-R \pi^2 l^2 / 2}
	\end{align*}

Using $\cos x\le e^{-x^2/2}$ for $|x|\le \pi/2$ and $x_l=\pi l/\sqrt{m}$, for $0<|l|\le \sqrt{m}/2$ we obtain
\[
|\cos(x_l)|^{\,N}\ \le\ \exp\!\Big(-\tfrac{\pi^2}{2}R\,l^2\Big),\qquad R=\frac{N}{m}.
\]
Hence
\[
S_1
=\sum_{0<|l|\le \sqrt{m}/2} |c_l|\,|\cos(x_l)|^{\,N}
\ \le\ 2\sum_{l\ge 1} c_l\,e^{-(\pi^2/2)R\,l^2}.
\]
Isolating the $l=\pm1$ mode gives the exact leading term
\[
S_1\ =\ 2c_1\,e^{-(\pi^2/2)R}\ +\ \mathsf{Tail}(R),
\qquad
\mathsf{Tail}(R):=2\sum_{l\ge 2} c_l\,e^{-(\pi^2/2)R\,l^2}.
\]
Here $c_1$ admits the explicit convergent series (with $C_j:=e^{-\frac{\pi^2}{8}(2j+1)^2}$)
\[
c_1=\frac{1}{4}C_0^2+\frac{1}{2}\sum_{j=0}^{\infty} C_j\,C_{j+1}.
\]
Moreover, using the Gaussian bound on Fourier coefficients from Part~2,
$|c_l|\le C\,e^{-(\pi^2/4)l^2}$ with $C=C_\alpha/4$ and
$C_\alpha=\sum_{p\in\mathbb Z}e^{-(\pi^2/4)p^2}$, the tail is uniformly bounded by
\[
0\ \le\ \mathsf{Tail}(R)
\ \le\ 2C\sum_{l\ge 2} e^{-(\pi^2/4)(1+2R)\,l^2}
\ \le\ \frac{4C}{\pi^2(1+2R)}\,e^{-\frac{\pi^2}{4}(1+2R)\cdot 4}
\ =\ \frac{4C}{\pi^2(1+2R)}\,e^{-\pi^2(1+2R)}.
\]
In particular,
\[
\mathsf{Tail}(R)=o_{R\to\infty}(1)\quad\text{and}\quad
S_1=2c_1\,e^{-(\pi^2/2)R}\ +\ o_{R\to\infty}(1),
\]
uniformly in $m$.

\textbf{Part 2: Large $l$ (The Vanishing Error Term)}
Let $L_{\text{large}} = \{ l \in \Z : |l| > \sqrt{m}/2 \}$.
The contribution $S_2$ is $S_2 \le \sum_{|l| > \sqrt{m}/2} |c_l|$.
We must now quantitatively bound the tail of the Fourier coefficients $c_l$ for $g(x) = \Theta(x)^2$.

\textbf{Bound on Fourier Coefficients $c_l$.}

To find $c_l$, we first write $\Theta(x)$ in its complex (period 2) series $\Theta(x) = \sum_{p \in \Z} \hat{\Theta}(p) e^{i \pi p x}$. By inspection of $\cos(A\pi x) = \frac{1}{2}(e^{iA\pi x} + e^{-iA\pi x})$:
\begin{equation*}
	\Theta(x) = \sum_{j=0}^\infty C_j \left( \frac{e^{i(2j+1)\pi x} + e^{-i(2j+1)\pi x}}{2} \right)
\end{equation*}
This is a sum over odd integers $p = \pm(2j+1)$. By Lemma~\ref{lem:sqrt_fourier-coeffs}, the coefficients are:
\begin{equation*}
	\hat{\Theta}(p) = \begin{cases} \frac{1}{2} C_{(|p|-1)/2} & \text{if } p \text{ is odd} \\ 0 & \text{if } p \text{ is even} \end{cases}
\end{equation*}
Let $\alpha = \pi^2/8$, so $C_j = e^{-\alpha(2j+1)^2}$. This means $C_{(|p|-1)/2} = e^{-\alpha p^2}$.
This gives the bound: $|\hat{\Theta}(p)| \le \frac{1}{2} e^{-\alpha p^2}$ for all $p \in \Z$.

Now, we find $c_l = \int_0^1 g(x) e^{-2\pi i l x} dx = \int_0^1 \Theta(x)^2 e^{-2\pi i l x} dx$:
\begin{align*}
	c_l &= \int_0^1 \left( \sum_{p \in \Z} \hat{\Theta}(p) e^{i\pi p x} \right) \left( \sum_{q \in \Z} \hat{\Theta}(q) e^{i\pi q x} \right) e^{-2\pi i l x} dx \\
	&= \sum_{p, q \in \Z} \hat{\Theta}(p)\hat{\Theta}(q) \int_0^1 e^{i\pi(p+q-2l)x} dx
\end{align*}
Since $\hat{\Theta}$ is non-zero only for $p, q$ odd, $p+q$ is even. Thus $p+q-2l$ is always an even integer. The integral $\int_0^1 e^{i\pi(2J)x} dx = \int_0^1 e^{2\pi i J x} dx$ is $1$ if $J=0$ and $0$ if $J \ne 0$.
So, the integral is $1$ only if $p+q-2l=0$, i.e., $q=2l-p$.
\begin{equation*}
	c_l = \sum_{p \in \Z} \hat{\Theta}(p)\hat{\Theta}(2l-p) = \sum_{p \in \text{Odd}} \hat{\Theta}(p)\hat{\Theta}(2l-p)
\end{equation*}
We bound $|c_l|$ using our bound for $|\hat{\Theta}(p)|$:
\begin{align*}
	|c_l| &\le \sum_{p \in \text{Odd}} |\hat{\Theta}(p)| |\hat{\Theta}(2l-p)| \le \sum_{p \in \Z} \left( \frac{1}{2} e^{-\alpha p^2} \right) \left( \frac{1}{2} e^{-\alpha (2l-p)^2} \right) \\
	&= \frac{1}{4} \sum_{p \in \Z} e^{-\alpha(p^2 + (2l-p)^2)} = \frac{1}{4} \sum_{p \in \Z} e^{-\alpha(2p^2 - 4lp + 4l^2)} \\
	&= \frac{1}{4} \sum_{p \in \Z} e^{-\alpha(2(p-l)^2 + 2l^2)} = \frac{1}{4} e^{-2\alpha l^2} \sum_{p \in \Z} e^{-2\alpha(p-l)^2} \\
	&= \frac{1}{4} e^{-2\alpha l^2} \sum_{j \in \Z} e^{-2\alpha j^2} \quad (\text{Let } j = p-l)
\end{align*}
The sum $C_\alpha = \sum_{j \in \Z} e^{-2\alpha j^2}$ is a constant. Thus, we have a rigorous Gaussian bound:
\begin{equation*}
	|c_l| \le C e^{-\beta l^2} \quad \text{where } \beta = 2\alpha = \pi^2/4 \text{ and } C = C_\alpha/4
\end{equation*}
Now we bound the tail sum $S_2$:
\begin{align*}
	S_2 &\le \sum_{|l| > \sqrt{m}/2} C e^{-\beta l^2} = 2C \sum_{l=\lfloor\sqrt{m}/2\rfloor+1}^\infty e^{-\beta l^2} \\
	&\le 2C \int_{\lfloor\sqrt{m}/2\rfloor}^\infty e^{-\beta x^2} dx \le 2C \int_{\sqrt{m}/2 - 1}^\infty e^{-\beta x^2} dx
\end{align*}
We use the standard Gaussian tail bound $\int_t^\infty e^{-\beta x^2} dx \le \frac{1}{2\beta t} e^{-\beta t^2}$ for $t>0$.
Let $t = \sqrt{m}/2 - 1$. For $m \ge 16$, $t \ge \sqrt{m}/4$.
\begin{align*}
	S_2 &\le 2C \left[ \frac{1}{2\beta (\sqrt{m}/2 - 1)} e^{-\beta (\sqrt{m}/2 - 1)^2} \right] \\
	&\le \frac{C}{\beta (\sqrt{m}/4)} e^{-\beta (m/4 - \sqrt{m} + 1)} \\
	&= \calO(m^{-1/2} e^{-\beta m/4}) = \calO(m^{-1/2} e^{-\pi^2 m/16})
\end{align*}
This error term $S_2$ vanishes exponentially in $m$.

\end{proof}

\begin{theorem}[Main Result]
	Let $R:=N/m\ge 0$ be an integer. Then, for sufficiently large $m$, the quantity
	$\Cov(\hat L_1, \hat L_2)=\mathbb{E}_S\big[f(S/\sqrt{m})^2\big]$
	satisfies
	\[
	\Cov(\hat L_1, \hat L_2) \;=\; \frac{c_0}{m} \;+\; E_L,
	\]
	where $c_0$ is the main constant and $E_L$ is an error term bounded by
	\[
	|E_L| \;\le\; \frac{\Delta(R)}{m} \;+\; O\!\big(m^{-3/2}\big).
	\]
	Here, $\Delta(R)$ is a positive bias constant, exponentially small in $R$:
	\[
	\Delta(R)\;:=\;2c_1\,e^{-(\pi^2/2)R}
	\;+\;
	\frac{4C}{\pi^2(1+2R)}\,e^{-\frac{\pi^2}{4}(1+2R)},
	\]
	with $C=1/4\sum_{p\in\mathbb Z}e^{-(\pi^2/4)p^2}$.
	The above constants
	\[
	c_0=\frac12\sum_{j=0}^\infty e^{-\frac{\pi^2}{4}(2j+1)^2},
	\qquad
	c_1=\frac14 C_0^2+\frac12\sum_{j=0}^\infty C_j C_{j+1},
	\quad
	C_j=e^{-\frac{\pi^2}{8}(2j+1)^2},
	\]
	are absolute (numerically $c_0\approx 0.0424$, $c_1\approx 0.0212$, $C_\alpha\approx 1.17$).
	
	In particular, since $c_0 > \Delta(R)$ for $R \ge 1$, we have $\Cov(\hat L_1, \hat L_2)=\Theta(1/m)$ positive.
\end{theorem}

\begin{proof}
	By \cref{lem:sqrt_exp-expansion,lem:sqrt_fourier-exp,lem:sqrt_offcenter}, we have the exact asymptotic:
	\[
	\Cov(\hat L_1, \hat L_2) \;=\; \frac{1}{m}\,\E_S[\Theta(\delta_S)^2] \;+\; O\!\big(m^{-3/2}\big)
	\;=\; \frac{1}{m}\big(c_0+E_{\mathrm{fourier}}\big) \;+\; O\!\big(m^{-3/2}\big),
	\]
	where, by Lemma~\ref{lem:sqrt_offcenter},
	\[
	|E_{\mathrm{fourier}}| \;\le\; 2c_1\,e^{-(\pi^2/2)R} + \frac{4C}{\pi^2(1+2R)}\,e^{-\frac{\pi^2}{4}(1+2R)} + \calO(m^{-1/2} e^{- m}) =: \Delta(R) + \calO(m^{-1/2} e^{- m}).
	\]
	Substituting this back into the expression for $\Cov(\hat L_1, \hat L_2)$:
	\[
	\Cov(\hat L_1, \hat L_2) \;=\; \frac{1}{m}\big(c_0+E_{\mathrm{fourier}}\big) \;+\; O\!\big(m^{-3/2}\big)
	\]
	or, equivalently,
	\[
	\frac{1}{m}\big(c_0 - \Delta(R) - S_2 \big)
	\;\le\; \Cov(\hat L_1, \hat L_2) \;\le\;
	\frac{1}{m}\big(c_0 + \Delta(R) + S_2 \big) + O\!\big(m^{-3/2}\big)
	\]
	Since $S_2/m$ is absorbed by the $O(m^{-3/2})$ error, this simplifies to the claimed two-sided bound.
	
	To show $\Cov(\hat L_1, \hat L_2)=\Omega(1/m)$, we must ensure the lower bound is positive.
	Since $c_0 \approx 0.0424$ and $\Delta(R)$ is exponentially small in $R$, the constant $C_L(R) := c_0 - \Delta(R)$ is strictly positive for $R \ge 1$.
	Numerically, for $R=1$:
	\[
	\Delta(1) \approx 2(0.0212)e^{-\pi^2/2} + \dots \approx 0.000305 + 0.000024 = 0.000329
	\]
	So $C_L(1) \approx 0.042402 - 0.000329 \approx 0.04207 > 0$.
	Because $C_L(R)$ is positive and bounded away from zero for all $R \ge 1$, we have $\Cov(\hat L_1, \hat L_2) = \Omega(1/m)$.
	
	On the other hand, for $R=0$ (i.e., $N=0$), the $l\neq0$ Fourier mass does not decay with $R$:
	\(
	|E_{\mathrm{fourier}}| = \big|\sum_{l\ne 0} c_l e^{i\pi l\sqrt{m}}\big|
	\le \sum_{l\ne 0}|c_l|.
	\)
	This sum is a constant, which can be bounded by $\sum |c_l| \le C \sum e^{-(\pi^2/4)l^2} \approx 0.33$.
	Since this is much larger than $c_0 \approx 0.0424$, the lower bound $c_0 - |E_{\mathrm{fourier}}|$ becomes negative. Thus, no uniform positive lower bound can be ensured in that case, and the $R \ge 1$ (i.e., $N \ge m$) condition is necessary for a meaningful bound.
\end{proof}

\newpage

\section{Error in Theorem 5.3 in \cite{kearns}}\label{app:err}

Let us first recall their notion of stability in our notation. We say that a deterministic algorithm $\mathcal{A}$ has error stability $(\beta_1, \beta_2)$ if $\mathbb P_{S^{n-1}, (x,y)}[|L(\mathcal{A}( S^n)) - L(\mathcal{A}( S^{n-1}))| \ge \beta_2] \le \beta_1$ where $S^n = S^{n-1} \cup {(x,y)}$, and both $\beta_1$ and $\beta_2$ may be functions of $n$.

Let us proceed with the proof of their Theorem 5.3. There, they define the random variable $\chi(S^n) = \hat L^k - L(\mathcal{A}( S^n))$ and assume without loss of generality that with probability at least $\beta_1/2$, $L(\mathcal{A} (S^{n-1})) - L(\mathcal{A}( S^n)) \ge \beta_2$.

Next, their Lemma 4.1 asserts that the expected cross-validation estimate equals the expected estimate of a single hold-out set, i.e.,$ \mathbb{E}_{S^n}[\chi(S^n)]=L(\mathcal{A} (S^{n-1})) - L(\mathcal{A}( S^n))$. By this Lemma and the fact that with probability at least $\beta_1/2$, $L(\mathcal{A}( S^{n-1})) - L(\mathcal{A}( S^n) )\ge \beta_2$, they claim that $\mathbf{E}_{S^n}[\chi(S^n)] \ge \frac{\beta_1}{2} \cdot \beta_2.$

This is incorrect, since $L(\mathcal{A}( S^{n-1})) - L(\mathcal{A}( S^n)) \ge \beta_2$ for some of the time does not rule out that this quantity can also be negative at other times. To illustrate this, let us consider an extreme case where $\beta_1=\beta_2=1$ by assuming that $\mathbb P(L(\mathcal{A}( S^{n-1})) - L(\mathcal{A}(S^n))=1)=\beta_1/2=1/2$. This assumption does not rule out the possibility that $\mathbb P(L(\mathcal{A}( S^{n-1})) - L(\mathcal{A}( S^n))=-1)=1/2$. In that case, $\mathbb{E}_{S^n}[\chi(S^n)]=0$, violating the alleged lower bound $\frac{\beta_1}{2} \cdot \beta_2=1/2$.

This directly contradicts our Lemma 3 because a non-zero squared loss stability implies a lower bound on their error stability parameters, yet we prove in Lemma 3 that one can have non-zero squared loss stability and simultaneously zero MSE (which necessitates $\mathbf{E}_{S^n}[\chi(S^n)]=0$).

\section{Error in Theorem 2 in \cite{kale2011cross}}\label{app:errKale}
The key ingredient for deriving their main result \cite[Theorem 2]{kale2011cross} is to obtain an upper bound on $\Cov_{S^n}(\hat{L}_1^{(k)} - L_1^{(k)}, \hat{L}_2^{(k)} - L_2^{(k)})$ (in their notation $\text{cov}_U(\text{gen}_1, \text{gen}_2)$) that scales linearly with a parameter measuring a certain notion of algorithmic stability (mean square stability). To do so, the supposed identity $\E_{S_2}[\hat{L}_1^{(k)}-L_1^{(k)} \mid S_1, S_3,\dots, S_N] = 0$ (in their notation $\E_{T'}[\text{gen}_1 \mid S, T] = 0$) is used twice. Define $S' := S^n \setminus S_2$ and $S'' := S^n \setminus (S_1 \cup S_2)$. We see that
$$
\begin{aligned}
\E_{S_2}[\hat{L}_1^{(k)} - L_1^{(k)} \mid S_1, S_3,\dots, S_N] &= \E_{S_2}\left[\frac{1}{k} \sum_{z' \in S_1} \ell(\mathcal{A}(S^n_{-1}), z') - \E_{z}[\ell(\mathcal{A}(S^n_{-1}), z)] \mid S'\right] \\
&= \E_{S_2}\left[\frac{1}{k} \sum_{z' \in S_1} [\ell(\mathcal{A}(S^n_{-1}), z')] \mid S'\right] - \E_{S_2, z}[\ell(\mathcal{A}(S^n_{-1}), z) \mid S'']
\end{aligned}
$$
where the two terms in the last line are functions of $S'$ and $S''$ respectively, and their difference is non-zero in general.

\end{document}